\documentclass[11pt,reqno]{amsart}
\usepackage{enumerate}
\usepackage{mathptmx}       
\usepackage{helvet}         
\usepackage{courier}        
\usepackage[normalem]{ulem} 
\usepackage[T1]{fontenc}  
\usepackage{bm} 
\usepackage{mathrsfs} 
\usepackage{physics}  
\numberwithin{equation}{section}
\usepackage[colorlinks=true,linkcolor=blue,citecolor=magenta]{hyperref}

\usepackage{caption}
\usepackage{subcaption}
\usepackage{float}
\usepackage{multirow}

\usepackage{pgfplots}

\pgfplotsset{compat = newest}

\usepackage{tikz}
\usetikzlibrary{backgrounds}
\usetikzlibrary{patterns,fadings}
\usetikzlibrary{arrows,decorations.pathmorphing}
\usetikzlibrary{calc}

\usepackage{mathtools} 
\DeclarePairedDelimiter\ceil{\lceil}{\rceil} 
\DeclarePairedDelimiter\floor{\lfloor}{\rfloor} 
\newcommand{\eps}{\epsilon}
\newcommand{\inner}[1]{\langle #1 \rangle}  

\usepackage{MnSymbol} 
\newcommand{\nnorm}[1]{{\left\vert\kern-0.25ex\left\vert\kern-0.25ex\left\vert #1
		\right\vert\kern-0.25ex\right\vert\kern-0.25ex\right\vert}} 
\renewcommand{\leq}{\leqslant}
\renewcommand{\geq}{\geqslant}
\newcommand{\rmd}{\mathrm{d}}   
\usepackage{stmaryrd}
\usepackage{xparse} \DeclarePairedDelimiterX{\Iintv}[1]{\llbracket}{\rrbracket}{\iintvargs{#1}}
\NewDocumentCommand{\iintvargs}{>{\SplitArgument{1}{,}}m}
{\iintvargsaux#1} %
\NewDocumentCommand{\iintvargsaux}{mm} {#1\mkern1.5mu..\mkern1.5mu#2}

\usepackage{fullpage} 
\allowdisplaybreaks 

\theoremstyle{plain}
\newtheorem{Th}{Theorem}[section]
\newtheorem{Lemma}[Th]{Lemma}
\newtheorem{Cor}[Th]{Corollary}
\newtheorem{Prop}[Th]{Proposition}

\theoremstyle{definition}
\newtheorem{Def}[Th]{Definition}

\newtheorem{Rem}[Th]{Remark}
\newtheorem{?}[Th]{Problem}
\newtheorem{Ex}[Th]{Example}
\newtheorem{Not}[Th]{Notation}

\newcommand{\R}{\mathbb{R}}
\newcommand{\Z}{\mathbb{Z}}
\newcommand{\T}{\mathbb{T}}
\newcommand{\br}{\mathbf{c}}

\title[]{From Exclusion to Slow and Fast Diffusion}

\author{P.  Gon\c calves}
\address{Patr\'icia Gon\c calves, Center for Mathematical Analysis,  Geometry and Dynamical Systems,
	Instituto Superior T\'ecnico, Universidade de Lisboa,
	Av. Rovisco Pais, 1049-001 Lisboa, Portugal.}
\email{{\tt pgoncalves@tecnico.ulisboa.pt}}

\author{G. S. Nahum}
\address{Gabriel Nahum, Center for Mathematical Analysis,  Geometry and Dynamical Systems, Instituto Superior T\'ecnico, Universidade de Lisboa,
	Av. Rovisco Pais, 1049-001 Lisboa, Portugal.}
\email{{\tt gabriel.nahum@tecnico.ulisboa.pt}}

\author{M. Simon}
\address{Marielle Simon, Univ Lyon, CNRS, Université Claude Bernard Lyon 1, UMR 5208, Institut Camille Jordan, F-69622 Villeurbanne, France}
\email{{\tt msimon@math.univ-lyon1.fr}}

\date{\today}

\begin{document}
	\begin{abstract}
		We construct a nearest-neighbour interacting particle system of exclusion type, which illustrates a transition from slow to fast diffusion. More precisely, the hydrodynamic limit of this microscopic system  in the diffusive space-time scaling is the parabolic equation $\partial_t\rho=\nabla (D(\rho)\nabla \rho)$, with diffusion coefficient $ D(\rho)=m\rho^{m-1} $ where $ m\in(0,2] $, including therefore the fast diffusion regime in the range $ m\in(0,1) $, and the porous {medium} equation for $ m\in(1,2) $. The construction of the model  is based on the generalized binomial theorem, and  interpolates continuously in $ m $ the already known microscopic \textit{porous {medium} model} with parameter $ m=2 $, the \textit{symmetric simple exclusion process} {with} $ m=1 $, going down to a \textit{fast diffusion model} up to  {any} $ m>0$. The derivation of the hydrodynamic limit for the local density of particles on the one-dimensional torus is achieved \textit{via} the entropy method -- with additional technical difficulties depending on the regime (slow or fast diffusion) {and where new properties of the \textit{porous medium model} {need} to be derived.}
	\end{abstract}
	
	\maketitle
	
	\section{Introduction}
	
	\subsection{Scientific context}	A typical question in the field of statistical mechanics is related to the  derivation of the macroscopic evolution equations  from the stochastic dynamical interaction of microscopic particles.  Over the last four decades, there has been a remarkable progress in the derivation of these equations, which are  partial differential equations (PDEs), governing the space-time evolution of the conserved quantities of the microscopic system, \textit{i.e.~}the well-known \emph{hydrodynamic limit}, see for instance \cite[Chapter 3]{SPOHN:book} {for an introduction on the subject}. {In particular, \emph{stochastic lattice gases}, a specific type of models where particles interact on a lattice and evolve according to a Markovian dynamics, have been the subject of intense scrutiny \cite{KL:book} and
		many results have been obtained by both physicists and mathematicians on their microscopic and macroscopic
		behavior.}  
	The nature {of the hydrodynamic equations} depends on the dynamics at the particle level and it can be for instance:  parabolic, hyperbolic, or even of  a fractional form.
	
	An equation {which} has received a lot of attention in the last years in the PDE's community is the following equation, posed for every $(t,u) \in \R_+\times \T$ where $\T$ is the one-dimensional torus $[0,1)$ with $0\equiv 1$, and given for  $ m\in\mathbb{R}$, $m\neq 0$, by
	\begin{align}\label{PDE:formal}
		\partial_t\rho=\partial_{{uu}}(\rho^m), \qquad (t,{u})\in \R_+\times\mathbb T.
	\end{align}
	This is a parabolic equation, with diffusion coefficient  given by
	\begin{align}\label{diffusion}
		D(\rho)=m\rho^{m-1}.
	\end{align}
	For $ m>1 $, \eqref{PDE:formal} is the \emph{porous {medium} equation}, referred to as PME; for $ m=1 $ it is the standard heat equation (HE), while for $ m<1 $ it belongs to the class of fast diffusion equations, and in this case we will refer to it as  FDE. The rigorous analysis of \eqref{PDE:formal} has attracted a lot of interest in the past decades, we refer the reader to \cite{vazquez} for a review on this subject.

	From the particle systems' point of view, the rigorous derivation of \eqref{PDE:formal} has been successfully achieved for particular values of $m$, in several different ways.
	The HE has been obtained  as the hydrodynamic limit of  the \emph{symmetric simple exclusion process}  (SSEP) (see, for example,  \cite[Chapter 4]{KL:book}). In this process, particles evolve on the discrete torus $ \mathbb{T}_N=\R/N\Z $ and after an exponential clock of rate one, a particle  jumps to one of its {two} nearest-neighbours chosen with equal probability, but the jump only occurs if the destination site is empty ({this is the \emph{exclusion rule})}, otherwise it is suppressed and all the clocks, which are independent of each other, restart.  The configuration of particles in the system at time $t>0$ is denoted by $\eta_t=(\eta_t(x))_{x\in\T_N}$ and it is an element of $\Omega_N:=\{0,1\}^{\T_N}$, where  $\eta(x)\in\{0,1\}$ denotes the number of particles at position $x$. {Moreover, the process $\{\eta_t\}_{t\geqslant 0}$ is a Markov process on $\Omega_N$, and the \emph{jump rate} from a site $x$ to site $x+1$ is given by $\eta(x)(1-\eta(x+1))$ while the jump rate from site $x+1$ to site $x$ is given by $(1-\eta(x))\eta(x+1)$.}
	
	In \cite{GLT}, the authors derived the PME for any integer value of $m\geq 2$ by considering an exclusion process  with \emph{degenerate} rates. More precisely, as above, particles evolve on the discrete torus $ \mathbb{T}_N$ according to the exclusion rule, but the jump rate depends on the number of particles in the vicinity of the edge where the jump occurs. To be concrete, if, for example, $m=2$, then the jump rate from a site $x$ to the site $x+1$ is given by $\eta(x)(1-\eta(x+1))(\eta(x-1)+\eta(x+2))$ and the rate from $x+1$ to $x$ is given by $\eta(x+1)(1-\eta(x))(\eta(x-1)+\eta(x+2))$. This means that {for a jump from $x$ to $x+1$ to happen, one imposes to have \textit{at least one} particle in the vicinity $\{x-1,x+2\}$} ({see Figure \ref{fig:PMM2} for an example of transition rates for this model}). {Besides, one can easily compute} the {microscopic} instantaneous current of the system, \emph{i.e.}~the difference between the jump rate from $x$  to $x+1$ and the jump rate from $x+1$ to $x$, {which}  is then equal to $(\eta(x)-\eta(x+1))(\eta(x-1)+\eta(x+2))$. {Remarkably, this microscopic current} can be rewritten as a {discrete} gradient of some function $h({\eta}),$ {see Lemma \ref{lem:grad} below}. In {fact}, the choice for those specific rates {is made in order} to have the aforementioned gradient property of the instantenous current, which turns the system into a \emph{gradient} one, and classical methods can be explored without too many complications, {see \cite[Chapters 5 and 6]{KL:book}}. Since particles only swap positions on the torus, the number of particles is conserved by the dynamics. The PME {\eqref{PDE:formal} with $m=2$} has {then} been obtained as the hydrodynamic limit {of} the {empirical} density of particles. This rationale was extended to  any integer $ m\geq2 $,  and the resulting microscopic system is now called the \textit{porous medium model}, denoted by PMM($ m-1 $) {with} hydrodynamic equation {\eqref{PDE:formal}}. {Later} in \cite{BDGN}, {the same} PME for any integer $ m\geq2 $ {has been obtained} on the interval $[0,1]$, with different types of boundary conditions (Dirichlet, Robin and Neumann), again as the hydrodynamic limit of the same constrained exclusion process, but in contact with stochastic reservoirs, {which inject and destroy particles at the two extremities with some rate} which is regulated by a parameter, {giving} rise to the aforementioned boundary conditions.
	
	{Another approach had previously been developed in} \cite{SU93,ES96,  FIS97}. {First},  the porous medium equation when $ m=2$ was derived  in \cite{SU93,ES96} from a model in which the occupation number is a continuous variable {(therefore belonging to another class of models)}. More precisely, the model consists of configurations of sticks or energies; the configurations evolve randomly through exchanges of stick portions between nearest-neighbours through a \emph{zero-range} pressure mechanism, and the conservation {law} is the total stick-length. Later in \cite{FIS97} the authors  extended  the derivation of the hydrodynamic limit from the previous model, {and obtained the PME for all range $m>1$}.
	
	{Finally,} {concerning} the fast diffusion case, {few results are available in the literature.} In \cite{HJV20} the FDE with $ m=-1 $ {has been} derived {as the hydrodynamic limit of a} \textit{zero-range process} ({the number of particles per site can be any non-negative integer}) evolving on {the discrete torus}, with {a jump} rate function adjusted to observe frequently a large number of particles, with a specific "weight" associated to each particle. The formalization of the hydrodynamic limit was achieved by using Yau's relative entropy method \cite{yau} with some  adaptations  including spectral gap estimates. The derivation of the FDE for general $ m<1 $ was left  there as an open problem.
	
	\subsection{Construction of  new models} {In this paper we address two questions}: first, how can we generalize the family of  PMMs, {namely exclusion processes},  to $ m $ not being an integer? Second, due to the different nature of the interacting particle systems constructed to derive \eqref{PDE:formal} under the slow-diffusion regime and the fast-diffusion regime, is there a single family parametrized by $ m $ that interpolates between the slow and the  fast diffusion?

	{Here we} give some answers in the direction of the first question, and a positive answer regarding the second. We  {construct} a family of {exclusion processes} parametrized by $ m\in{[0,2]} $ and evolving on the one-dimensional (discrete) torus $ \mathbb{T}_N $ {and we prove that their} hydrodynamic limit is given by  \eqref{PDE:formal}. The motivation for the definition of our models comes from the analysis of the diffusion coefficient $D(\rho)=m\rho^{m-1} $ and the generalized binomial theorem (Proposition \ref{th:gen_bin} below). As a consequence, the resulting family of models interpolates continuously in $ m $ between the SSEP and the PMM(1), in a sense that we shall explain more precisely later on ({see \eqref{intro:interpol}} below). The point is that the generalized binomial theorem allows representing  the diffusion coefficient $D(\rho)=m(1-(1-\rho))^{m-1}$ in terms of a series, \emph{i.e.}
	\begin{align}\label{into:series}
		D(\rho)=\sum_{k\geq 1} \binom{m}{k} (-1)^{k-1} k (1-\rho)^{k-1},
	\end{align}
	{which} can be properly truncated into a polynomial. Above  $ \binom{m}{k} $ is the generalized binomial coefficient, see \eqref{eq:binom} {for the definition}. In the construction {of} the new interpolating model based {on} \eqref{into:series}, the family $\{\text{PMM}(k)\}_{k\geqslant 0}$ can be seen as a ``polynomial basis''.  {Remember that} the porous medium models PMM($k$) considered in \cite{GLT} are of gradient type, and moreover {it can be easily seen that} the Bernoulli product measures with constant parameter are invariant for {each PMM($k$)}.  {Remarkably}, the {interpolating} model {keeps both properties}, and {moreover it becomes} irreducible, {in the sense that every particle configuration can be changed into any other configuration with the same number of particles through successive jumps that happen with positive probability. {We note that} this \emph{irreducibility property} was not verified for the original PMM($k$), and in fact one of the technical difficulties of \cite{GLT} was to work with the so-called \emph{mobile clusters}, \textit{i.e.}~couple of particles at distance at most two, that allow the transport of blocked particles in the system, {but they are not needed here}.}

	{Let us now} be more precise. {As before,  $\{\eta_t\}_{t\geqslant 0}$ is a Markov process on $\Omega_N$, and it can be entirely defined through its \emph{infinitesimal generator}, denoted below by $\mathcal{L}_N^{m-1}$, which is an operator acting on functions defined on $\Omega_N$. In order to give a precise definition, we first need to introduce the infinitesimal generators related to the basis mentioned above: let $ \mathcal{L}_N^{\overline{\text{PMM}}({k})} $ } be the generator of a process  {defined like} {the}  PMM($ k $), but with the constraints acting on \textit{empty} sites, {instead of} particles (in other words, {for $ k=1 $}, the jump from $x$ to $x+1$ happens if there is at least one empty site in $\{x-1,x+2\}$, {see Figure \ref{fig:PMMholes}}).   {We are now ready to introduce the infinitesimal generator of the interpolating model, which is a linear combination of the latter}, and is defined for any $m\in(0,2]$ by 
	\begin{align}\label{intro:gen}
		\mathcal{L}_N^{(m-1)} =  \sum_{k=1}^{\ell_N} \binom{m}{k}(-1)^{k-1} \mathcal{L}_N^{\overline{\text{PMM}}(k-1)}
		,
		\quad
		\text{where} \quad		2\leq\ell_N \xrightarrow[N\to+\infty]{}+\infty.
	\end{align}
	The treatment of a linear combination of models with $ \ell_N\to+\infty $ as $N\to+\infty$ is one of the novelties of this work. It is also worth pointing out that although \eqref{PDE:formal} only has local interactions, we do \textit{not} require that $ \ell_N=o(N) $, and it can be of any order as long as $ {N}\geqslant \ell_N\to+\infty $. {In fact several} difficulties in this paper arise from maintaining $ \ell_N $ with no order restrictions. To achieve this, some new ideas and properties of the family $ \{\text{PMM}(k)\}_{k\geq 0} $ are explored. The interpolating property {invoked above} is a consequence of the definition of the generalized binomial coefficients. Concretely, denoting by $ r_N^{(m-1)}(\eta) $ the jump rate {appearing in $\mathcal{L}_N^{(m-1)}$ at the edge} $ \{0,1\} $ ({for a jump happening from $0$ to $1$ or $1$ to $0$}), for some fixed configuration $ \eta $ and fixed $ N $ it holds that
	\begin{align}\label{intro:interpol}
		\lim_{m \nearrow 1}r_N^{(m-1)}(\eta)
		=\mathbf{r}_{0,1}^{\text{SSEP}}(\eta)
		=\lim_{m\searrow 1}r_N^{(m-1)}(\eta)
		\qquad \text{and}\qquad
		\lim_{m\nearrow 2}r_N^{(m-1)}(\eta)
		=\mathbf{r}_{0,1}^{\text{PMM}(1)}(\eta),
	\end{align}
	where $ \mathbf{r}_{0,1}^{\text{SSEP}}(\eta) $ and $ \mathbf{r}_{0,1}^{\text{PMM}(1)}(\eta) $ are the {jump} rates at the edge $ \{0,1\} $, for the SSEP and PMM($ 1 $), respectively. To better visualize how these rates can deform the SSEP into a slow or fast diffusion model we refer the reader to Figure \ref{fig:1} and to the discussion just before it.
	
	We remark that the sign of the generalized binomial coefficients $ \binom{m}{k} $ changes {according to} the values of $ m $ and $ k $.  This oscillating nature is {the reason why one may} find rates for which \eqref{intro:gen} is not well-defined for $ m>2 $ and {why} an extension of our models to $m>2$ is still out of reach. For $ m\in(0,2) $, the sign of these coefficients lead to an interpretation of the resulting models as the SSEP with either a \textit{penalization} or \textit{reinforcement} given by porous medium models (with constraints on the empty sites), as explained in \eqref{PMM_rewrite}, {and this also explains why the interpolating model becomes irreducible}. 
	This is presented in more details in Proposition \ref{prop:low_bound_r}.
	
	%
	
	\subsection{Main result and strategy}
	Proving a \textit{hydrodynamic limit} is, in plain terms, a law of large numbers for the conserved quantity of the system, in our case the  density of particles. Concretely, the empirical measure associated to the {particle} density {at time $t>0$} is defined {for any
		$\eta \in \Omega_N$}, as follows
	\begin{align*}
		\pi_t^N(\eta,\mathrm{d}u)=\frac1N \sum_{x\in\mathbb{T}_N}\eta_t(x)\delta_{x/N}(\mathrm{d}u).
	\end{align*}
	In other words $\pi_t^N(\eta,\mathrm{d}u)$ is a {random} measure on the continuous {torus} $\T$ and performs the link between the microscopic and macroscopic space scales, {\emph{via}} $ x\mapsto N^{-1}x $. The main result of this paper  states that starting from a \emph{local equilibrium} distribution ({see Definition \ref{def:ass}}), this {random empirical measure, {taken under the diffusive time-rescaling $ t\mapsto N^2t $}, converges in probability as $N\to+\infty$,  to a deterministic measure $\rho_t(u)du$, where $\rho_t(u)$ is  {the unique} weak solution of the \textit{hydrodynamic equation} \eqref{PDE:formal} for $ m\in(0,2) $.}

	Our proof follows the entropy method introduced by \cite{GPV}, which highly relies on the fact that the microscopic model of particles is gradient and has the irreducibility property. 
	The overall strategy can be split into three steps: (i) {we prove} tightness of the sequence of measures induced by the density empirical measure; (ii) {we obtain an} \emph{energy estimate} {which gives information on the regularity of the density profile, and this information is crucial for the proof of uniqueness of weak solutions}; (iii) {we characterize} uniquely the  limiting points.		Different {technical} problems arise for both  slow ($m>1$) and  fast ($m<1$) regimes.  Since we deal with systems whose jump rates are of polynomial form, we need to show that these polynomials are such that the equations for the empirical measures can be recovered. This is known in the literature as the \emph{replacement lemmas} which are one of the most  difficult challenges in the derivation of hydrodynamic limits from microscopic systems.  In particular,  the replacement lemmas are specific to each regime (see Lemmas \ref{lem:rep_shift}, \ref{lem:rep_boxes} for the slow regime and Lemmas \ref{lem:rep_FDM-tight}, \ref{lem:rep_FDM} for the fast regime). Fundamental to the proof of those lemmas is the energy lower bound (Proposition \ref{prop:energy}) which compares the Dirichlet form of our process with the ``Carr\'e-du-Champ" operator, and the results of Subsection \ref{sec:main_model}, {where} we derive some new properties of the family $\{\text{PMM}(k)\}_{k\geq 0} $, in particular we prove {several bounds on their rates} {{which also show} that our models are well-defined.}
	In the fast regime, {it is surprising that} the tightness step {requires}  the replacement Lemma \ref{lem:rep_FDM-tight}, due to the supremum of the rates being unbounded as $ N\to+\infty $.  {Finally} the characterization of the limit points is the most technical part, and also uses  several replacement lemmas. We note that the scheme which is implemented for the slow regime is a simplification of the scheme of \cite{BDGN}.
	
	The application of those replacement lemmas involves some novelties due to the summation with binomial coefficients  in the definition of $ \mathcal{L}^{(m-1)} $. {Roughly speaking}, the replacement lemmas link the microscopic and macroscopic scales by approximating the product of $ k $ occupation variables by $ k $ empirical averages over independent boxes -- first by \textit{microscopic} boxes (``one-block estimate''), then by approximating the microscopic boxes by \textit{mesoscopic} boxes (``two-blocks estimate''). {Here}, very importantly, the size of these boxes needs to be adjusted dynamically with $ k $ for the series of errors to vanish in the limit $ N\to+\infty $. However, {this dynamical argument alone would require to impose} stronger assumptions on the explosion of $ \ell_N $. To avoid this, it is fundamental to first slow down the explosion by replacing $ \ell_N $ by $ (\ell_N)^n $ with $ 0<n<1 $. This argument depends on the order of the tail of the series $ \sum_{k\geq1}\abs{\binom{m}{k}} $. Naturally, the treatment of this series also requires a sharp non-asymptotic estimate on the binomial coefficients, {see} Lemma \ref{lem:bin_bound}.
	
	{Finally,} there were  some technical issues regarding the \textit{energy estimate},  {precisely when} showing that the (weak) solution of \eqref{PDE:formal} (Definition \ref{def:weak}) belongs to the target Sobolev space. This is {crucial} because it allows us to argue that the solution to the PDE is H\"{o}lder continuous, which in turn is essential to show that it is well approximated locally by the empirical measure. {The weak differentiability of specific functions of $ \rho $ is also needed to prove uniqueness, giving us that the whole sequence of measures converges thanks to tightness. Specifically, if $ \rho^m $ belongs to the target Sobolev space (which is the case for $ m\in(1,2) $), uniqueness follows by simple energy arguments (see Lemma \ref{lem:uniq_PME}), while if $ \rho $ only belongs to the target Sobolev space (when $ m\in(0,1)$), then the proof is more involved (see Lemma \ref{lem:uniq_FDE}), and it is an adaptation of the argument for \textit{very weak} solutions in \cite{vazquez}.
	}

	\subsection{Extensions and future work}\label{sec:ext}
	
	Now we comment a bit on possible extensions of our results. First we note that for $ m>2 $ there are configurations where the rates $ r_N^{(m-1)}(\eta) $ {are negative} and {therefore} the model is not well-defined. An example is $ m\in(2,3) $ with $ \eta(0)+\eta(1)=1 $ and $ \eta(-1)=0,\;\eta(x)=1\ $ for $ x\neq-1,0,1 $. The extension to $ m>2 $ requires a different approach and will be the subject of study on a forthcoming work.

	We also highlight that the  derivation of fractional equations from microscopic systems has attracted a lot of attention recently. In another forthcoming work we will use the mechanism based on the generalized binomial theorem to construct a well-defined Markov generator interpolating the long-range SSEP (introduced in \cite{LR:JARA08}) and the long-range PMM($ 1 $) (introduced in \cite{LR:CDG2022}), whose hydrodynamic limit follows $ \partial_t\rho=-(-\Delta)^{\frac{\gamma}{2}}\rho^m $ with $ m\in(0,2] $ and $ \gamma\in(1,2) $. This is work in progress. 
	
	As a final note, our main goal was to introduce a toy model in the simplest context. From the stochastic process point of view, it would be interesting to extend our results to higher (finite) dimensions. Moreover,  the lower bound in Proposition \ref{prop:energy} could be used to extend our results to the open boundary setting, following similar arguments as  in \cite{BDGN} and {using} our approach {for} the treatment of the sum up to $ \ell_N $, with some adaptations. Fixing the rate of creation/annihilation of particles to be proportional to $ N^{-\theta} $ for $ \theta\geq0 $, one could obtain {different} boundary regimes: Dirichlet ($\theta\in[0,1)$), non-linear Robin ($ \theta=1 $) and Neumann ($ \theta>1 $); with the specific expressions as in \cite{BDGN} but with $ m=2 $ there extended to $ m\in(0,2)\backslash\{1\} $.
	All this is left for future work.

	\subsection{Outline of the paper} 	The present work is organized as follows: Section \ref{sec:models} is devoted to introducing the family of porous medium models which will be the building blocks  to construct our new models {and used to prove some of the important properties of the latter}; particularly, in Subsection \ref{sec:main_model} we construct the interpolating models, prove that they are well-defined, and in Subsection \ref{subsec:interp_prop} we study some of their monotonicity properties {and present our main result.}
	Then we  prove the {convergence towards the} hydrodynamic limit in Section \ref{sec:HL}.  Section \ref{sec:replace} is devoted to the statement and proof of the so-called \textit{replacement Lemmas}, which are in the heart of the proof of the hydrodynamic limit. Finally, in Section \ref{sec:energy} we obtain the \textit{energy estimates}. In Appendix \ref{app:aux_res} we prove an auxiliary result regarding the generalized binomial coefficients and in Appendix \ref{app:PDE} uniqueness and regularity results regarding the weak solution of the hydrodynamic equations are derived.

	\section{Microscopic models and Main Result}\label{sec:models}

	Let $ \mathbb{N}_+ $ be the set of positive natural numbers and denote by $ N \in\mathbb{N}_+ $ a scaling parameter. Denote by $ \mathbb{T}_N $ the one dimensional discrete torus, that is, $ \mathbb{T}_N=\{1,\dots, N\} $ with the identification $ 0\equiv N $.  For any $ x < y\in\mathbb{Z} $, that can be viewed as elements in $\T_N$ by considering their standard projections, we define $ \llbracket x,y\rrbracket $ as the discrete interval composed by all the discrete points between $ x,y $ (including $ x,y $) in $ \mathbb{T}_N $, where the order has been inherited from the one in $\Z$.
	
	The microscopic dynamics at the core of this paper is a system of particles which {evolves according to a Markov process}, satisfying the exclusion rule and situated on the discrete torus $\T_N$. A particle configuration $\eta$ is an element of $\Omega_N=\{0,1\}^{\T_N}$, namely $ \eta(x)\in\{0,1\} $ for any $ x\in\mathbb{T}_N $.
	Particles can jump to nearest-neighbour sites only, providing the latter are not already occupied. Before  defining the generator of the dynamics, let us introduce the following operators:
	
	\begin{Def}[Exchange of occupation variables]
		For any $ x,y,z\in\mathbb{T}_N $ let us consider the exchange of occupation variables $ \eta\mapsto \eta^{x,y} $ given by
		\begin{align*}
			\eta^{x,y}(z)
			=\mathbf{1}_{z\neq x,y}\;\eta(z)
			+\mathbf{1}_{z=x}\;\eta(y)
			+\mathbf{1}_{z=y}\;\eta(x).
		\end{align*}
		We  define the operator $ \nabla_{x,y} $ associated to the occupation exchange, given on any $ f:\Omega_N\to\mathbb{R} $ by
		\begin{align*}
			\nabla_{x,y} f(\eta)
			=f(\eta^{x,y})-f(\eta).
		\end{align*}
		Finally, for any $ x\in\mathbb{T}_N $, define the translation $ \tau_{x}\eta(y)=\eta(x+y) $ for $y\in\mathbb T_N$, and extend it to functions $ f:\Omega_N\to\mathbb{R} $ by $
		\tau_xf(\eta)=f(\tau_x\eta).$
	\end{Def}

	The rest of the section is organized as follows: first of all, we recall  the definition of the \emph{porous medium  models} from \cite{GLT}, which correspond to a microscopic description of the PME for any integer $m\geq 2$. Then, we define its \emph{flipped} version, in the sense that the kinetic constraint is imposed on empty sites  instead of particles. Finally, we  define a new microscopic family of models parametrized by $ m\in(0,2)\backslash\{1\}, $ which we call \textit{non-integer fast diffusion model} when $ m\in(0,1) $, and \textit{non-integer porous medium model} when $ m\in(1,2) $.

	\subsection{Porous media model with dynamical constraints on vacant sites}
	Roughly speaking, our models can be seen as the SSEP either reinforced or penalized by a linear combination of \emph{kinetically constrained  exclusion processes} (KCEP), which the family PMM($k$) belongs to. Let us first recall the definition of the known models which will come into play.

	\begin{Def}[Symmetric Simple Exclusion Process]
		We denote by SSEP on $ \mathbb{T}_N $ the Markov process with state space $\Omega_N$ generated by the following operator $ \mathcal{L}_N^{\text{SSEP}} $, which acts on $ f:\Omega_N\to\mathbb{R} $ as:
		\begin{align*}
			(\mathcal{L}_N^{\text{SSEP}}f)(\eta)=\sum_{x\in\mathbb{T}_N}
			\mathbf{a}_{x,x+1}(\eta)
			(\nabla_{x,x+1}f)(\eta)
		\end{align*} for any $\eta\in\Omega_N$,
		where
		\begin{equation}\label{eq:s}
			\mathbf{a}_{0,1}(\eta) = \eta(0)(1-\eta(1))+\eta(1)(1-\eta(0)), \qquad \mathbf{a}_{x,x+1}(\eta) =\mathbf{a}_{x+1,x}(\eta)=\tau_x \mathbf{a}_{0,1}(\eta).
		\end{equation}
		Note that the latter equals $ 1 $ if exactly one site among $\{x,x+1\}$ is occupied by a particle, and $ 0 $ otherwise. Due to the symmetry of the rates we will short-write $
		\mathbf{a}:= \mathbf{a}_{0,1}=\mathbf{a}_{1,0}.$
	\end{Def}
	
	\begin{Def}[Porous Medium  Model for any integer $ k\geq 1$, \cite{GLT}]
		For any $ k\in\mathbb{N}_+ $ let us denote by PMM($ k $) the  \emph{porous medium model} on $ \mathbb{T}_N $ with parameter $ k $, as the Markov process with state space $\Omega_N$ generated by the following operator $ \mathcal{L}_N^{\text{PMM}(k)} $, which acts on $ f:\Omega_N\to\mathbb{R} $ as:
		\begin{align*}
			(\mathcal{L}_N^{\text{PMM}(k)}f)(\eta)=\sum_{x\in\mathbb{T}_N}
			\mathbf{c}^{(k)}_{x,x+1}(\eta)
			\mathbf{a}_{x,x+1}(\eta)
			(\nabla_{x,x+1}f)(\eta)
		\end{align*} for any $\eta\in\Omega_N$,
		where $\; \mathbf{c}^{(k)}_{x,x+1}(\eta) = \tau_x \mathbf{c}_{0,1}^{(k)}(\eta) $ with
		\begin{align}\label{rate:pmm_int}
			\mathbf{c}_{0,1}^{(k)}(\eta) =\sum_{j=1}^{k+1} 
			\mathbf{s}_j^{(k)}(\eta)
			\quad\text{and}\quad
			\mathbf{s}_j^{(k)}(\eta)=\prod_{\substack{i=-(k+1)+j\\i\neq0,1}}^j\eta(i).
		\end{align}
	\end{Def}
	\begin{Not}
		We write
		\begin{equation}\label{eq:ck}
			\mathbf{r}^{(k)}_{x,x+1}(\eta)
			=\mathbf{c}^{(k)}_{x,x+1}(\eta)\;\mathbf{a}_{x,x+1}(\eta) = \mathbf{r}^{(k)}_{x+1,x}(\eta)
		\end{equation}
		for the rate at which the occupation variables $\eta(x)$ and $\eta(x+1)$ are exchanged in  PMM($k$). The quantity $\mathbf{c}_{x,x+1}^{(k)}(\eta)$ is the constraint to be satisfied for the jump to happen.
		Again due to the symmetry of the rate and constraint, we short-write
		\begin{equation}\label{eq:defs}
			\mathbf{c}^{(k)}(\eta)\equiv \mathbf{c}_{0,1}^{(k)}(\eta)
			\quad\text{ and }\quad
			\mathbf{r}^{(k)}(\eta)\equiv \mathbf{r}_{0,1}^{(k)}(\eta).
		\end{equation}
	\end{Not}
	As it can be seen from \eqref{rate:pmm_int} and Figure \ref{fig:constraint}, a jump crossing the bond $\{x,x+1\}$ is allowed only if at least $k$ consecutive particles out of the edge $\{x,x+1\}$ are situated in the box $\llbracket x-k, x+(k+1) \rrbracket $.
	
	\begin{figure}[H]
		\tikzset{every picture/.style={line width=0.75pt}} 
		
		\begin{tikzpicture}[x=0.70pt,y=0.70pt,yscale=-1,xscale=1]
			
			\draw    (240.5,135.25) -- (541.5,135.25) (274.5,131.25) -- (274.5,139.25)(308.5,131.25) -- (308.5,139.25)(342.5,131.25) -- (342.5,139.25)(376.5,131.25) -- (376.5,139.25)(410.5,131.25) -- (410.5,139.25)(444.5,131.25) -- (444.5,139.25)(478.5,131.25) -- (478.5,139.25)(512.5,131.25) -- (512.5,139.25) ;
			\draw  [fill={rgb, 255:red, 155; green, 155; blue, 155 }  ,fill opacity=0.5 ] (461.77,118.85) .. controls (461.77,109.74) and (469.36,102.35) .. (478.73,102.35) .. controls (488.09,102.35) and (495.68,109.74) .. (495.68,118.85) .. controls (495.68,127.96) and (488.09,135.35) .. (478.73,135.35) .. controls (469.36,135.35) and (461.77,127.96) .. (461.77,118.85) -- cycle ;
			\draw  [fill={rgb, 255:red, 155; green, 155; blue, 155 }  ,fill opacity=0.5 ] (427.86,118.85) .. controls (427.86,109.74) and (435.45,102.35) .. (444.82,102.35) .. controls (454.18,102.35) and (461.77,109.74) .. (461.77,118.85) .. controls (461.77,127.96) and (454.18,135.35) .. (444.82,135.35) .. controls (435.45,135.35) and (427.86,127.96) .. (427.86,118.85) -- cycle ;
			\draw  [fill={rgb, 255:red, 155; green, 155; blue, 155 }  ,fill opacity=0.5 ] (427.86,85.85) .. controls (427.86,76.74) and (435.45,69.35) .. (444.82,69.35) .. controls (454.18,69.35) and (461.77,76.74) .. (461.77,85.85) .. controls (461.77,94.96) and (454.18,102.35) .. (444.82,102.35) .. controls (435.45,102.35) and (427.86,94.96) .. (427.86,85.85) -- cycle ;
			\draw  [color={rgb, 255:red, 0; green, 0; blue, 0 }  ,draw opacity=1 ] (360.7,31.75) -- (427.2,31.75) -- (427.2,159.25) -- (360.7,159.25) -- cycle ;
			\draw  [fill={rgb, 255:red, 155; green, 155; blue, 155 }  ,fill opacity=0.5 ] (326.14,52.85) .. controls (326.14,43.74) and (333.73,36.35) .. (343.09,36.35) .. controls (352.45,36.35) and (360.05,43.74) .. (360.05,52.85) .. controls (360.05,61.96) and (352.45,69.35) .. (343.09,69.35) .. controls (333.73,69.35) and (326.14,61.96) .. (326.14,52.85) -- cycle ;
			\draw  [fill={rgb, 255:red, 155; green, 155; blue, 155 }  ,fill opacity=0.5 ] (326.14,85.85) .. controls (326.14,76.74) and (333.73,69.35) .. (343.09,69.35) .. controls (352.45,69.35) and (360.05,76.74) .. (360.05,85.85) .. controls (360.05,94.96) and (352.45,102.35) .. (343.09,102.35) .. controls (333.73,102.35) and (326.14,94.96) .. (326.14,85.85) -- cycle ;
			\draw  [fill={rgb, 255:red, 155; green, 155; blue, 155 }  ,fill opacity=0.5 ] (292.23,52.85) .. controls (292.23,43.74) and (299.82,36.35) .. (309.18,36.35) .. controls (318.55,36.35) and (326.14,43.74) .. (326.14,52.85) .. controls (326.14,61.96) and (318.55,69.35) .. (309.18,69.35) .. controls (299.82,69.35) and (292.23,61.96) .. (292.23,52.85) -- cycle ;
			\draw    (240.5,102.25) -- (541.5,102.25) (274.5,98.25) -- (274.5,106.25)(308.5,98.25) -- (308.5,106.25)(342.5,98.25) -- (342.5,106.25)(376.5,98.25) -- (376.5,106.25)(410.5,98.25) -- (410.5,106.25)(444.5,98.25) -- (444.5,106.25)(478.5,98.25) -- (478.5,106.25)(512.5,98.25) -- (512.5,106.25) ;
			\draw    (240.5,69.25) -- (541.5,69.25) (274.5,65.25) -- (274.5,73.25)(308.5,65.25) -- (308.5,73.25)(342.5,65.25) -- (342.5,73.25)(376.5,65.25) -- (376.5,73.25)(410.5,65.25) -- (410.5,73.25)(444.5,65.25) -- (444.5,73.25)(478.5,65.25) -- (478.5,73.25)(512.5,65.25) -- (512.5,73.25) ;
			
			\draw (267,144.5) node [anchor=north west][inner sep=0.75pt]  [font=\scriptsize] [align=left] {\mbox{-}3};
			\draw (507.67,144.5) node [anchor=north west][inner sep=0.75pt]  [font=\scriptsize] [align=left] {4};
			\draw (474.17,144.5) node [anchor=north west][inner sep=0.75pt]  [font=\scriptsize] [align=left] {3};
			\draw (440,145) node [anchor=north west][inner sep=0.75pt]  [font=\scriptsize] [align=left] {2};
			\draw (405.67,145.5) node [anchor=north west][inner sep=0.75pt]  [font=\scriptsize] [align=left] {1};
			\draw (371.67,144.5) node [anchor=north west][inner sep=0.75pt]  [font=\scriptsize] [align=left] {0};
			\draw (335.17,144.5) node [anchor=north west][inner sep=0.75pt]  [font=\scriptsize] [align=left] {\mbox{-}1};
			\draw (301,144.5) node [anchor=north west][inner sep=0.75pt]  [font=\scriptsize] [align=left] {\mbox{-}2};

		\end{tikzpicture}
		\caption{PMM($ 2 $) valid constraints for which a particle swaps positions in the edge $ \{0,1\} $.}\label{fig:constraint}
	\end{figure}
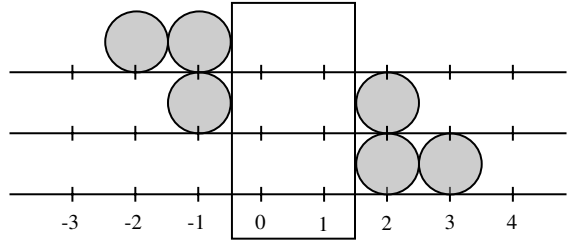
	
	An illustration of the dynamics for $k=1$ is also provided in Figure \ref{fig:PMM2}.
	\begin{Rem}[$k=0$]\label{rem:SSEP=PMM(0)} Note that for $k=0$, $\mathbf{c}^{(0)}(\eta)\equiv 1$ and therefore $ \mathbf{r}^{(0)}(\eta)=\mathbf{a}(\eta)$, which corresponds to the exchange rate in SSEP. It will be useful to interpret $\text{PMM}(0)=\text{SSEP}$.
		
	\end{Rem}

	\begin{Def}[Flipped configuration]
		For any $ \eta\in\Omega_N $, let $ \eta\mapsto\overline{\eta} $ be the map that flips holes with particles, namely: for any $x\in\T_N$, $			\overline{\eta}(x)=1-\eta(x).$
	\end{Def}
	We are now ready to introduce the flipped porous medium model.
	\begin{Def}
		For any $ k\in\mathbb{N}_+ $, let us denote by $\overline{\text{PMM}}(k)$ the \emph{flipped porous medium model} with parameter $k$ with dynamical constraints on the vacant sites, as the Markov process on $\Omega_N$ generated by the following operator $ \mathcal{L}_N^{\overline{\text{PMM}}(k)} $, which acts on functions $ f:\Omega_N\to\mathbb{R} $ as
		\begin{equation}\label{eq:dualPM}
			\big(\mathcal{L}_N^{\overline{\text{PMM}}(k)}f\big)(\eta)
			=\sum_{x\in\mathbb{T}_N}\mathbf{c}^{(k)}_{x,x+1}(\overline{\eta})\mathbf{a}_{x,x+1}(\eta)
			(\nabla_{x,x+1}f)(\eta)
		\end{equation} for any $\eta\in\Omega_N$.
	\end{Def}
	Note that  the process above can be interpreted as the \textit{empty sites} following the same constraint as in PMM($ k $): a jump crossing the bond $\{x,x+1\}$ is allowed only if  at least $k$ ``consecutive" \emph{empty sites} out of the edge $\{x,x+1\}$ are situated in the box $\llbracket x-k, x+(k+1) \rrbracket
	$. An illustration of the dynamics is provided in Figure \ref{fig:PMMholes}. We also highlight that the parameter $ k $ in the PMM($ k $) corresponds to the exponent of the diffusion coefficient, $ D(\rho)=(k+1)\rho^k $, hence to the equation \eqref{PDE:formal} with $ m=k+1 $.
	
	\begin{figure}[H]
		\centering
		\begin{minipage}{0.45\textwidth}
			\centering

			\tikzset{every picture/.style={line width=0.75pt}} 
			
			\begin{tikzpicture}[x=0.65pt,y=0.65pt,yscale=-1,xscale=1]
				
				\draw    (182.97,121.5) -- (517.77,121.5) (216.97,117.5) -- (216.97,125.5)(250.97,117.5) -- (250.97,125.5)(284.97,117.5) -- (284.97,125.5)(318.97,117.5) -- (318.97,125.5)(352.97,117.5) -- (352.97,125.5)(386.97,117.5) -- (386.97,125.5)(420.97,117.5) -- (420.97,125.5)(454.97,117.5) -- (454.97,125.5)(488.97,117.5) -- (488.97,125.5) ;
				\draw  [fill={rgb, 255:red, 155; green, 155; blue, 155 }  ,fill opacity=0.5 ] (335.71,100.35) .. controls (335.71,91.24) and (343.3,83.85) .. (352.67,83.85) .. controls (362.03,83.85) and (369.62,91.24) .. (369.62,100.35) .. controls (369.62,109.46) and (362.03,116.85) .. (352.67,116.85) .. controls (343.3,116.85) and (335.71,109.46) .. (335.71,100.35) -- cycle ;
				\draw  [fill={rgb, 255:red, 155; green, 155; blue, 155 }  ,fill opacity=0.5 ] (437.44,100.35) .. controls (437.44,91.24) and (445.03,83.85) .. (454.39,83.85) .. controls (463.76,83.85) and (471.35,91.24) .. (471.35,100.35) .. controls (471.35,109.46) and (463.76,116.85) .. (454.39,116.85) .. controls (445.03,116.85) and (437.44,109.46) .. (437.44,100.35) -- cycle ;
				\draw  [fill={rgb, 255:red, 155; green, 155; blue, 155 }  ,fill opacity=0.5 ] (267.89,100.35) .. controls (267.89,91.24) and (275.48,83.85) .. (284.85,83.85) .. controls (294.21,83.85) and (301.8,91.24) .. (301.8,100.35) .. controls (301.8,109.46) and (294.21,116.85) .. (284.85,116.85) .. controls (275.48,116.85) and (267.89,109.46) .. (267.89,100.35) -- cycle ;
				\draw  [fill={rgb, 255:red, 155; green, 155; blue, 155 }  ,fill opacity=0.5 ] (233.98,100.35) .. controls (233.98,91.24) and (241.58,83.85) .. (250.94,83.85) .. controls (260.3,83.85) and (267.89,91.24) .. (267.89,100.35) .. controls (267.89,109.46) and (260.3,116.85) .. (250.94,116.85) .. controls (241.58,116.85) and (233.98,109.46) .. (233.98,100.35) -- cycle ;
				\draw    (319.01,79.73) .. controls (322.12,61) and (352.8,62.05) .. (352.67,82.85) ;
				\draw [shift={(352.67,82.85)}, rotate = 270.36] [color={rgb, 255:red, 0; green, 0; blue, 0 }  ][line width=0.75]    (0,3.35) -- (0,-3.35)   ;
				\draw [shift={(318.76,82.85)}, rotate = 270.12] [fill={rgb, 255:red, 0; green, 0; blue, 0 }  ][line width=0.08]  [draw opacity=0] (7.14,-3.43) -- (0,0) -- (7.14,3.43) -- (4.74,0) -- cycle    ;
				\draw    (284.85,82.85) .. controls (284.89,61.94) and (315.9,60.99) .. (318.57,80.03) ;
				\draw [shift={(318.76,82.85)}, rotate = 270.36] [fill={rgb, 255:red, 0; green, 0; blue, 0 }  ][line width=0.08]  [draw opacity=0] (7.14,-3.43) -- (0,0) -- (7.14,3.43) -- (4.74,0) -- cycle    ;
				\draw [shift={(284.85,82.85)}, rotate = 90.12] [color={rgb, 255:red, 0; green, 0; blue, 0 }  ][line width=0.75]    (0,3.35) -- (0,-3.35)   ;
				\draw    (217.28,79.73) .. controls (220.39,61) and (251.07,62.05) .. (250.94,82.85) ;
				\draw [shift={(250.94,82.85)}, rotate = 270.36] [color={rgb, 255:red, 0; green, 0; blue, 0 }  ][line width=0.75]    (0,3.35) -- (0,-3.35)   ;
				\draw [shift={(217.03,82.85)}, rotate = 270.12] [fill={rgb, 255:red, 0; green, 0; blue, 0 }  ][line width=0.08]  [draw opacity=0] (7.14,-3.43) -- (0,0) -- (7.14,3.43) -- (4.74,0) -- cycle    ;
				\draw    (352.67,82.85) .. controls (352.71,61.94) and (383.72,60.99) .. (386.39,80.03) ;
				\draw [shift={(386.58,82.85)}, rotate = 270.36] [fill={rgb, 255:red, 0; green, 0; blue, 0 }  ][line width=0.08]  [draw opacity=0] (7.14,-3.43) -- (0,0) -- (7.14,3.43) -- (4.74,0) -- cycle    ;
				\draw [shift={(352.67,82.85)}, rotate = 90.12] [color={rgb, 255:red, 0; green, 0; blue, 0 }  ][line width=0.75]    (0,3.35) -- (0,-3.35)   ;
				\draw    (420.74,79.73) .. controls (423.84,61) and (454.52,62.05) .. (454.39,82.85) ;
				\draw [shift={(454.39,82.85)}, rotate = 270.36] [color={rgb, 255:red, 0; green, 0; blue, 0 }  ][line width=0.75]    (0,3.35) -- (0,-3.35)   ;
				\draw [shift={(420.48,82.85)}, rotate = 270.12] [fill={rgb, 255:red, 0; green, 0; blue, 0 }  ][line width=0.08]  [draw opacity=0] (7.14,-3.43) -- (0,0) -- (7.14,3.43) -- (4.74,0) -- cycle    ;
				\draw    (454.39,82.85) .. controls (454.44,61.94) and (485.45,60.99) .. (488.12,80.03) ;
				\draw [shift={(488.3,82.85)}, rotate = 270.36] [fill={rgb, 255:red, 0; green, 0; blue, 0 }  ][line width=0.08]  [draw opacity=0] (7.14,-3.43) -- (0,0) -- (7.14,3.43) -- (4.74,0) -- cycle    ;
				\draw [shift={(454.39,82.85)}, rotate = 90.12] [color={rgb, 255:red, 0; green, 0; blue, 0 }  ][line width=0.75]    (0,3.35) -- (0,-3.35)   ;
				
				\draw (246.67,129) node [anchor=north west][inner sep=0.75pt]  [font=\scriptsize] [align=left] {1};
				\draw (485.33,129) node [anchor=north west][inner sep=0.75pt]  [font=\scriptsize] [align=left] {8};
				\draw (450.83,129) node [anchor=north west][inner sep=0.75pt]  [font=\scriptsize] [align=left] {7};
				\draw (416.67,129) node [anchor=north west][inner sep=0.75pt]  [font=\scriptsize] [align=left] {6};
				\draw (382.83,129) node [anchor=north west][inner sep=0.75pt]  [font=\scriptsize] [align=left] {5};
				\draw (348.33,129) node [anchor=north west][inner sep=0.75pt]  [font=\scriptsize] [align=left] {4};
				\draw (314.83,129) node [anchor=north west][inner sep=0.75pt]  [font=\scriptsize] [align=left] {3};
				\draw (280.67,129) node [anchor=north west][inner sep=0.75pt]  [font=\scriptsize] [align=left] {2};
				\draw (333.23,52.9) node [anchor=north west][inner sep=0.75pt]  [font=\scriptsize]  {$1$};
				\draw (300.23,52.9) node [anchor=north west][inner sep=0.75pt]  [font=\scriptsize,color={rgb, 255:red, 0; green, 0; blue, 0 }  ,opacity=1 ]  {$2$};
				\draw (230.83,52.9) node [anchor=north west][inner sep=0.75pt]  [font=\scriptsize]  {$1$};
				\draw (366.23,52.9) node [anchor=north west][inner sep=0.75pt]  [font=\scriptsize,color={rgb, 255:red, 208; green, 2; blue, 27 }  ,opacity=1 ]  {$0$};
				\draw (434.56,52.9) node [anchor=north west][inner sep=0.75pt]  [font=\scriptsize,color={rgb, 255:red, 208; green, 2; blue, 27 }  ,opacity=1 ]  {$0$};
				\draw (468.56,52.9) node [anchor=north west][inner sep=0.75pt]  [font=\scriptsize,color={rgb, 255:red, 208; green, 2; blue, 27 }  ,opacity=1 ]  {$0$};
				\draw (212.67,129) node [anchor=north west][inner sep=0.75pt]  [font=\scriptsize] [align=left] {0};

			\end{tikzpicture}
			\caption{PMM($ 1 $) transition rates.} \label{fig:PMM2}
		\end{minipage}\qquad
		\begin{minipage}{0.45\textwidth}
			\centering

			\tikzset{every picture/.style={line width=0.75pt}} 
			
			\begin{tikzpicture}[x=0.65pt,y=0.65pt,yscale=-1,xscale=1]
				
				\draw    (155.3,112.83) -- (490.1,112.83) (189.3,108.83) -- (189.3,116.83)(223.3,108.83) -- (223.3,116.83)(257.3,108.83) -- (257.3,116.83)(291.3,108.83) -- (291.3,116.83)(325.3,108.83) -- (325.3,116.83)(359.3,108.83) -- (359.3,116.83)(393.3,108.83) -- (393.3,116.83)(427.3,108.83) -- (427.3,116.83)(461.3,108.83) -- (461.3,116.83) ;
				\draw  [fill={rgb, 255:red, 155; green, 155; blue, 155 }  ,fill opacity=0.5 ] (308.05,91.68) .. controls (308.05,82.57) and (315.64,75.18) .. (325,75.18) .. controls (334.36,75.18) and (341.95,82.57) .. (341.95,91.68) .. controls (341.95,100.8) and (334.36,108.18) .. (325,108.18) .. controls (315.64,108.18) and (308.05,100.8) .. (308.05,91.68) -- cycle ;
				\draw  [fill={rgb, 255:red, 155; green, 155; blue, 155 }  ,fill opacity=0.5 ] (409.77,91.68) .. controls (409.77,82.57) and (417.36,75.18) .. (426.73,75.18) .. controls (436.09,75.18) and (443.68,82.57) .. (443.68,91.68) .. controls (443.68,100.8) and (436.09,108.18) .. (426.73,108.18) .. controls (417.36,108.18) and (409.77,100.8) .. (409.77,91.68) -- cycle ;
				\draw  [fill={rgb, 255:red, 155; green, 155; blue, 155 }  ,fill opacity=0.5 ] (240.23,91.68) .. controls (240.23,82.57) and (247.82,75.18) .. (257.18,75.18) .. controls (266.55,75.18) and (274.14,82.57) .. (274.14,91.68) .. controls (274.14,100.8) and (266.55,108.18) .. (257.18,108.18) .. controls (247.82,108.18) and (240.23,100.8) .. (240.23,91.68) -- cycle ;
				\draw  [fill={rgb, 255:red, 155; green, 155; blue, 155 }  ,fill opacity=0.5 ] (206.32,91.68) .. controls (206.32,82.57) and (213.91,75.18) .. (223.27,75.18) .. controls (232.64,75.18) and (240.23,82.57) .. (240.23,91.68) .. controls (240.23,100.8) and (232.64,108.18) .. (223.27,108.18) .. controls (213.91,108.18) and (206.32,100.8) .. (206.32,91.68) -- cycle ;
				\draw    (291.34,71.06) .. controls (294.45,52.34) and (325.13,53.38) .. (325,74.18) ;
				\draw [shift={(325,74.18)}, rotate = 270.36] [color={rgb, 255:red, 0; green, 0; blue, 0 }  ][line width=0.75]    (0,3.35) -- (0,-3.35)   ;
				\draw [shift={(291.09,74.18)}, rotate = 270.12] [fill={rgb, 255:red, 0; green, 0; blue, 0 }  ][line width=0.08]  [draw opacity=0] (7.14,-3.43) -- (0,0) -- (7.14,3.43) -- (4.74,0) -- cycle    ;
				\draw    (257.18,74.18) .. controls (257.23,53.27) and (288.24,52.33) .. (290.91,71.36) ;
				\draw [shift={(291.09,74.18)}, rotate = 270.36] [fill={rgb, 255:red, 0; green, 0; blue, 0 }  ][line width=0.08]  [draw opacity=0] (7.14,-3.43) -- (0,0) -- (7.14,3.43) -- (4.74,0) -- cycle    ;
				\draw [shift={(257.18,74.18)}, rotate = 90.12] [color={rgb, 255:red, 0; green, 0; blue, 0 }  ][line width=0.75]    (0,3.35) -- (0,-3.35)   ;
				\draw    (189.62,71.06) .. controls (192.72,52.34) and (223.4,53.38) .. (223.27,74.18) ;
				\draw [shift={(223.27,74.18)}, rotate = 270.36] [color={rgb, 255:red, 0; green, 0; blue, 0 }  ][line width=0.75]    (0,3.35) -- (0,-3.35)   ;
				\draw [shift={(189.36,74.18)}, rotate = 270.12] [fill={rgb, 255:red, 0; green, 0; blue, 0 }  ][line width=0.08]  [draw opacity=0] (7.14,-3.43) -- (0,0) -- (7.14,3.43) -- (4.74,0) -- cycle    ;
				\draw    (325,74.18) .. controls (325.04,53.27) and (356.05,52.33) .. (358.73,71.36) ;
				\draw [shift={(358.91,74.18)}, rotate = 270.36] [fill={rgb, 255:red, 0; green, 0; blue, 0 }  ][line width=0.08]  [draw opacity=0] (7.14,-3.43) -- (0,0) -- (7.14,3.43) -- (4.74,0) -- cycle    ;
				\draw [shift={(325,74.18)}, rotate = 90.12] [color={rgb, 255:red, 0; green, 0; blue, 0 }  ][line width=0.75]    (0,3.35) -- (0,-3.35)   ;
				\draw    (393.07,71.06) .. controls (396.18,52.34) and (426.86,53.38) .. (426.73,74.18) ;
				\draw [shift={(426.73,74.18)}, rotate = 270.36] [color={rgb, 255:red, 0; green, 0; blue, 0 }  ][line width=0.75]    (0,3.35) -- (0,-3.35)   ;
				\draw [shift={(392.82,74.18)}, rotate = 270.12] [fill={rgb, 255:red, 0; green, 0; blue, 0 }  ][line width=0.08]  [draw opacity=0] (7.14,-3.43) -- (0,0) -- (7.14,3.43) -- (4.74,0) -- cycle    ;
				\draw    (426.73,74.18) .. controls (426.77,53.27) and (457.78,52.33) .. (460.45,71.36) ;
				\draw [shift={(460.64,74.18)}, rotate = 270.36] [fill={rgb, 255:red, 0; green, 0; blue, 0 }  ][line width=0.08]  [draw opacity=0] (7.14,-3.43) -- (0,0) -- (7.14,3.43) -- (4.74,0) -- cycle    ;
				\draw [shift={(426.73,74.18)}, rotate = 90.12] [color={rgb, 255:red, 0; green, 0; blue, 0 }  ][line width=0.75]    (0,3.35) -- (0,-3.35)   ;
				
				\draw (219,120.33) node [anchor=north west][inner sep=0.75pt]  [font=\scriptsize] [align=left] {1};
				\draw (457.67,120.33) node [anchor=north west][inner sep=0.75pt]  [font=\scriptsize] [align=left] {8};
				\draw (423.17,120.33) node [anchor=north west][inner sep=0.75pt]  [font=\scriptsize] [align=left] {7};
				\draw (389,120.33) node [anchor=north west][inner sep=0.75pt]  [font=\scriptsize] [align=left] {6};
				\draw (355.17,120.33) node [anchor=north west][inner sep=0.75pt]  [font=\scriptsize] [align=left] {5};
				\draw (320.67,120.33) node [anchor=north west][inner sep=0.75pt]  [font=\scriptsize] [align=left] {4};
				\draw (287.17,120.33) node [anchor=north west][inner sep=0.75pt]  [font=\scriptsize] [align=left] {3};
				\draw (253,120.33) node [anchor=north west][inner sep=0.75pt]  [font=\scriptsize] [align=left] {2};
				\draw (305.56,44.23) node [anchor=north west][inner sep=0.75pt]  [font=\scriptsize]  {$1$};
				\draw (272.56,44.23) node [anchor=north west][inner sep=0.75pt]  [font=\scriptsize,color={rgb, 255:red, 0; green, 0; blue, 0 }  ,opacity=1 ]  {$0$};
				\draw (203.16,44.23) node [anchor=north west][inner sep=0.75pt]  [font=\scriptsize]  {$1$};
				\draw (338.56,44.23) node [anchor=north west][inner sep=0.75pt]  [font=\scriptsize,color={rgb, 255:red, 208; green, 2; blue, 27 }  ,opacity=1 ]  {$2$};
				\draw (406.89,44.23) node [anchor=north west][inner sep=0.75pt]  [font=\scriptsize,color={rgb, 255:red, 208; green, 2; blue, 27 }  ,opacity=1 ]  {$2$};
				\draw (440.89,44.23) node [anchor=north west][inner sep=0.75pt]  [font=\scriptsize,color={rgb, 255:red, 208; green, 2; blue, 27 }  ,opacity=1 ]  {$2$};
				\draw (185,120.33) node [anchor=north west][inner sep=0.75pt]  [font=\scriptsize] [align=left] {0};

			\end{tikzpicture}
			\caption{$\overline{\text{PMM}}(1) $ transition rates.}  \label{fig:PMMholes}
		\end{minipage}
	\end{figure}

	%
	%
		%
		%
		%
		%
		%
		%
		%
		%
		%

		\subsection{The interpolating model}\label{sec:main_model}
		
		Recall Remark \ref{rem:SSEP=PMM(0)}, where we made the observation that SSEP$=$PMM($ 0 $). The construction of the interpolating model will be based on two main ingredients: the generalized binomial theorem and the fact that the family $\{\text{PMM}(k)\}_{k\geq0} $ can be seen as a "polynomial basis" for the diffusion coefficient $ D(\rho)=m\rho^{m-1} $.

		\subsubsection{Construction}
		We base our analysis in the next identity: for any $\rho \in (0,1)$
		\begin{equation}\label{eq:expand} m\rho^{m-1} = m(1-(1-\rho))^{m-1} = m\sum_{k\geq 0} \binom{m-1}{k}(-1)^k (1-\rho)^k= \sum_{k\geq 1} \binom{m}{k} (-1)^{k-1} k (1-\rho)^{k-1} \end{equation} where the generalized binomial coefficient is given by the formula
		\begin{equation}\binom{c}{k} = \frac{(c)_k}{k!} = \frac{c(c-1)\cdots(c-(k-1))}{k!},\qquad c\in\R\label{eq:binom}\end{equation} and therefore we have the identity $m\binom{m-1}{k} = (k+1) \binom{m}{k+1} $. This is a particular case of the generalized binomial expansion for real coefficients:
		\begin{Prop}[Generalized Binomial Theorem]\label{th:gen_bin}
			For any $ x,y,c\in\mathbb{R} $ such that $ \abs{x}>\abs{y} $ we have that
			\begin{align*}
				(x+y)^c&=\sum_{k=0}^\infty\binom{c}{k}x^{c-k}y^k, 
			\end{align*}
			where $\binom{c}{k}$ has been defined in \eqref{eq:binom}.
		\end{Prop}
		\begin{proof}
			The proof is standard and as such we only outline the main steps. Without loss of generalization let $ x\neq0 $. Writing $ z=y/x $ we have $ (x+y)^c=x^c(1+z)^c $. Let $f(z)=(1+z)^c$ be defined for $|z|<1$. Then, by induction we see that $ \frac{d^kf}{dz^k}(z)=(c)_k(1+z)^{c-k} $ for any $ k\geq1 $ integer. To conclude we recall the Taylor expansion of $ f $ and apply Lemma \ref{lem:bin_bound} stated below,  which guarantees the convergence.
		\end{proof}
		
		Proposition \ref{th:gen_bin} implies the convergence of the series appearing in \eqref{eq:expand} for any $ \rho\in(0,1) $.
		For $ \rho\in\{0,1\} $ and $ m\in(1,2) $ or $ \rho=1 $ and $ m\in(0,1) $ one can also easily guarantee the convergence by replacing $ \rho $ by $ 1 $ or $ 0 $ in  each term of the series as written in \eqref{eq:expand}. For $ m\in(0,1) $ and $ \rho=0 $ the series is divergent. This will not be a problem, since due to the gradient property of the model {we shall see that the main object of study will be $ \rho^m $ and not $ \rho^{m-1} $.}
		
		\begin{Def}[Interpolating model] Let $m\in{[0,2]}$, $ N\in\mathbb{N}_+ $ and $ \ell_N\in\mathbb{N},$ with $\ell_N\geq 2 $. We define the generator
			\begin{align}
				\mathcal{L}_N^{(m-1)}: =  \sum_{k=1}^{\ell_N} \binom{m}{k}(-1)^{k-1} \mathcal{L}_N^{\overline{\text{PMM}}(k-1)}   \label{PMM:m1}
			\end{align} where  $\mathcal{L}_N^{\overline{\text{PMM}}(k)}$ has been defined in \eqref{eq:dualPM}.
			More precisely, this generator acts on functions $ f:\Omega_N\to\mathbb{R} $ as
			\begin{align*}
				(\mathcal{L}_N^{(m-1)}f)(\eta)
				=\sum_{x\in\mathbb{T}_N}c_N^{(m-1)}(\tau_x\eta)\mathbf{a}_{x,x+1}(\eta)(\nabla_{x,x+1}f)(\eta),
			\end{align*}
			where \begin{equation}\label{eq:transitionrates}
				c_N^{(m-1)}(\eta)=\sum_{k=1}^{\ell_N} \binom{m}{k} (-1)^{k-1} \mathbf{c}^{(k-1)}(\overline\eta)
			\end{equation}
			and we shorten the rate $ r_N^{(m-1)}(\eta)=c_N^{(m-1)}(\eta)\;\mathbf{a}(\eta) $. We call \emph{non integer porous medium model} (resp.~\emph{fast diffusion model}), and we denote it by PMM($ m-1 $) (resp.~by FDM($ m-1 $)), the Markov process whose infinitesimal generator is given by \eqref{PMM:m1} with $ m\in(1,2) $ (resp.~$ m\in(0,1) $).
		\end{Def}

		\begin{Rem}[About the restrictions on $ \ell_N $]
			Although there is no particular assumption on the order at which $ \ell_N\to+\infty $, note that if $ \ell_N>N $ then for $ N\leq k \leq\ell_N $ we have that $ \mathbf{r}^{(k)}(\eta)\neq0 $ if, and only if, every site is occupied except one at the node $ \{0,1\} $. Due to the mass conservation, this would be achievable only by starting from a configuration with one empty site only, hence no macroscopic evolution of the local density. This is a particular technical consequence of working on the torus, therefore we assume throughout the paper that $ \ell_N\leq N $.
		\end{Rem}
		The goal now is to show that the model is well-defined. In other words, {we are going to prove} that the map $ \eta\mapsto c_N^{(m-1)}(\eta) $ is non-negative. The key argument is the following remark about the sign of $(-1)^{k-1} \binom{m}{k}$. By definition, \begin{itemize} \item if $m\in (0,1)$, then $
			(-1)^{k-1} \binom{m}{k} >0$ for any $ k \geq 1$,
			\item  if $m\in(1,2)$, then  \[
			(-1)^{k-1} \binom{m}{k}   >0 \quad  \text{ if } k =1, \quad \text{ and }  \quad (-1)^{k-1} \binom{m}{k} <0 \quad \text{ if } k \geq 2.\] \end{itemize} Therefore   we can rewrite
		\begin{align}\label{PMM_rewrite}
			\mathcal{L}_N^{(m-1)}
			=
			m\mathcal{L}_N^{\text{SSEP}}
			-\text{sign}(m-1)
			\sum_{k=2}^{\ell_N}\abs{\binom{m}{k}}\mathcal{L}_N^{\overline{\text{PMM}}(k-1)}
			,
			\qquad m\in(0,2)\backslash\{1\}.
		\end{align}
		%
		We also need non-asymptotic bounds for the generalized binomial coefficients: from Lemma \ref{lem:bin_bound} one can extract that for $ m\in\mathbb{R} $ and $ k\geq2 $
		\begin{align}\label{eq:usefulbound}
			\frac{1}{(k+1)^m}\lesssim \abs{\binom{m-1}{k}}\lesssim \frac{1}{k^m}. 
		\end{align}
		{The notation $ f(k)\lesssim g(k) $ shortens that {there exists $ C>0 $}, {such that} for all $ k\in\mathbb{N}$, $ \abs{f(k)}\leq C\abs{g(k)} $.}

		Now we state and prove the main technical result of this section, which contains two estimates: {the lower bounds show that the generators are well-defined and permit to prove an energy bound (given in Proposition \ref{prop:energy}), which is essential to the proof of the forthcoming \emph{replacement lemmas}; the upper bounds reflect the boundedness of the rates as $ N\to+\infty $.}
		
		\begin{Prop}\label{prop:low_bound_r} If $ \ell_N\gg1 $, then for any $\eta\in\Omega_N$,
			\begin{align*}
				r_N^{(m-1)}(\eta)\geq
				\begin{cases}
					m \; \mathbf{r}^{(0)}(\eta), &m\in(0,1), \vphantom{\Big)}\\
					m\delta_N \; \mathbf{r}^{(0)}(\eta)+\binom{m}{2}\;\mathbf{r}^{(1)}(\eta), &m\in(1,2),
				\end{cases}
				\quad\text{and}\quad
				r_N^{(m-1)}(\eta)
				\leq\begin{cases} \displaystyle
					\sum_{k=1}^{\ell_N}\abs{\binom{m}{k}}k, &m\in(0,1),\\
					\displaystyle	m \mathbf{r}^{(0)}(\eta), &m\in(1,2),\vphantom{\bigg(}
				\end{cases}
			\end{align*}
			where $ (\ell_N+1)^{-(m-1)}\lesssim\delta_N=\sum_{k\geq \ell_N}\abs{\binom{m-1}{k}}\lesssim (\ell_N)^{-(m-1)} $. Moreover, when $m\in(0,1)$,
			\begin{align*}
				\sum_{k=1}^{\ell_N}\abs{\binom{m}{k}}k
				=
				\max_{\eta\in\Omega_N}r_N^{(m-1)}(\eta)\xrightarrow[N\to+\infty]{}+\infty.
			\end{align*}
		\end{Prop}
		
		\begin{proof}
			We start with the case $ m\in (1,2) $. From \eqref{PMM_rewrite}, we rewrite
			\begin{align*}
				r_N^{(m-1)}(\eta)
				&=m-\sum_{k=2}^{\ell_N}\abs{\binom{m}{k}}k
				+\sum_{k=2}^{\ell_N}\abs{\binom{m}{k}}\big(k-\mathbf{r}^{(k-1)}(\overline{\eta})\big)
				\geq
				m-\sum_{k=2}^{\ell_N}\abs{\binom{m}{k}}k
				+\binom{m}{2}\big(2-\mathbf{r}^{(1)}(\overline{\eta})\big),
			\end{align*}
			where for the last inequality we used the fact that, by definition, $ \mathbf{c}^{(k-1)}(\overline{\eta})\leq k $, and we bounded from below all but the first term of the second summation in $ k $ by zero. Then, since the alternating sum of the binomial coefficients vanishes, we obtain, for any $ \ell_N\in\mathbb{N}_+ $, that
			\begin{align}\label{ineq:sum_coeff}
				m-\sum_{k=2}^{\ell_N}\abs{\binom{m}{k}}k
				= m\bigg(1-\sum_{k=1}^{\ell_N-1}\abs{\binom{m-1}{k}}\bigg) > m \bigg(1-\sum_{k=1}^{+\infty}\abs{\binom{m-1}{k}}\bigg)=0
			\end{align}
			and	therefore we get that $ r_N^{(m-1)}>0 $. To conclude, we note that  $2-\mathbf{c}^{(1)}(\overline{\eta}) = \mathbf{c}^{(1)}(\eta)$ and we set
			\begin{align*}
				\delta_N:
				=1-\sum_{k=1}^{\ell_N-1}\abs{\binom{m-1}{k}}
				=\sum_{k\geq\ell_N}\abs{\binom{m-1}{k}} >0.
			\end{align*}
			Recalling \eqref{eq:usefulbound}, we are reduced to estimate the tail of the $ m-$series:
			\begin{align}\label{eq:m-series}
				\frac{c}{(m-1)(\ell_N+1)^{m-1}}
				\leq
				\sum_{k\geq \ell_N+1}\frac{1}{k^m}\leq \frac{C}{(m-1)(\ell_N)^{m-1}}
			\end{align}
			with $ c,C>0 $ being constants independent of $N $. Putting the inequalities together, the proof of the lower bound follows. To prove the upper bound, we only  keep the first term  in the definition \eqref{eq:transitionrates} of $ r_N^{(m-1)} $, since the other ones are negative.
			
			The case $ m\in (0,1) $ is straightforward from \eqref{PMM_rewrite}.
			To conclude, we see that the maximum is obtained when $ \mathbf{r}^{(k-1)}(\overline{\eta})=k $, that is, when the window $ \llbracket -\ell_N+1,\ell_N\rrbracket\backslash \{0,1\} $ is completely empty and $ \eta(0)+\eta(1)=1 $. The lower bound for the binomial coefficients in \eqref{eq:usefulbound} then shows that this maximum tends to infinity as $ N\to+\infty $.
		\end{proof}
		\begin{Rem}[On the sharpness of the bounds in Proposition \ref{prop:low_bound_r}]\label{rem:bound} The estimates of Proposition \ref{prop:low_bound_r}  are not sharp. Instead, the goal of the lower bound for $ m\in(1,2) $ is to relate our process with the simpler process induced by the generator
			\begin{align*}
				m\delta_N\mathcal{L}_N^{\text{PMM}(0)}+\tfrac{m(m-1)}{2}\mathcal{L}_N^{\text{PMM}(1)},
			\end{align*}
			which is very close to the one studied in \cite{GLT}, where the porous medium model is perturbed by a "small" SSEP dynamics.
			
			The lower bound for $ m\in (0,1) $ is here to emphasize that the transition rates will always be \textit{greater} than those of the SSEP (modulo a multiplicative constant), as expected, since under this regime the macroscopic diffusion is faster than the one of the heat equation ($ m=1 $). This will be useful, in particular, for the proof of the  replacement  Lemma \ref{lem:rep_FDM}.
			
			Finally, let us highlight that the divergence $ \max_{\eta\in\Omega_N}r_N^{(m-1)}(\eta)\to +\infty $ as $N \to+\infty$ gives us an extra difficulty in the proof of tightness (see in particular \eqref{h:treat_FDM}) {and makes it impossible to argue, as for $ m\in(1,2) $, that $ \rho^m $ is weak differentiable (see the last step in the proof of Proposition \ref{prop:energy_est_PME})}.
		\end{Rem}

		\subsection{Characterization of the interpolating family}\label{subsec:interp_prop}
		
		In this subsection we present further properties of the interpolating model. We start by explaining how this model interpolates between the SSEP and the PMM($ 1 $).
		\begin{Prop}[Interpolation property] \label{prop:interp}
			For $ m\in(1,2) $, $ N\in\mathbb{N} $ and $ \ell_N\geq2 $ fixed, the process $\mathcal{L}_N^{(m-1)}$ interpolates between $\mathcal{L}_N^{\mathrm{PMM}(0)}$ and $\mathcal{L}_N^{\mathrm{PMM}(1)}$ in the following sense: for all $ \eta\in\Omega_N $,
			\begin{align}\label{interp_1}
				\lim_{m \nearrow 1}r_N^{(m-1)}(\eta)
				=\mathbf{r}^{\mathrm{PMM}(0)}(\eta)
				=\lim_{m\searrow 1}r_N^{(m-1)}(\eta)
				\qquad \text{and}\qquad
				\lim_{m\nearrow 2}r_N^{(m-1)}(\eta)
				=\mathbf{r}^{\mathrm{PMM}(1)}(\eta).
			\end{align}
		\end{Prop}
		\begin{proof}
			The limit to SSEP as $ m\to1 $ from either above or below is a direct consequence of the interpolation property of the binomial coefficients, while the limit to PMM($ 1 $) is both consequence of this, but also of {some rearrangement in the summation which defines the rates, and which implies $2-\mathbf{c}^{(1)}(\overline{\eta}) = \mathbf{c}^{(1)}(\eta)$,} see also \eqref{eq:summation} below.
		\end{proof}
		
		{From \cite{GLT},} the \emph{grand-canonical invariant measures} for the PMM($k$) (and therefore for the $\overline{\text{PMM}}(k)$) are the Bernoulli product measures $\nu_\rho^N$   of parameter $\rho\in[0,1]$, namely, their marginal  is given on $x\in\T_N$ by
		\begin{equation} \label{eq:bernoulli} \nu_\rho^N\big(\eta\in\Omega_N:~\eta(x)=1\big)=\rho.\end{equation}
		The next lemma gives information on the invariant measures of our models.
		\begin{Lemma}[Invariant measures and irreducibility]\label{lem:first}
			{Let $m\in(0,2)$.}	For any $\rho\in[0,1]$, the Bernoulli product measure $\nu_\rho^N$ defined in \eqref{eq:bernoulli} is invariant for the Markov process generated by $\mathcal{L}_N^{(m-1)}$. Moreover, for any  $k\in\{0,\dots,N\}$, the hyperplane
			\[ \mathscr{H}_{N,k}=\Big\{\eta\in\Omega_N \; : \; \sum_{x\in\T_N}\eta(x)=k\Big\}
			\] is irreducible under the Markov process generated by $\mathcal{L}_N^{({m-1})}$.
		\end{Lemma}
		\begin{proof}
			The irreducibility of the process on the above hyperplanes is consequence of the fact that  $ c_N^{({m-1})}(\eta)>0 $ for any $ \eta\in\Omega_N $, as shown in Proposition \ref{prop:low_bound_r}, and so the exclusion rule is the only constraint. We already know from \cite{GLT} that  the product measure $ \nu_\rho^N $ is invariant for $\overline{\text{PMM}}$($ k $), for any $ k\in\mathbb{N}_+ $. In particular, it is also invariant for linear combinations of such models.
		\end{proof}
		
		For a good understanding of the interpolating model it is important to describe some properties of the integer family $ \{\text{PMM}(k)\}_{k\in\mathbb{N}} $. {In the following,} Lemma \ref{lem:up_speed} and Proposition \ref{prop:m_mono} describe new properties of the aforementioned family which will be important later on. Moreover, {thanks to some preliminary computations}, it can be seen that the macroscopic density of particles evolves \emph{diffusively}, with a diffusion coefficient that can be computed explicitly. More precisely, let us introduce the following operator:
		
		\begin{Def}[Translation operators] \label{def:translation}
			Let $ \mathbf{1} $ be the identity function on $ \Omega_N $, and consider the operators $ \nabla^\pm $ associated to the translation operator given by $ \nabla^+=\tau_1-\mathbf{1} $ and $ \nabla^-=\mathbf{1}-\tau_{-1} $, that is, for any function $ f:\Omega_N\to\mathbb{R} $, we define $ (\nabla^+f)(\eta)=f(\tau_{1}\eta)-f(\eta) , (\nabla^-f)(\eta)=f(\eta)-f(\tau_{-1}\eta) $, and for any $ x\in\mathbb{T}_N $ consider
			$
			(\nabla_x^\pm f)(\eta)=(\nabla^\pm f)(\tau_x\eta).
			$

		\end{Def}
		
		{As noted in \cite{GLT}}, it is straightforward to check
		that, for any $x\in\T_N$,
		\begin{equation}\label{eq:gen-current}		\mathcal{L}_N^{{\text{PMM}}(k)}\big(\eta(x)\big)
			=\nabla^- \left(
			\mathbf{c}^{(k)}(\tau_x\eta)\nabla^+\eta(x)
			\right).
		\end{equation}
		Therefore, 		  the \textit{microscopic density current} for PMM($k$) between  sites $x$ and $x+1$, is equal to
		\begin{align*}
			{-\mathbf{c}^{(k)}(\tau_x\eta)\nabla^+\eta(x)
			}=:\mathbf{j}_{\{x,x+1\}}^{(k)}(\eta).
		\end{align*}
		It turns out, see \cite{GLT},  that this quantity is itself a discrete gradient, namely
		\begin{align*}
			\mathbf{j}_{\{x,x+1\}}^{(k)}(\eta)=\nabla^+\mathbf{h}^{(k)}(\eta),
		\end{align*} where $\mathbf{h}^{(k)}$ is given in Lemma \ref{lem:grad}. 
		We highlight that although this gradient property {was already} known (see \cite{GNP21} for instance), the expression \eqref{expr:h2} for $\mathbf{h}^{(k)}$ is new ({we give the original expression of $\mathbf{h}^{(k)}$ in the appendix, see \eqref{expr:h1}}). {Then, note that} the expectation of $\mathbf{c}^{(k)}(\tau_x\eta)$ under the invariant measure $\nu_\rho^N$ is 
		\begin{align}\label{diff:nu}
			\int \mathbf{c}^{(k)}(\tau_x\eta) \mathrm{d}\nu_\rho^N(\eta)=  (k+1)\rho^{k}=D(\rho)
		\end{align}
		which is the diffusion coefficient of the PME($ k $)  \eqref{PDE:formal}, \textit{i.e.}, for $m=k+1 \in \mathbb{N}_+$.  Similarly, since $ \eta(1)-\eta(0)=-(\overline{\eta}(1)-\overline{\eta}(0)) $, the gradient property is also true for $\overline{\text{PMM}}$($k$). One can readily check that the expected diffusion equation associated to the microscopic dynamics of  $\overline{\text{PMM}}(k)$ has diffusion coefficient $ \overline{D}(\rho)=(k+1)(1-\rho)^{k}. $
		{Let us now state more precisely the aforementioned gradient property, which} will be proved in Appendix \ref{app:aux_res}. We recall the definition of $ \mathbf{s}_j^{(k)} $ in \eqref{rate:pmm_int}.
		\begin{Lemma}[Gradient property]\label{lem:grad}
			For any $ k\in\mathbb{N} $, \emph{PMM}($ k $) is a gradient model. Precisely, for any $\eta\in\Omega_N$ we have that  
			$
			\mathbf{c}^{(k)}(\eta)\nabla^+\eta(0)
			=
			\nabla^+\mathbf{h}^{(k)}(\eta),
			$ 
			where
			\begin{align}
				\mathbf{h}^{(k)}(\eta)
				&=\prod_{i=0}^{k}\eta(i)
				+\sum_{i=0}^{k-1}
				(\eta(i)-\eta(i+1))\sum_{j=1}^{k-i}\mathbf{s}_j^{(k)}(\tau_i\eta).
				\label{expr:h2}
			\end{align}
		\end{Lemma}
		
		Now, for the interpolating model {generated by $\mathcal{L}_N^{(m-1)}$}, similarly to \eqref{eq:gen-current}, a straightforward computation gives for all $ x\in\mathbb{T}_N $
		\begin{align}\label{curr}
			\mathcal{L}_N^{(m-1)}\big(\eta(x)\big)=N^2\nabla^{-}\left( c_N^{(m-1)}(\tau_x\eta)\nabla^{+}\eta(x)\right), 
		\end{align}
		and we can {easily deduce from the previous lemma that}
		\begin{align}\label{grad:non_int}
			c_N^{(m-1)}(\eta)\nabla^{+}\eta(0)
			=
			\nabla^+h_N^{(m-1)}(\eta), \quad\text{where}\quad
			h_{N}^{(m-1)}(\eta)
			=\sum_{k=1}^{\ell_N}\binom{m}{k}(-1)^k
			\mathbf{h}^{(k-1)}(\overline{\eta}).
		\end{align}

		\subsection{{Properties on the rates}}
		
		We {start} by stating and proving two important properties {of the basis family $\{\text{PMM}(k)\}_{k\in\mathbb N}$}. The first one {(Lemma \ref{lem:up_speed})} {will be used later} in Propositions \ref{prop:tight}, \ref{prop:energy} and Lemma \ref{lem:rep_FDM-tight}, while the second {one (Proposition \ref{prop:r_seq}) will provide some interesting monotonicity property of the rates for both} the integer and non-integer families, see Propositions \ref{prop:m_mono} and \ref{prop:m_mono2} {at the end of this section}. Recall the definition of $ \mathbf{r}^{(k)} $ from \eqref{eq:ck}.
		
		\begin{Lemma}[Bound on the rates]\label{lem:up_speed}
			For all $ \ell,k\in\mathbb{N}_+ $ such that $ \ell\geq k $ and any $ \eta\in\Omega_N $ we have that
			\begin{align*}
				\sum_{n=1}^\ell\mathbf{r}^{(k)}(\tau_n\eta)
				\leq 2(\ell+k).
			\end{align*}
		\end{Lemma}
		\begin{proof}
			Note that
			\begin{align*}
				\sum_{n=1}^\ell\mathbf{r}^{(k)}(\tau_n\eta)
				=\sum_{n=1}^\ell
				\mathbf{a}(\tau_n\eta)
				\sum_{j=1}^{k+1}
				\mathbf{s}_j^{(k)}(\tau_n\eta)
				=\sum_{p=2}^{\ell+k+1}
				\sum_{n=1}^{\ell}
				\sum_{j=1}^{k+1}
				\mathbf{a}(\tau_n\eta)\mathbf{s}_j^{(k)}(\tau_n\eta)
				\mathbf{1}_{\{j+n=p\}}
				\leq
				\sum_{p=2}^{\ell+k+1}2=2(\ell+k).
			\end{align*}
			The inequality can be justified as follows.	Fixed $ p $, the quantity $ \mathbf{s}_j^{(k)}(\tau_n\eta) $ depends on the occupation of the sites
			\begin{align*}
				\llbracket -(k+1)+j+n,j+n\rrbracket\backslash\{n,n+1\}
				=\llbracket -(k+1)+p,p\rrbracket\backslash\{n,n+1\}.
			\end{align*}
			Because $ 1\leq j\leq k+1 $, then $ \{n,n+1\}\in \llbracket -(k+1)+p,p\rrbracket $ for sure. There are a number of pairs $ (j,n) $ such that $ j+n=p $, but for all of those pairs the box $ \llbracket -(k+1)+p,p\rrbracket$ is the same. Thus, for each $ p $ fixed, there are at most two pairs $ (n,j),(n',j') $ such that $ p=n+j=n'+j' $ and $ \mathbf{a}(\tau_n\eta)\mathbf{s}_j^{(k)}(\tau_n\eta)=\mathbf{a}(\tau_{n'}\eta)\mathbf{s}_{j'}^{(k)}(\tau_{n'}\eta)=1 $. Specifically, if $ (n,j) $ is as previously, then $ (n',j')=(n+1,j-1) $ or $ (n',j')=(n-1,j+1) $.
			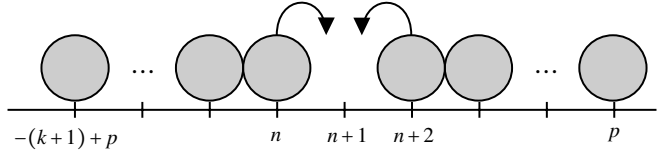
\begin{figure}[H]
				\centering

				\tikzset{every picture/.style={line width=0.75pt}} 
				
				\begin{tikzpicture}[x=0.75pt,y=0.75pt,yscale=-1,xscale=1]
					
					\draw    (220.3,140) -- (547.6,140) (254.3,136) -- (254.3,144)(288.3,136) -- (288.3,144)(322.3,136) -- (322.3,144)(356.3,136) -- (356.3,144)(390.3,136) -- (390.3,144)(424.3,136) -- (424.3,144)(458.3,136) -- (458.3,144)(492.3,136) -- (492.3,144)(526.3,136) -- (526.3,144) ;
					\draw  [fill={rgb, 255:red, 155; green, 155; blue, 155 }  ,fill opacity=0.5 ] (305.23,118.85) .. controls (305.23,109.74) and (312.82,102.35) .. (322.18,102.35) .. controls (331.55,102.35) and (339.14,109.74) .. (339.14,118.85) .. controls (339.14,127.96) and (331.55,135.35) .. (322.18,135.35) .. controls (312.82,135.35) and (305.23,127.96) .. (305.23,118.85) -- cycle ;
					\draw [color={rgb, 255:red, 0; green, 0; blue, 0 }  ,draw opacity=1 ]   (356.09,102.35) .. controls (356.13,81.44) and (378.95,80.45) .. (380.88,99.48) ;
					\draw [shift={(381.01,102.3)}, rotate = 270.36] [fill={rgb, 255:red, 0; green, 0; blue, 0 }  ,fill opacity=1 ][line width=0.08]  [draw opacity=0] (8.93,-4.29) -- (0,0) -- (8.93,4.29) -- cycle    ;
					\draw [color={rgb, 255:red, 0; green, 0; blue, 0 }  ,draw opacity=1 ]   (399.18,99.28) .. controls (401.47,80.55) and (424.04,81.55) .. (423.91,102.35) ;
					\draw [shift={(398.99,102.4)}, rotate = 270.12] [fill={rgb, 255:red, 0; green, 0; blue, 0 }  ,fill opacity=1 ][line width=0.08]  [draw opacity=0] (8.93,-4.29) -- (0,0) -- (8.93,4.29) -- cycle    ;
					\draw  [fill={rgb, 255:red, 155; green, 155; blue, 155 }  ,fill opacity=0.5 ] (440.86,118.85) .. controls (440.86,109.74) and (448.45,102.35) .. (457.82,102.35) .. controls (467.18,102.35) and (474.77,109.74) .. (474.77,118.85) .. controls (474.77,127.96) and (467.18,135.35) .. (457.82,135.35) .. controls (448.45,135.35) and (440.86,127.96) .. (440.86,118.85) -- cycle ;
					\draw  [fill={rgb, 255:red, 155; green, 155; blue, 155 }  ,fill opacity=0.5 ] (406.95,118.85) .. controls (406.95,109.74) and (414.55,102.35) .. (423.91,102.35) .. controls (433.27,102.35) and (440.86,109.74) .. (440.86,118.85) .. controls (440.86,127.96) and (433.27,135.35) .. (423.91,135.35) .. controls (414.55,135.35) and (406.95,127.96) .. (406.95,118.85) -- cycle ;
					\draw  [fill={rgb, 255:red, 155; green, 155; blue, 155 }  ,fill opacity=0.5 ] (339.14,118.85) .. controls (339.14,109.74) and (346.73,102.35) .. (356.09,102.35) .. controls (365.45,102.35) and (373.05,109.74) .. (373.05,118.85) .. controls (373.05,127.96) and (365.45,135.35) .. (356.09,135.35) .. controls (346.73,135.35) and (339.14,127.96) .. (339.14,118.85) -- cycle ;
					\draw  [fill={rgb, 255:red, 155; green, 155; blue, 155 }  ,fill opacity=0.5 ] (237.41,118.85) .. controls (237.41,109.74) and (245,102.35) .. (254.36,102.35) .. controls (263.73,102.35) and (271.32,109.74) .. (271.32,118.85) .. controls (271.32,127.96) and (263.73,135.35) .. (254.36,135.35) .. controls (245,135.35) and (237.41,127.96) .. (237.41,118.85) -- cycle ;
					\draw  [fill={rgb, 255:red, 155; green, 155; blue, 155 }  ,fill opacity=0.5 ] (508.68,118.85) .. controls (508.68,109.74) and (516.27,102.35) .. (525.64,102.35) .. controls (535,102.35) and (542.59,109.74) .. (542.59,118.85) .. controls (542.59,127.96) and (535,135.35) .. (525.64,135.35) .. controls (516.27,135.35) and (508.68,127.96) .. (508.68,118.85) -- cycle ;
					
					\draw (521.67,147.5) node [anchor=north west][inner sep=0.75pt]  [font=\scriptsize] [align=left] {$ p $};
					\draw (413.67,147.5) node [anchor=north west][inner sep=0.75pt]  [font=\scriptsize] [align=left] {$ n+2 $};
					\draw (380.17,147.5) node [anchor=north west][inner sep=0.75pt]  [font=\scriptsize] [align=left] {$ n+1 $};
					\draw (351.67,149.5) node [anchor=north west][inner sep=0.75pt]  [font=\scriptsize] [align=left] {$ n $};
					\draw (222,147.5) node [anchor=north west][inner sep=0.75pt]  [font=\scriptsize] [align=left] {$ -(k+1)+p $};
					\draw (281.32,118.25) node [anchor=north west][inner sep=0.75pt]    {$\cdots $};
					\draw (484.77,118.25) node [anchor=north west][inner sep=0.75pt]    {$\cdots $};

				\end{tikzpicture}
				\caption{Configuration with $ \mathbf{a}(\tau_n\eta)\mathbf{s}_j^{(k)}(\tau_n\eta)=\mathbf{a}(\tau_{n+1}\eta)\mathbf{s}_{j-1}^{(k)}(\tau_{n+1}\eta)=1 $ and $ p $ fixed.}
			\end{figure}
		\end{proof}
		
		Now we state a monotonicity property. {The following proposition} is used {right after} in Proposition \ref{prop:m_mono2} to prove an analogous property for the {interpolating model}.
		\begin{Prop}\label{prop:r_seq}
			For any $ \eta\in\Omega_N $, the sequence $ \big\{\frac1k \mathbf{c}^{(k-1)}(\eta)\big\}_{k\geq 1} $ is non-increasing.
		\end{Prop}
		\begin{proof} {In order to prove the result, it is enough to show that
				\[u_k(\eta):=\frac{k+1}{k} \mathbf{c}^{(k-1)}(\eta) - \mathbf{c}^{(k)}(\eta) \geqslant 0,\] for any $\eta\in\Omega_N$.
				It turns out that this expression can we rewritten in terms of the products $\mathbf{s}_j^{(k)}$ defined in \eqref{eq:defs}, after flipping some of the configuration values $\eta(x)$. Let us be more precise. }
			
			To simplify the presentation let us introduce some notation: for any $ A\subseteq\mathbb{T}_N $ define the flip $ \eta\mapsto\overline{\eta}^A $ as
			$
			\overline{\eta}^A(x)
			=\overline{\eta}(x)\mathbf{1}_{x\in A}
			+\eta(x)\mathbf{1}_{x\notin A}
			$. 
			{Straightforward computations show that
				\begin{equation}\label{eq:summation}u_k(\eta)= \sum_{j=1}^{k+1} \bigg\{\frac{k-(j-1)}{k}
					\mathbf{s}_j^{(k)}
					\big(\overline{\eta}^{\{-(k+1)+j\}}\big)
					+
					\frac{j-1}{k}
					\mathbf{s}_j^{(k)}
					\big(\overline{\eta}^{\{j\}}\big)\bigg\}. \end{equation}
				Indeed, this is a consequence of the fact that}
		{\begin{itemize}\item  for any $ j\in\llbracket 1,k\rrbracket $ it holds
				\begin{align*}
					\mathbf{s}_j^{(k)}\big(\eta^{\{-(k+1)+j\}}\big)
					&=\overline{\eta}(-(k+1)+j)\mathbf{s}_j^{(k-1)}(\eta)
					=\mathbf{s}_j^{(k-1)}(\eta)-\mathbf{s}_j^{(k)}(\eta)
				\end{align*}
				\item and for any $ j\in\llbracket 2,k+1\rrbracket $ we have
				\begin{align*}
					\mathbf{s}_j^{(k)}\big(\eta^{\{j\}}\big)
					&=
					\mathbf{s}_{j-1}^{(k-1)}(\eta)\overline{\eta}(j)
					=\mathbf{s}_{j-1}^{(k-1)}(\eta)-\mathbf{s}_j^{(k)}(\eta).
		\end{align*}\end{itemize}}
		{Two changes of variables in the two terms of the summation in \eqref{eq:summation} then lead to the desired result.}
	\end{proof}

	Due to the analytical nature of the generalized binomial coefficients, a combinatorial interpretation of the whole model is not appropriate, as opposed to the integer case. Additionally, the problem of quantifying how, fixed some configuration, the rates change by varying $ m $ is not easy since the rates {depend in a complex manner on} $ m $ and  the  behaviour of the rate (with respect to $ m $) is different for distinct configurations. Instead of doing an extensive study of the form of the rates, we gather information about some simple monotonicity aspects of the model. We show that the reinforcement/penalization of the SSEP {given} in \eqref{PMM_rewrite} {is} non-increasing in $ k $;  {then we derive} a property of the interpolating family analogous to Proposition \ref{prop:r_seq}; and {finally} we plot on Figure \ref{fig:1} the rates in some equivalence classes of configurations which cover the values of $ c_N^{(m-1)}(\eta) $. This is, to our mind, a satisfying solution to observe the continuous deformation of the SSEP into a slow or fast diffusion model.
	\begin{Prop}\label{prop:m_mono}
		Fixed any $ \eta\in\Omega_N $ and $ m\in [0,2] $ the sequence $ \{\abs{\binom{m}{k}}\mathbf{c}^{(k-1)}(\overline{\eta})\}_{k\geq 2} $ 
		is decreasing up to the smallest $ k $ such that $ \mathbf{c}^{(k-1)}(\overline{\eta})=0 $.
	\end{Prop}
	\begin{proof}
		Recall that we proved in Proposition \ref{prop:r_seq} that for any $ \eta\in\Omega_N $ the sequence $ \big\{\frac1k \mathbf{c}^{(k-1)}(\eta)\big\}_{k\geq 1} $ is non-increasing. From the definition of the binomial coefficients, for $m\in(0,2)$ the sequence $\big\{k\abs{\binom{m}{k}}\big\}_{k\geq 2}$ is decreasing, since
		\[ (k+1)\abs{\binom{m}{k+1}} = k \abs{\binom{m}{k}} \; \frac{|m-k|}{k},\]
		and whenever $k\geq 2$ and $m\in(0,2)$ we have $|m-k|=k-m<k$.
	\end{proof}
	
	Before stating the monotonicity property, note that we have the following limit 
	\begin{align*}
		\lim_{m \searrow 0}\frac{1}{m}c_N^{(m-1)}(\eta)
		=\sum_{k=0}^{\ell_N-1}\frac{\mathbf{c}^{(k)}(\overline{\eta})}{k+1}.
	\end{align*}
	
	\begin{Prop} \label{prop:m_mono2}
		For any $ \eta\in\Omega_N $ the sequence $ \{\frac1mc_N^{(m-1)}(\eta)\}_{m\in[0,2]} $ is non-increasing.
	\end{Prop}
	\begin{proof}
		From Proposition \ref{prop:low_bound_r} we can extract that $ \frac1mc_N^{(m-1)}\geq \mathbf{c}^{(0)} $ for $ m\in(0,1) $, and $ \mathbf{c}^{(0)}\geq \frac1mc_N^{(m-1)} $ for $ m\in(1,2) $. It remains to see the monotonicity of the sequence in the statement {according to the values of} $ m\in[0,2]\backslash\{1\} $. Assuming that the aforementioned sequence is non-increasing, since the binomial coefficients are continuous functions of $ m $ the interpolation property allows us to take the limit $ m\to 2 $ and as such we only need  to focus on $ m\in[0,2)\backslash\{1\} $. Rewrite 
		\begin{align*}
			\frac1m c_N^{(m-1)}(\eta)
			&=
			\mathbf{1}_{\{m\in(0,1)\}}\sum_{k=0}^{\ell_N-1} \abs{\binom{m-1}{k}}\frac{\mathbf{c}^{(k)}(\overline{\eta})}{k+1}
			+
			\mathbf{1}_{\{m\in(1,2)\}}
			\left(
			1-\sum_{k=1}^{\ell_N-1} \abs{\binom{m-1}{k}}\frac{\mathbf{c}^{(k)}(\overline{\eta})}{k+1}
			\right).
		\end{align*}
		For any $ k\geq 2 $ we compute
		\begin{align*}
			\frac{\rmd}{\rmd m}\abs{(m-1)_k}
			&=-\abs{(m-1)_k}f_k(m) 
			\quad{\text{where}}\quad f_k(m):=\sum_{j=1}^k\frac{1}{j-m}.
		\end{align*}
		This means that
		\begin{align}\label{mono_0}
			\frac{\rmd }{\rmd m}\bigg(\frac1m c_N^{(m-1)}(\eta)\bigg)
			&=-\frac{1}{2}\mathbf{c}^{(1)}(\overline{\eta})
			+\text{sign}(m-1)\sum_{k=2}^{\ell_N-1} \abs{\binom{m-1}{k}}
			f_k(m)
			\frac{\mathbf{c}^{(k)}(\overline{\eta})}{k+1}.
		\end{align}
		If $ m\in[0,1) $ then $ f_k(m)>0 $ which concludes the proof. For $ m\in(1,2) $ we need some extra work. We claim that differentiating with respect to $ m $ both sides of 
		\begin{align}\label{mono_1}
			0=1-\sum_{k=1}^{+\infty} \abs{\binom{m-1}{k}}
			\quad\text{one obtains that}\quad
			1=\sum_{k=2}^{+\infty} \abs{\binom{m-1}{k}}f_k(m).
		\end{align}
		For $ m\in[\frac32,2) $ we have $ f_k(m)>0 $ for all $ k\geq2 $ since $ f_2(m)>0 $ and $ f_k(m) $ is increasing in $ k $. If $ m\in(1,\frac32) $ then for each $ m $ there must be some $ k_0>2 $ such that $ f_k>0 $ for all $ k\geq k_0 $ so that the second summation on the previous display is equal to one. Let $ \ell_N $ be large enough so that $ k_0<\ell_N $ (otherwise the {result is obvious}). Then we can bound \eqref{mono_0} from above by taking the limit $ \ell_N\to+\infty $. Since the sequence of maps $ \{\frac{1}{k+1}\mathbf{c}^{(k)}\}_{k\geq0} $ is non-increasing, for any $ j\leq k_0\leq i$ we have
		\begin{align*}
			\frac{1}{i+1}\mathbf{c}^{(i)}
			\leq
			\frac{1}{k_0+1}\mathbf{c}^{(k_0)}
			\leq
			\frac{1}{j+1}\mathbf{c}^{(j)}.
		\end{align*}
		Then we can bound
		\begin{align}\label{mono_2}
			\begin{split}
				\frac{\rmd }{\rmd m}\bigg(\frac1m c_N^{(m-1)}(\eta)\bigg)
				&\leq
				-\frac{1}{2}\mathbf{c}^{(1)}(\overline{\eta})
				+\sum_{k=2}^{k_0-1} \abs{\binom{m-1}{k}}
				f_k(m)
				\frac{\mathbf{c}^{(k_0)}(\overline{\eta})}{k_0+1}
				+\sum_{k\geq k_0}\abs{\binom{m-1}{k}}
				f_k(m)
				\frac{\mathbf{c}^{(k_0)}(\overline{\eta})}{k_0+1}
				\\
				&=
				-\frac{\mathbf{c}^{(1)}(\overline{\eta})}{2}+\frac{\mathbf{c}^{(k_0)}(\overline{\eta})}{k_0+1}\leq 0.
			\end{split}
		\end{align}
		To conclude the proof, it is enough to show that the sequence $ (a_n)_{n\geq2} $ given by $ 0<a_n:=\sum_{k=2}^{n} \abs{\binom{m-1}{k}}f_k(m) $ is uniformly bounded. Since $ f_1(m)<0 $, we  first bound $ f_k(m) $ by the  {corresponding} integral for $ k\geq2 $:
		\begin{align*}
			f_k(m)\lesssim\log(k-m)-\log(2-m).
		\end{align*}
		Recall the inequality $ \log x\leq \frac1s x^s $ for any $ x,s\in\mathbb{R}_+ $. From this and \eqref{eq:usefulbound} it holds
		\begin{align*}
			\sum_{k=2}^{n} \abs{\binom{m-1}{k}}f_k(m)
			\lesssim
			\sum_{k=2}^{n}\frac{1}{(k-m)^{m-s}}
			-
			\log(2-m)
			\sum_{k=2}^{n}\abs{\binom{m-1}{k}}.
		\end{align*}
		Setting $ 0<s $ such that $ m-s>1 $, observing that the quantity  on the right-hand side of the previous display is increasing in $ n $ and taking $ n\to+\infty $ we end the proof.
	\end{proof}
	We now plot the evolution of $ c_N^{(m-1)}(\eta)$ with respect to $m$ (for a fixed configuration $\eta$). To that aim, let us start with the following remark: for any $ k\geq {1} $ the value of $ \mathbf{c}^{(k)}(\overline\eta) $ is uniquely  determined by the positions of the first particle to the left of $ 0 $ and the first particle to the right of $ 1 $. More precisely, for any $ x_0,x_1\in\mathbb{T}_N $ consider the set $ \Omega_N^{x_0,x_1}=\big\{\eta\in\Omega_N\; :\;  \eta(-x_0)=\eta(x_1)=1,\; \eta(x)=0, \;\text{ for all }\: x\in\llbracket -x_0+1,x_1-1\rrbracket\backslash \{0,1\}  \big\}. $
	%
	%
		%
		%
		%
		%
		%
		%

		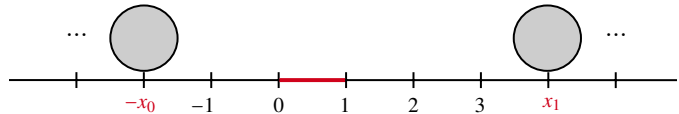
\begin{figure}[H]

			\tikzset{every picture/.style={line width=0.75pt}} 
			
			\begin{tikzpicture}[x=0.75pt,y=0.75pt,yscale=-1,xscale=1]
				
				\draw    (186.3,140) -- (523.2,140.05) (220.3,136.01) -- (220.3,144.01)(254.3,136.01) -- (254.3,144.01)(288.3,136.02) -- (288.3,144.02)(322.3,136.02) -- (322.3,144.02)(356.3,136.03) -- (356.3,144.03)(390.3,136.03) -- (390.3,144.03)(424.3,136.04) -- (424.3,144.04)(458.3,136.04) -- (458.3,144.04)(492.3,136.05) -- (492.3,144.05) ;
				\draw  [fill={rgb, 255:red, 155; green, 155; blue, 155 }  ,fill opacity=0.5 ] (440.86,118.85) .. controls (440.86,109.74) and (448.45,102.35) .. (457.82,102.35) .. controls (467.18,102.35) and (474.77,109.74) .. (474.77,118.85) .. controls (474.77,127.96) and (467.18,135.35) .. (457.82,135.35) .. controls (448.45,135.35) and (440.86,127.96) .. (440.86,118.85) -- cycle ;
				\draw  [fill={rgb, 255:red, 155; green, 155; blue, 155 }  ,fill opacity=0.5 ] (237.41,118.85) .. controls (237.41,109.74) and (245,102.35) .. (254.36,102.35) .. controls (263.73,102.35) and (271.32,109.74) .. (271.32,118.85) .. controls (271.32,127.96) and (263.73,135.35) .. (254.36,135.35) .. controls (245,135.35) and (237.41,127.96) .. (237.41,118.85) -- cycle ;
				\draw [color={rgb, 255:red, 208; green, 2; blue, 27 }  ,draw opacity=1 ][line width=1.5]    (322.7,140.02) -- (355.8,140.02) ;
				
				\draw (243,147.5) node [anchor=north west][inner sep=0.75pt]  [font=\scriptsize,color={rgb, 255:red, 208; green, 2; blue, 27 }  ,opacity=1 ] [align=left] {$\displaystyle -x_{0}$};
				\draw (454.17,147.5) node [anchor=north west][inner sep=0.75pt]  [font=\scriptsize,color={rgb, 255:red, 208; green, 2; blue, 27 }  ,opacity=1 ] [align=left] {$\displaystyle x_{1}$};
				\draw (419.67,147.5) node [anchor=north west][inner sep=0.75pt]  [font=\scriptsize] [align=left] {$\displaystyle 3$};
				\draw (386.17,147.5) node [anchor=north west][inner sep=0.75pt]  [font=\scriptsize] [align=left] {$\displaystyle 2$};
				\draw (351.67,147.5) node [anchor=north west][inner sep=0.75pt]  [font=\scriptsize] [align=left] {$\displaystyle 1$};
				\draw (318.17,147.5) node [anchor=north west][inner sep=0.75pt]  [font=\scriptsize] [align=left] {$\displaystyle 0$};
				\draw (276.5,147.5) node [anchor=north west][inner sep=0.75pt]  [font=\scriptsize] [align=left] {$\displaystyle -1$};
				\draw (485.5,115) node [anchor=north west][inner sep=0.75pt]  [font=\small] [align=left] {$\displaystyle ...$};
				\draw (213.5,115) node [anchor=north west][inner sep=0.75pt]  [font=\small] [align=left] {$\displaystyle ...$};

			\end{tikzpicture}
			\caption{Configuration belonging to $\Omega_N^{2,4}$.}
		\end{figure}
		It is simple to see that if $ \eta_0,\eta_1\in\Omega_N^{x_0,x_1}$ then $ \br^{(k)}(\overline\eta_0)=\br^{(k)}(\overline\eta_1) $ for all $ k \geq {1}$. Therefore we obtain $ c_N^{(m-1)}(\eta_0)=c_N^{(m-1)}(\eta_1) $, and for every $ \eta\in\Omega_N^{x_0,x_1} $ one can plot $ c_N^{(m)}(\eta) $ as a function of $m$, as in Figure \ref{fig:1}. To that end, for each $ m,\ell_N,x_0 $ and $ x_1 $ fixed and $ \xi\in\Omega_N^{x_0,x_1} $ we introduce $ \tilde{c}_N(x_0,x_1,m)\equiv c_N^{(m-1)}(\xi) $ .	
					\begin{figure}[H]
							\begin{table}[H]
									\begin{tabular}{lll}
											\subcaptionbox*{$ m=0.25 $}{\includegraphics[width=0.3\textwidth]{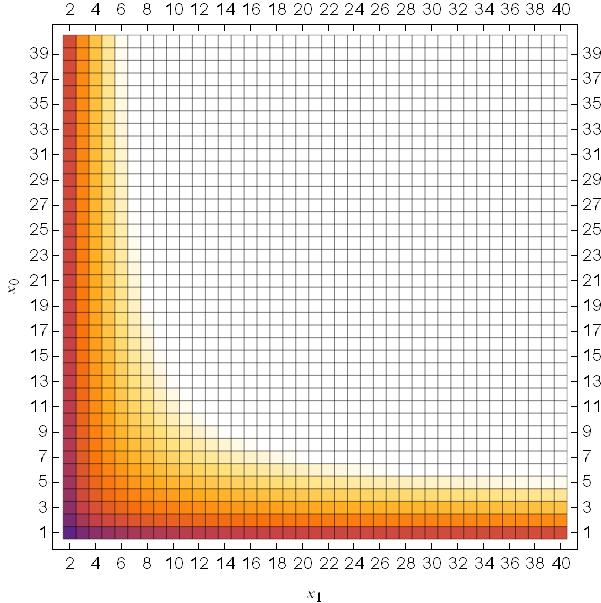}}
											&
											\subcaptionbox*{$ m=0.5 $}{\includegraphics[width=0.3\textwidth]{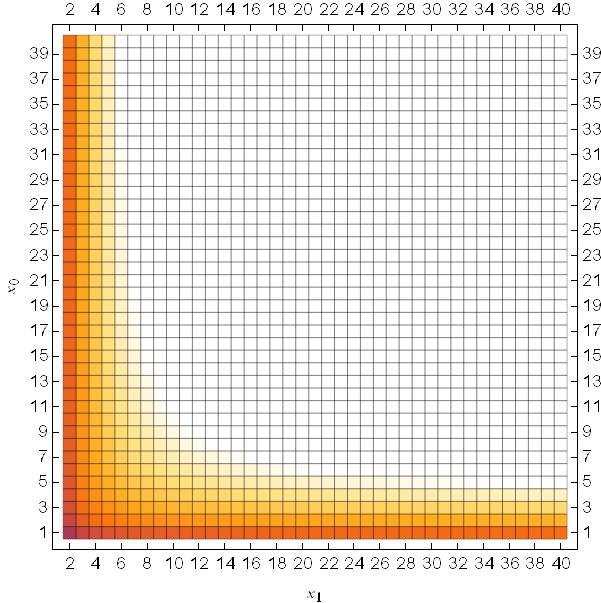}}
											&
											\subcaptionbox*{$ m=0.75 $}{\includegraphics[width=0.3\textwidth]{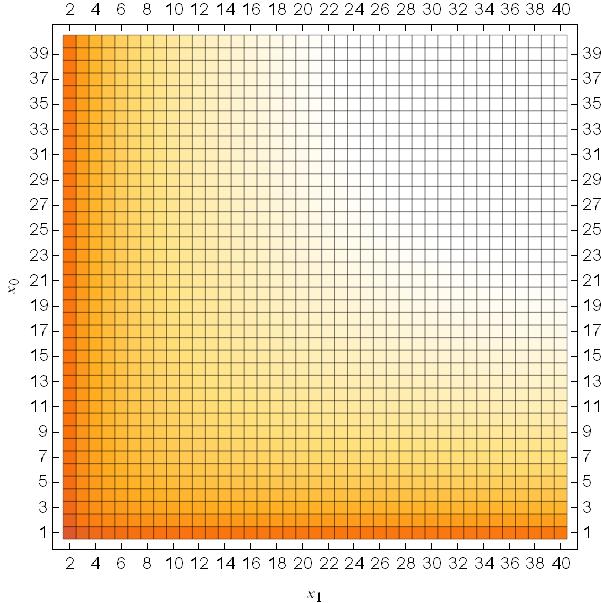}}
											\\
											\subcaptionbox*{$ m=1 $}{\includegraphics[width=0.3\textwidth]{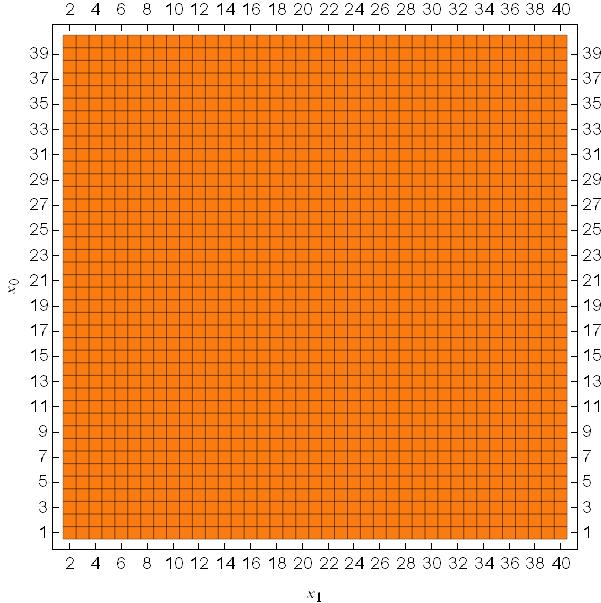}}
											&
											\subcaptionbox*{$ m=1.25 $}{\includegraphics[width=0.3\textwidth]{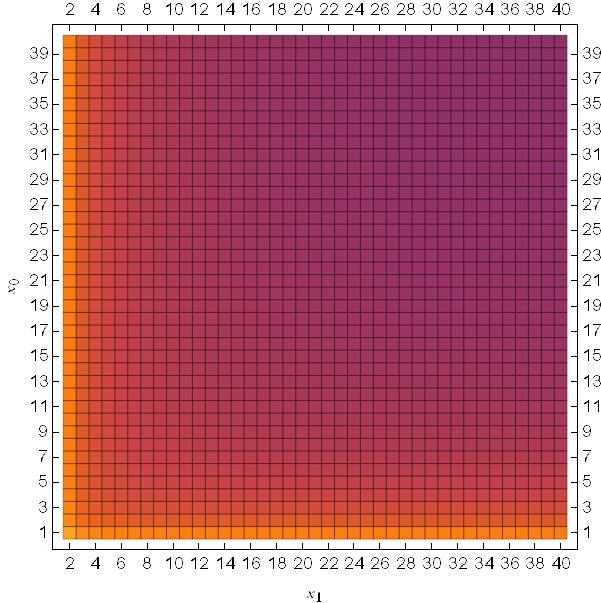}}
											&
											\subcaptionbox*{$ m=1.5 $}{\includegraphics[width=0.3\textwidth]{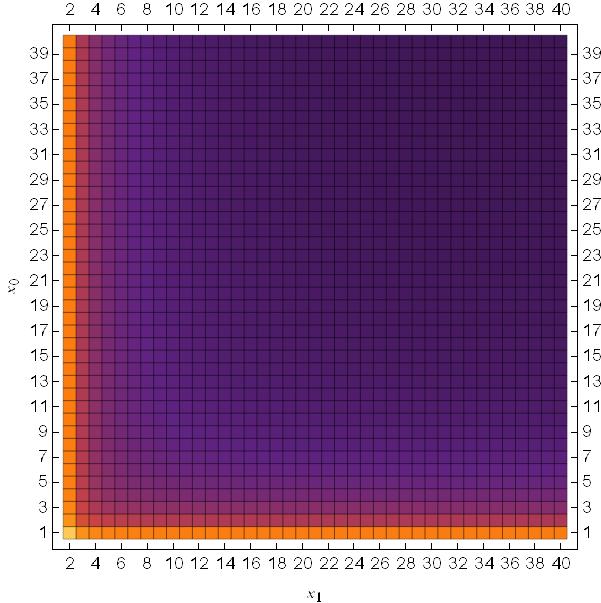}}
											\\
											\subcaptionbox*{$ m=1.75 $}{\includegraphics[width=0.3\textwidth]{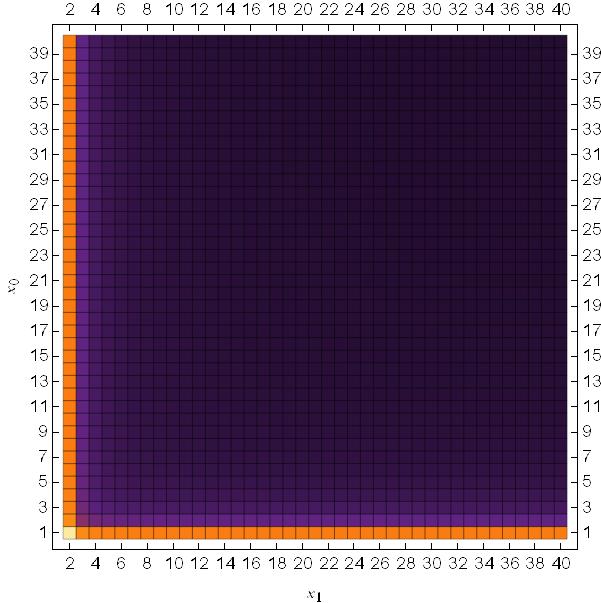}}
											&
											\subcaptionbox*{$ m=2 $}{\includegraphics[width=0.3\textwidth]{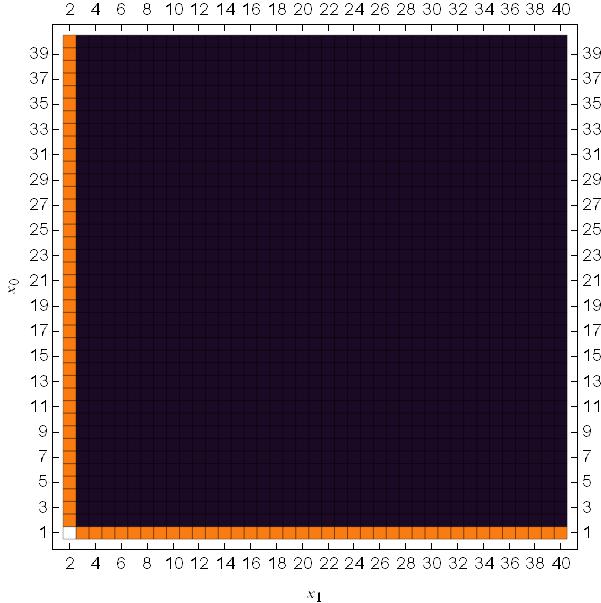}}
											&
											\qquad\includegraphics[width=0.04\textwidth]{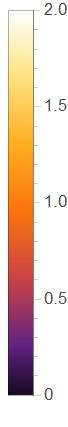}
										\end{tabular}
									\caption[ ]
									{\small Evolution of $ \tilde{c}_N(x_0,x_1,m) $ for $ \ell_N=40 $}\label{fig:1}
								\end{table}
						\end{figure}
		We stress that the previous figure presents the value of the constraint $ c_N^{(m-1)}(\eta) $ (equivalently, the rate $ r_N^{(m-1)}(\eta) $ when $ \eta(0)+\eta(1)=1 $) fixed $ x_0,x_1 $ and a representative $ \eta\in\Omega_N^{x_0,x_1} $, that is, a configuration with the first particle to the \textit{left} of the site $ 0 $ located at the site $ -x_0 $, and the first particle to the \textit{right} of the site $ 1 $ located at the site $ x_1 $. Note the symmetry of the plots with respect to $ x_0=x_1 $, which is consequence of the symmetry of the jumps. Fixed $ m $ and $ \ell_N $, varying $ x_0 $ and $ x_1 $ allow us to see all the possible values of the constraints. For example, for $ m=1 $ the rate is equal to $ 1 $ independently of $ x_0,x_1 $, hence the sub-figure, in this case, has the same colour for all $ x_0,x_1\leq 40 $. For $ m=2 $ the rate is non-zero if and only if there is at least one particle located at the site $ -1 $ or at the site $ 2 $. In other words, $ \mathbf{c}^{(1)}(\eta)=\eta(-1)+\eta(2) $. Therefore, we obtain in the respective sub-figure the horizontal and vertical orange lines, where $ \mathbf{c}^{(1)}(\eta)=1 $ for any $ \eta\in \Omega_N^{1,x_1}\cup\Omega_N^{x_0,2} $ with $ x_0\geq2 $ and $ x_1\geq3 $ and $ \mathbf{c}^{(1)}=0 $ otherwise; while at $ x_0=1 $ and $ x_1=2 $ the constraint attains its largest value, \textit{i.e.,} $ \mathbf{c}^{(1)}(\eta)=2 $ for all $ \eta\in\Omega_N^{1,2} $. In the \textit{fast-diffusion} regime, we see a "continuous" increase of the rates as $ x_0,x_1 $ increase, while the opposite in the \textit{slow-diffusion} regime. This is a clear consequence of the penalization/reinforcement terms, as seen in \eqref{PMM_rewrite}.

		\subsection{Main result}
		
		To expose our main result about the hydrodynamic limit of the interpolating model we first introduce some definitions. Let us fix a finite time horizon $[0,T]$, let $\mu_N$ be an initial probability measure on $\Omega_N$, and let $\{\eta_{N^2t}\}_{t\geq 0}$ be the Markov process generated by $N^2\mathcal{L}_N^{(m-1)}$ for $ m\in(0,2)\backslash\{1\} $, given in \eqref{PMM:m1}.
		
		\begin{Def}[Empirical measure]
			For any $ \eta\in\Omega_N $ define the empirical measure $ \pi^N(\eta,\mathrm{d}u) $ on the continuous torus $\mathbb{T}$ by
			\begin{align*}
				\pi^N(\eta,\mathrm{d}u)=\frac1N \sum_{x\in\mathbb{T}_N}\eta(x)\delta_{x/N}(\mathrm{d}u)
			\end{align*}
			where $ \delta_{x/N} $ is the Dirac measure at the macroscopic point $ x/N $. Moreover, we define its time evolution {in the diffusive time scale} by $ \pi^N_t(\eta,\rmd u)=\pi^N(\eta_{N^2t},\rmd u) $. For any function $ G:\mathbb{T}\to\mathbb{R} $,  we define the integral of $ G $ with respect to the empirical measure as
			\begin{align}\label{int:emp}
				\inner{\pi_t^N,G}
				=\int_{\mathbb{T}}G(u)\pi_t^N(\eta,\rmd u)
				=\frac{1}{N}\sum_{x\in\mathbb{T}_N}G(\tfrac{x}{N})\eta_{N^2t}(x).
			\end{align}
		\end{Def}
		Let $ \mathcal{M}_+ $ be the space of positive measures on $ [0,1] $ with total mass no larger than $ 1 $ and endowed with the weak topology.
		Let $ \mathcal{D}([0,T],\Omega_N) $ be the Skorokhod space of trajectories induced by $ \{\eta_{N^2t}\}_{t\in[0,T]} $ with initial measure $ \mu_N $. Denote by $ \mathbb{P}_{\mu_N} $ the induced probability measure on the space of trajectories $ \mathcal{D}([0,T],\Omega_N) $ and by $ \mathbb{Q}_N=\mathbb{P}_{\mu_N}\circ(\pi^N)^{-1} $ the probability measure on $ \mathcal{D}([0,T],\mathcal{M}_+) $ induced by $ \{\pi^N_t\}_{t\in[0,T]} $ and $ \mu_N $.
		
		\medskip
		
		Now we introduce the notion of weak solutions to  equation \eqref{PDE:formal} for $ m\in(0,2)$. For that purpose, for $ n\in\mathbb{N}_+\cup\{\infty\} $ let $ C^n(\mathbb{T}) $ be the set of $ n $ times continuously differentiable, real-valued functions defined on $ \mathbb{T} $; and let $ C^{n,p}([0,T]\times \mathbb{T}) $ be the set of all real-valued functions defined on $ [0,T]\times \mathbb{T} $ that are $ n $ times differentiable on the first variable and $ p $ times differentiable on the second variable and with continuous derivatives. Finally, for two functions $f,g\in L^2(\T)$, $\langle f,g\rangle$ denotes their standard euclidean product in $L^2(\T)$ and $\|\cdot\|_{L^2(\T)}$ is the associated norm. We remark that we use the notation $ \inner{\cdot,\cdot} $ {twice}, for the inner-product just introduced, and {also} in \eqref{int:emp}, although their difference will be clear from the context.

		\begin{Def}[Sobolev space] \label{def:sob}
			The semi inner-product $ \inner{\cdot,\cdot}_1 $ on the set $ C^\infty(\mathbb{T}) $ is given on $ G,H\in C^\infty(\mathbb{T}) $ by $ \inner{G,H}_1=\inner{\partial_uG,\partial_u H}=\int_\mathbb{T}\partial_uG(u)\partial_u H(u) \rmd u, $ and the associated semi-norm is denoted by $ \norm{\cdot}_1 $. Let $ \mathcal{H}^1(\mathbb{T}) $ be the Sobolev space on $ \mathbb{T} $, defined as the completion of $ C^\infty(\mathbb{T}) $ for the norm $ \norm{\cdot}_{\mathcal{H}^1(\mathbb{T})}^2=\norm{\cdot}^2_{L^2}+\norm{\cdot}_{1}^2 $, and let $ L^2([0,T];\mathcal{H}^1(\mathbb{T})) $ be the set of measurable functions $ f:[0,T]\to\mathcal{H}^1(\mathbb{T}) $ such that $ \int_0^T\norm{f_s}^2_{\mathcal{H}^1(\mathbb{T})}\rmd s<\infty $.
		\end{Def}

		\begin{Def}[Weak solutions to \eqref{PDE:formal}]\label{def:weak}
			Let $ \rho^{\rm ini}:\mathbb{T}\to[0,1] $ be a measurable function. We say that $ \rho:[0,T]\times \mathbb{T}\mapsto [0,1] $ is a weak solution of the FDE (resp. PME) with $ m\in(0,1) $ (resp. $ m\in(1,2) $) if
			\begin{enumerate}
				\item
				\begin{enumerate}
					\item For $ m\in(0,1) $ it holds $ \rho\in L^2([0,T];\mathcal{H}^1(\mathbb{T})) $,
					\item For $ m\in(1,2) $ it holds $ \rho^m\in L^2([0,T];\mathcal{H}^1(\mathbb{T})) $.
				\end{enumerate}
				\smallskip
				\item For any $ t\in[0,T] $ and $ G\in C^{1,2}([0,T]\times \mathbb{T}) $ it holds that
				\begin{equation}\label{weak}
					F(\rho^{\rm ini},\rho,G,t)
					:=
					\inner{\rho_t,G_t}-\inner{\rho^{\rm ini},G_0}
					-\int_0^t
					\big\{
					\inner{\rho_s,\partial_sG_s}
					+\inner{(\rho_s)^m,\partial_{uu} G_s}
					\big\}
					\rmd s \equiv 0.
				\end{equation}
			\end{enumerate}
		\end{Def}
		
		In the appendix, Lemmas \ref{lem:uniq_FDE} and \ref{lem:uniq_PME}, we will show that the weak solution given {by} last definition is unique, for $ m\in(0,1) $ and $ m\in(1,2) $, respectively.

		\begin{Def}[Local equilibrium distribution]\label{def:ass}
			Let $ \{\mu_N\}_{N\geq 1} $ be a sequence of probability measures on $ \Omega_N $, and let $f:\mathbb{T}\to[0,1] $ be a measurable function. If for any continuous function $ G:\mathbb{T}\to\mathbb{R} $ and every $ \delta>0 $ it holds
			\begin{align*}
				\lim_{N\to+\infty}\mu_N
				\left(
				\eta\in\Omega_N
				:\abs{\inner{\pi^N,G}-\inner{f,G}}>\delta
				\right)
				=0,
			\end{align*}
			we say that the sequence $ \{\mu_N\}_{N\geq 1} $ is a \emph{local equilibrium measure} associated to the profile $ f $.
		\end{Def}
		
		\begin{Ex} An example of a measure satisfying Definition \ref{def:ass} is the product Bernoulli measure, given on $x\in\T_N$ by
			\[ \nu_{\rho^{\rm ini}(\cdot)}^N(\eta\in\Omega_N:\;\eta(x)=1) = \rho^{\rm ini}\big(\tfrac x N\big), \] where $ \rho^{\rm ini}:\mathbb{T}\to[0,1] $ is a measurable Lipschitz profile. Then $\nu^N_{\rho^{\rm ini}(\cdot)}$ is a local equilibrium measure associated to $\rho^{\rm ini}$. For more details see, for instance, the proof of \cite[Theorem 2.2]{BMNS}.
		\end{Ex}
		
		{We are now ready to state the main result of this paper:}
		
		\begin{Th}[Hydrodynamic limit] \label{theo:hydro}
			Let $ \rho^{\rm ini}:\T\to[0,1] $ be a measurable function and let $ \{\mu_N\}_{N\geq 1} $ be a local equilibrium measure associated to it.
			Then, for any $ t\in [0,T] $, $ \delta>0 $ and $\mathbb{N}\ni\ell_N \to \infty$ such that $\ell_N\leq N $,  it holds
			\begin{align*}
				\lim_{N\to+\infty}
				\mathbb{P}_{\mu_N}
				\left(
				\abs{\inner{\pi_t^N,G}-\inner{\rho_t,G}}>\delta
				\right)=0,
			\end{align*}
			where $ \rho $ is the unique weak solution of \eqref{PDE:formal} in the sense of Definition \ref{def:weak}, with initial condition $\rho^{\rm ini}$.
		\end{Th}

		\section{Proof of Theorem \ref{theo:hydro}}
		\label{sec:HL}
		
		We first outline the proof. As previously mentioned, to prove the hydrodynamic  limit we use the classical entropy method introduced in \cite{GPV}. The general scheme is the following:  we prove that  the sequence of  empirical measures is tight (as proved in Subsection \ref{subsec:tight}), which implies the existence of weakly convergent subsequences; and then we prove that the limiting measure is concentrated on paths of absolutely continuous measures with respect to the Lebesgue measure, whose density is a weak solution to the hydrodynamic equation \eqref{PDE:formal} (proved in Section \ref{subsec:char}). To do so we shall need an energy estimate (Section \ref{sec:energy}), which gives us some regularity of the solution to the PDE, and  replacement lemmas (Section \ref{sec:replace}) whose role is to close the equations for {the limiting profile}  at the microscopic level. Proving uniqueness of weak solutions (see Appendix \ref{app:PDE}), we see that the limit of the sequence of measures is then unique and we can conclude that the whole sequence converges to that limit.

		We introduce some discrete operators that will be important in what follows. Let us extend Definition \ref{def:translation} to functions defined on $\T_N$ (instead of $\Omega_N)$.
		Without loss of generality, we adopt the same notation. Namely, if $f:\T_N\to\R$ then its gradients are $ {\nabla}^+f=(\tau_1-\mathbf{1})f $ and $ {\nabla}^-f=(\mathbf{1}-\tau_{-1})f $, where $ \mathbf{1} $ is now the identity function defined on $ \mathbb{T}_N $. 
		Finally, for any $ N \in \mathbb{N}_+ $, we also define the \emph{rescaled gradients} on $ \mathbb{T}_N $ as
		$		{\nabla}^{\pm,N}=N{\nabla}^\pm,
		$
		and the \emph{rescaled Laplacian} as
		$		\Delta^N={\nabla}^{+,N}\circ {\nabla}^{-,N}
		={\nabla}^{-,N}\circ {\nabla}^{+,N}.$
		
		\subsection{Tightness}\label{subsec:tight}
		
		Let us start by exploiting the gradient property of our model. Recall that we consider the evolution in the diffusive  time scale $ tN^2$, that is, given by the generator
		$
		\mathcal{L}:=N^2\mathcal{L}_N^{(m-1)}.
		$
		From Dynkin's formula \cite[Appendix 1, Lemma 5.1]{KL:book}, we know that for any  $ G\in C^{1,2}([0,T]\times\mathbb{T}) $
		\begin{align}\label{dynk_0}
			M_t^N(G)
			:=\inner{\pi_t^N,G_t}-\inner{\pi_0^N,G_0}
			&-\int_0^t
			(\partial_s+\mathcal L)	\inner{\pi_s^N, G_s}\rmd s
		\end{align}
		is a martingale with respect to the natural filtration of the process.
		Observe that Lemma \ref{lem:grad} and a summation by parts imply that
		\begin{align}\label{gen-inner}
			\mathcal{L}\inner{\pi^N_s,G_s}
			&=\frac{1}{N}\sum_{x\in\mathbb{T}_N}\Delta^NG_s(\tfrac{x}{N})\sum_{k=1}^{\ell_N}\binom{m}{k}(-1)^k\mathbf{h}^{(k-1)}_s(\tau_x\overline{\eta}),
		\end{align}
		where we defined for any $ k\in\mathbb{N} $ and any $ s\in[0,t] $ the time evolution $ \mathbf{h}_s^{(k-1)}(\eta)=\mathbf{h}^{(k-1)}(\eta_{N^2s}) $. We highlight the flip $ \eta\mapsto \overline{\eta} $ which comes from the definition of the rates in \eqref{eq:transitionrates}. Therefore the martingale rewrites as
		\begin{align}\label{dynk}
			\begin{split}
				M_t^N(G)
				=\inner{\pi_t^N,G_t}-\inner{\pi_0^N,G_0}
				&-\int_0^t
				\inner{\pi_s^N,\partial_s G_s}\rmd s
				\\&-\int_0^t
				\frac{1}{N}\sum_{x\in\mathbb{T}_N}\Delta^NG_s(\tfrac{x}{N})\sum_{k=1}^{\ell_N}\binom{m}{k}(-1)^k\mathbf{h}_s^{(k-1)}(\tau_x\overline{\eta})
				\rmd s.
			\end{split}
		\end{align}
		

		\begin{Prop}[Tightness]\label{prop:tight}
			The sequence of probability measures $ (\mathbb{Q}_N)_{N\in\mathbb{N}} $ is tight with respect to the Skorokhod topology of $ \mathcal{D}\left([0,T],\mathcal{M}_+\right) $.
		\end{Prop}
		
		\begin{proof}
			To prove tightness we resort to Aldous' conditions (see, for instance, \cite[proof of Proposition 4.1]{GMO22} or, equivalently, \cite[Proposition 3.3]{BDGN} for more details). Since the occupation variable is bounded {by 1}, it is enough to show that for all $ \eps>0 $ 
			\begin{align}
				\limsup_{\gamma\to0}\limsup_{N\to+\infty}
				\mathbb{P}_{\mu_N}
				\left(
				\sup_{\abs{t-s}\leq\gamma}
				\abs{
					\inner{\pi^N_t,G}-\inner{\pi^N_s,G}
				}>\eps
				\right),
			\end{align}
			where $ G $ is a time-independent function belonging to a dense subset of $ C([0,1]) $ with respect to the uniform topology. From the fact that $ M_t^N(G) $ is a martingale (with respect to the natural filtration of the process), the previous condition can be reduced to the study of the quadratic variation of \eqref{dynk_0} and the boundedness of the generator, \textit{i.e.,} it is enough  to prove that
			\begin{align}\label{aldous0}
				\lim_{\gamma\to0}\limsup_{N\to+\infty}
				\left\{\mathbb{P}_{\mu_N}\left(
				\sup_{\abs{t-s}\leq\gamma}
				\abs{
					M^N_{t}(G)-M^N_{s}(G)
				}
				>\frac{\eps}{2}\right)
				+
				\mathbb{P}_{\mu_N}\left(
				\sup_{\abs{t-s}\leq\gamma}\abs{
					\int_s^t
					{\mathcal{L}}\inner{\pi_s^N,G}\rmd s
				}
				>\frac{\eps}{2}
				\right)
				\right\}=0.
			\end{align}
			We apply the triangular, Jensen and Doob's inequalities in the first term above, and Proposition \ref{prop:pedro} in the second term, reducing to the treatment of 
			\begin{align}\label{aldous}
				\limsup_{N\to+\infty}
				\mathbb{E}_{\mu_N}\left[
				\left(
				M^N_{T}(G)
				\right)^2
				\right]^\frac12 =0
				\quad\text{and}\quad
				\lim_{\gamma\to0}\limsup_{N\to+\infty}
				\mathbb{E}_{\mu_N}\left[
				\abs{
					\int_s^t
					{\mathcal{L}}\inner{\pi_s^N,G}\rmd s
				}
				\right]=0.
			\end{align}
			Recalling from \cite[Appendix  A, Lemma 5.1]{KL:book}  the expression for the quadratic variation of the martingale, we have that the first expectation in \eqref{aldous} equals
			\begin{align*}
				\mathbb{E}_{\mu_N}\left[
				\int_0^{T}F_s^N(G)\rmd s
				\right],
				\quad\text{where}\quad
				F_s^N(G)=N^2\left(
				\mathcal{L}_N^{(m-1)}\inner{\pi_s^N,G}^2
				-2\inner{\pi_s^N,G}\mathcal{L}_N^{(m-1)}\inner{\pi_s^N,G}
				\right).
			\end{align*}
			Since our transition rates are symmetric, we get
			\begin{align*}
				F_s^N(G)
				&=\frac{1}{N^2}\sum_{x\in\mathbb{T}_N}c_N^{(m-1)}(\tau_x\eta_{N^2s})
				\left(\eta_{N^2s}(x+1)-\eta_{N^2s}(x)\right)^2
				\left(\nabla^{+,N}G(\tfrac{x}{N})\right)^2
				\\
				&
				{			\lesssim
					\frac{1}{N^2}\norm{\partial_uG}_{L^\infty(\mathbb{T})}^{2}
					\sum_{k=1}^{\ell_N}\abs{\binom{m}{k}}\sum_{x\in\mathbb{T}_N}\mathbf{r}^{(k-1)}(\tau_x\eta_{N^2s})
				}			
				{		\lesssim\frac{1}{N}\norm{\partial_uG}_{L^\infty(\mathbb{T})}^{2},
				}
			\end{align*}
			{where we used Lemma \ref{lem:up_speed} for the last inequality.} This concludes the proof of the first condition in \eqref{aldous}. For the second, we split the proof in two cases $m\in(0,1)$ and $m\in(1,2)$.
			
			Assume first that $m \in (1,2)$. From {\eqref{expr:h2}} (or {more obviously} \eqref{expr:h1}) we  have that $ |\mathbf{h}^{(k-1)}(\eta)|\leq k $. Therefore, using the inequality \eqref{ineq:sum_coeff}, the quantity \eqref{gen-inner} can be bounded from above by
			\begin{align*}
				\frac{1}{N}\sum_{x\in\mathbb{T}_N}\abs{\Delta^NG_s(\tfrac{x}{N})}\sum_{k=1}^{\ell_N}\abs{\binom{m}{k}}k
				\lesssim
				\norm{\partial_{uu} G}_{L^1(\mathbb{T})}+\frac1N,
			\end{align*}
			which implies the second requirement in \eqref{aldous}. This finishes the proof in the case $m\in(1,2)$.
			
			For $ m\in(0,1) $ we need some extra work. Recalling that in the fast diffusion case the generator can be rewritten as in \eqref{PMM_rewrite}, we see that the second expectation in \eqref{aldous} equals
			\begin{align}\label{tight:expect_fast}
				\mathbb{E}_{\mu_N}\bigg[
				\bigg|
				\int_s^t
				\frac1N\sum_{x\in\mathbb{T}_N}\Delta_NG(\tfrac{x}{N})\sum_{k=1}^{\ell_N}\abs{\binom{m}{k}}\tau_x\mathbf{h}^{(k-1)}(\overline{\eta}_{N^2s})
				\rmd s
				\bigg|
				\bigg].
			\end{align}
			It will be fundamental to identify $ \mathbf{h}^{(k-1)} $ as a function of the constraints $ \mathbf{c}^{(k-1)} $, as in \eqref{expr:h2}. From the triangular inequality we bound the expectation \eqref{tight:expect_fast} from above by
			\begin{multline}\label{h:treat_FDM}
				\mathbb{E}_{\mu_N}\bigg[
				\bigg|
				\int_s^t
				\frac1N\sum_{x\in\mathbb{T}_N}\Delta_NG(\tfrac{x}{N})\sum_{k=1}^{\ell_N}\abs{\binom{m}{k}}
				\prod_{j=0}^{k-1}\overline{\eta}_{N^2s}(x+j)
				\rmd s
				\bigg|
				\bigg]
				\\
				+
				\frac1N\sum_{x\in\mathbb{T}_N}\abs{\Delta_NG(\tfrac{x}{N})}
				\mathbb{E}_{\mu_N}\bigg[
				\bigg|
				\int_s^t
				\sum_{k=1}^{\ell_N}
				\abs{\binom{m}{k}}
				\tau_x
				\left\{
				\sum_{i=0}^{k-2}
				(\overline{\eta}_{N^2s}(i)-\overline{\eta}_{N^2s}(i+1))\sum_{j=1}^{k-1-i}\mathbf{s}_j^{(k-1)}(\tau_i\overline{\eta}_{N^2s})
				\right\}
				\rmd s
				\bigg|
				\bigg]
			\end{multline}
			where, by convention, $ \sum_{\emptyset}\equiv 0 $. Since $ m\in(0,1) $ and the process is of exclusion type, we have
			\begin{align*}
				\sum_{k=1}^{\ell_N}\abs{\binom{m}{k}}
				\prod_{j=0}^{k-1}\overline{\eta}_{N^2s}(x+j)
				\leq
				\sum_{k=1}^{\ell_N}\abs{\binom{m}{k}}< 1
			\end{align*}
			and due to the regularity of the test function the first expectation in \eqref{h:treat_FDM} can be bounded as:
			\begin{align*}
				\mathbb{E}_{\mu_N}\bigg[
				\bigg|
				\int_s^t
				\frac1N\sum_{x\in\mathbb{T}_N}\Delta_NG(\tfrac{x}{N})\sum_{k=1}^{\ell_N}\abs{\binom{m}{k}}
				\prod_{j=0}^{k-1}\overline{\eta}_{N^2s}(x+j)
				\rmd s
				\bigg|
				\bigg]
				\lesssim
				\abs{t-s}\left(\norm{\partial_{uu}G}_{L^1(\mathbb{T})}+\frac1N\right).
			\end{align*}
			{The treatment of the second expectation in \eqref{h:treat_FDM} is more demanding. Concretely, since $ m\in(0,1) $ the tail of the series $ \sum_{k\geq1}\abs{\binom{m}{k}} $ is too heavy to either argue directly via Lemma \ref{lem:rep_FDM} or slow down the speed of explosion of $ \ell_N $ (as we shall do in a different context shortly), while maintaining $ \ell_N $ with no particular order of explosion. One then needs to invoke {the forthcoming replacement} Lemma \ref{lem:rep_FDM-tight} instead, by taking advantage of the particular expression of $ \mathbf{h}^{(k-1)} $ in \eqref{expr:h2} and bounding from above as
				\begin{align*}
					\sum_{i=1}^{k-2}
					(\overline{\eta}(i)-\overline{\eta}(i+1))\sum_{j=1}^{k-1-i}\mathbf{s}_j^{(k-1)}(\tau_i\overline{\eta})
					\leq
					\sum_{i=1}^{k}
					\abs{\overline{\eta}(i)-\overline{\eta}(i+1)}\mathbf{c}^{(k-1)}(\tau_i{\overline{\eta}})
					=\sum_{i=1}^{k}\mathbf{r}^{(k-1)}(\tau_i{\overline{\eta}}).
				\end{align*}
				{One can now} use Lemma \ref{lem:rep_FDM-tight} {for each term of the summation over $ x\in\mathbb{T}_N $, with $ \varphi_i^{(k)}(\eta)=\sum_{j=1}^{k-i}\mathbf{s}_j^{(k)}(\tau_i\eta) $,} and obtain  the {final} upper bound
				\begin{align*}
					\frac1B+\sigma B\frac{(\ell_N)^{1-m}}{N}.
				\end{align*}
				Recalling that $ \ell_N\leq N $ and $ 1-m\in(0,1) $, taking the limit $ N\to+\infty $ and then $ B\to+\infty $ we finish the proof.
			}
		\end{proof}

\subsection{Characterization of limit points}\label{subsec:char}
The goal of this subsection is to show that the limiting points of $ (\mathbb{Q}_N)_{N\in\mathbb{N}} $, which we know to {exist} as a consequence of the results of the previous section, are concentrated on trajectories of absolutely continuous measures with respect to the Lebesgue measure, whose density is a weak solution to either the FDE or the PME, depending on the value of $ m $. Showing the aforementioned {absolute} continuity  is simple since we deal with an exclusion process,  and its proof can be found (modulo small adaptations) for instance in \cite[page 57]{KL:book}. From this and the previous proposition, we know (without loss of generality) that for any $ t\in[0,T] $, the sequence $ (\pi_t^{N}(\eta,\rmd u))_{N\in\mathbb{N}} $ converges weakly with respect to $ \mathbb{Q}_N $ to an absolutely continuous measure $ \pi_{\cdot}(\rmd u)=\rho_\cdot(u)\rmd u $. In the next result we obtain information about the profile $ \rho $.
\begin{Prop}\label{prop:char}
For any limit point $ \mathbb{Q} $ of $ (\mathbb{Q}_N)_{N\in\mathbb{N}} $ it holds \[ \mathbb{Q}
\bigg(\pi\in\mathcal{D}([0,T],\mathcal{M}_+)\; : \; \begin{cases}
\text{for any } t\in[0,T],\; \pi_t(du)=\rho_t(u)du,\;  \text{where } \rho\in L^2([0,T];\mathcal{H}^1(\mathbb{T})) \\ \text{for any } t \in [0,T] \text{ and any } G\in C^{1,2}([0,T]\times\mathbb T), \; F(\rho^{\rm ini},\rho,G,t) =0 \end{cases}	\bigg)
=1, \]
where  $F(\rho^{\rm ini},\rho,G,t)$  is given in  \eqref{weak}.
\end{Prop}

Before showing Proposition \ref{prop:char}, we introduce some definitions and technical results.

\begin{Def}
For any $ x\in\mathbb{T}_N $ and $ \ell\in\mathbb{N} $ consider the following microscopic box of size $\ell$, and the empirical average over it, given by
\begin{align*}
\Lambda_x^\ell=\llbracket x,x+\ell-1\rrbracket,
\quad\text{and}\quad
\eta^\ell(x)=\frac1\ell \sum_{y\in\Lambda_x^\ell}\eta(y).
\end{align*}
Moreover, for $ \eps>0 $ and $ u,v\in\mathbb{T} $, let $ {\iota}_\eps^u(v)=\frac{1}{\eps}\mathbf{1}_{v\in[u,u+\eps)}$.
\end{Def}
\begin{Lemma}\label{lem:rho_emp-approx}
Let  $ m\in(0,2)\backslash\{1\} $ be fixed.	For any $ \eps>0 $,  for a.e.~$ u\in\mathbb{T} $ and $ s\in[0,T] $ it holds that
\begin{align*}
\abs{\rho_s(u)-\inner{\pi_s,{\iota}^u_\eps}}\lesssim \eps^{\alpha},
\quad \text{ where }
\alpha\equiv \alpha(m):=\frac12 \mathbf{1}_{\{m\in(0,1)\}}+\frac14 \mathbf{1}_{\{m\in(1,2)\}}.
\end{align*}
\end{Lemma}
\begin{proof}
This is a direct consequence of the fact that $ \pi_t(\rmd u)=\rho_t(u)\rmd u $, plus two facts: first, we have that $\rho \text{ (resp.~} \rho^m\text{)}$ {belongs to} $L^2([0,T];\mathcal{H}^1(\mathbb{T}))$ for $ m\in(0,1) $ (resp.~$ m\in(1,2) $), and this will be proved in Section \ref{sec:energy}; and second, we have the H\"{o}lder continuity of $ \rho $ (see Proposition \ref{prop:continuity_FDE} {in the  case $m \in (0,1)$} and Corollary \ref{cor:cont} {in the case $m\in (1,2)$}).
\end{proof}

\begin{Lemma}\label{lem:eps_seq}
Let  $ m\in(0,2)\backslash\{1\} $ be fixed and take $ \alpha $ as in the previous lemma.	 For any $ \eps>0 $, consider the sequence $ (\eps_k)_{k\geq1} $ defined by
\[
\eps_k=k^{-\beta}\eps, \quad \text{ for some } \beta>\frac{2-m}{\alpha}>0.
\]
Then, for any $ k\in\mathbb{N}_+ $, a.e.~$ u\in\mathbb{T} $ and $s\in[0,T]$, it holds that
\begin{align*}
\bigg|\sum_{k\geq 2}\binom{m}{k}(-1)^k(1-\rho_s(u))^{k}
-
\sum_{k\geq 2}\binom{m}{k}(-1)^k\prod_{j=0}^{k-1}
\left(1-\inner{\pi_s,{\iota}_{\eps_{k}}^{u+j\eps_{k}}}\right)
\bigg|\lesssim\eps^\alpha.
\end{align*}
\end{Lemma}
\begin{proof}
We first observe that for any $ a_0, b_0, a_1, b_1 $ we can rewrite $ a_0a_1=a_0(a_1-b_1)+(a_0-b_0)b_1+b_0b_1 $. With this rationale, summing and subtracting appropriate terms we can rewrite
\begin{align*}
(1-\rho_s(u))^k
=
\prod_{j=0}^{k-1}
\left(1-\inner{\pi_s,{\iota}_{\eps_{k}}^{u+j\eps_{k}}}\right)
+\delta_{k,s}(u),
\quad\text{with}\quad
\delta_{k,s}(u)
\leq
\sum_{i=0}^{k-1}\abs{\rho_s(u)-\inner{\pi_s,{\iota}_{\eps_{k}}^{u+i\eps_{k}}}}
\end{align*}
since for any $ u\in\mathbb{T} $ and $s\in[0,T]$ it holds that $ \rho_s(u)\leq 1 ,$ and $\inner{\pi_s,{\iota}_{\eps_{k}}^{u+i\eps_{k}}}\leq 1 $.

From Lemma \ref{lem:rho_emp-approx} and the H\"{o}lder continuity of $ \rho $ ({Proposition \ref{prop:continuity_FDE} and Corollary \ref{cor:cont}}), for any $ i\geq 1 $ we can estimate
\begin{align}\label{eps_seq:eq0}
\abs{\rho_s(u)-\inner{\pi_s,{\iota}_{\eps_{k}}^{u+i\eps_{k}}}}
\leq
\abs{
\rho_s(u+i\eps_{k})
-\inner{\pi,{\iota}_{\eps_{k}}^{u+i\eps_{k}}}
}
+\abs{
\rho_s(u+i\eps_{k})
-\rho_s(u)
}
\lesssim i\eps_k^\alpha.
\end{align}
For $ i=0 $ we resort directly to Lemma \ref{lem:rho_emp-approx}. In this way, and from the upper bound {of the binomial coefficient given} in Lemma \ref{lem:bin_bound} we have
\begin{align*}
\delta_{k,s}(u)
\lesssim
\eps_{k}^\alpha
\left(1
+
\sum_{i=1}^{k-1}
i
\right)
\lesssim
k^{2}\eps_{k}^\alpha
\quad \text{ which implies }  \quad
\sum_{k\geq 2}\abs{\binom{m}{k}}\delta_{k,s}(u)
\lesssim
\eps^\alpha\sum_{k\geq 2}
\frac{1}{k^{m-1+\alpha\beta}}.
\end{align*}
The condition on $ \beta $ given in the statement of the lemma guarantees the convergence of the series above.
\end{proof}

The largest issue now is how to handle the products of occupation variables in the martingale decomposition \eqref{dynk}. The final goal is to close the equation, relating the correlation terms with the power terms in the weak formulation \eqref{weak}. The idea behind the forthcoming approach is to replace a product of $ \rho's $ by a product of empirical averages with respect to different, non-intersecting boxes. This last requirement avoids the correlations between the occupation variables on these microscopic boxes. For the macroscopic replacements to be justified, we need information on the regularity of the weak solution.

In order to prove the Proposition \ref{prop:char} we will make use of several  replacement lemmas, whose statements and proofs will be given in Section \ref{sec:replace}. The fact that the limiting measure $\mathbb{Q}$ concentrates on absolute continuous trajectories of measures that have a density in the right Sobolev space is also provided by Proposition \ref{prop:power_in_sob}, proved in Section \ref{sec:energy}.

\begin{proof}[Proof of Proposition \ref{prop:char}.]

From Proposition \ref{prop:power_in_sob} we know that  $ \rho\in L^2([0,T];\mathcal{H}^1(\mathbb{T})) $. If  $ \mathbb{Q} $ is a  limit point of  $ (\mathbb{Q}_N)_{N\in\mathbb{N}} $ then \[ \mathbb{Q}\big(	\text{for any } t\in[0,T],\; \pi_t(du)=\rho_t(u)du,\;  \text{where } \rho\in L^2([0,T];\mathcal{H}^1(\mathbb{T}))\big)=1 .\] In the weak formulation \eqref{weak}, let us replace $ \rho^m $ by its binomial expansion as in \eqref{eq:expand}. Since we are on the torus we have $ \inner{\partial_{uu}G,1}=0 $, and therefore the binomial series starts from the second term. {Otherwise, this would lead to boundary conditions.} In this way, it is enough to show that for any $ \delta>0 $ it holds
\begin{align}\label{q:prob_equiv}
\mathbb{Q}\left(
\sup_{t\in[0,T]}
\bigg|
\inner{G_t,\rho_t}-\inner{G_0,\rho^{\rm ini}}-\int_0^t\inner{\rho_s,\partial_sG_s}\rmd s
-\int_0^t
\Big\langle
\partial_{uu} G_s,
\sum_{k\geq 1}\binom{m}{k}(-1)^k
(1-\rho_s)^{k}
\rmd s
\Big\rangle\bigg|
>\delta
\right)=0.
\end{align}
Last probability is bounded from above by
\begin{align} \notag
\mathbb{Q}\Bigg(
\sup_{t\in[0,T]}
&\bigg|
\inner{G_t,\rho_t}  -\inner{G_0,\rho_0}-\int_0^t\inner{\rho_s,\partial_sG_s}\rmd s
+m
\int_0^t
\inner{\partial_{uu}G_s,1-\rho_s}
\rmd s
\\
& \qquad\qquad \qquad \qquad
-\int_0^t
\sum_{k\geq 2}\binom{m}{k}(-1)^k
\Big\langle
\partial_{uu} G_s,
\prod_{j=0}^{k-1}
\big\langle
1-\pi_s,{\iota}_{\eps_{k}}^{\cdot+j\eps_{k}}
\big\rangle
\Big\rangle \rmd s\bigg|>\frac{\delta}{2^2}
\Bigg)\label{prob_eq0:1}
\\\label{prob_eq0:2} 
&+\mathbb{Q}\left(
\sup_{t\in[0,T]}
\bigg|
\int_0^t
\sum_{k\geq 2}\binom{m}{k}(-1)^k
\Big\langle
\partial_{uu} G_s,
(1-\rho_s)^{k}
-\prod_{j=0}^{k-1}
\big\langle
1-\pi_s,{\iota}_{\eps_{k}}^{\cdot\;+j\eps_{k}}
\big\rangle
\Big\rangle \rmd s
\bigg|>\frac{\delta}{2}\right)
\\
&+\mathbb{Q}
\left(
\big|
\inner{G_0,\rho_0-\rho^{\rm ini}}
\big|
>\frac{\delta}{2^2}
\right),\label{prob_eq0:3} 
\end{align}
with $ (\eps_k)_{k\geq0} $ as in Lemma \ref{lem:eps_seq}, with $ \beta $ {there} still to be fixed. Observe that the third probability \eqref{prob_eq0:3} is equal to zero since the initial probability measure $ \mu_N $ is a local equilibrium measure associated to the profile $ \rho^{\rm ini} $. From Markov's inequality and Lemma \ref{lem:eps_seq}, the second probability \eqref{prob_eq0:2} is no larger than $ 2\eps^\alpha/\delta $, reducing us to treat the first probability \eqref{prob_eq0:1}.
\par
We now want to apply Portmanteau's Theorem, and relate the micro and macro scales. For that purpose we need to argue that the whole function of our trajectories is continuous with respect to the Skorokhod topology, thus preserving the open sets. Although this is not the case due to the cutoff functions $\iota_\varepsilon$, one can perform approximations  of these functions by continuous functions, as in \cite{FGN} and \cite{BDGN}. Moreover, one has to be careful with the martingale \eqref{dynk} which involves a \textit{finite} sum. We first treat the truncation problem, then the continuity. Let us fix $ 1<\ell_{1/\eps}\xrightarrow[\eps\to0]{}+\infty $.
Note that for any bounded sequence $ (a_k)_{k\geq1} $ we have
\begin{align*}
\bigg|	\sum_{k\geq 2}
\binom{m}{k}(-1)^ka_k
-
\sum_{k=2}^{\ell_{1/\eps}}
\binom{m}{k}(-1)^ka_k\bigg|
\lesssim
\frac{1}{(\ell_{1/\eps})^m}.
\end{align*}
In this way, we truncate the sum in some $ \ell_{1/\eps} $ step, approximate the necessary functions by continuous functions, apply Portmanteau's Theorem and then replace back the approximated functions with a vanishing error, as $\varepsilon \to 0$. As for the continuity problem, since $ G $ is continuous, one needs only to mollify the terms involving the cut-off functions $ {\iota_\varepsilon} $. Since these terms are bounded in $ L^\infty(\mathbb{T}) $ this approximation by smooth functions converges a.e.~to the original functions (see for instance the proof of Lemma \ref{lem:uniq_FDE}). However, it will be important to perform this approximation with some care and take advantage of $ (1-\rho)^k $ being decreasing in $ k $ to mollify each product of $ k $ terms in a small neighbourhood depending on $ k $. More precisely, fix $ \eps>0 $ and $ t\in(0,T], $ and for any $ 0\leq s\leq t $ consider the map
\begin{align*}
\pi
\mapsto
\Phi_\pi^\eps(s,\cdot)
=\inner{\pi_s,{\iota}^{\cdot}_\eps}
,\qquad \pi\in\mathcal{D}([0,T],\mathcal{M}_+).
\end{align*}
Note that $ \Phi_\pi^\eps(s,\cdot) $ can be discontinuous. Let $ \varphi $ be some mollifier  and for each $\tilde{\eps}>0$,   define $ \varphi_{\tilde{\eps}}(u)={\tilde{\eps}}^{-1}\varphi({\tilde{\eps}}^{-1}u) $. One can argue that the convolution function $ \Phi_\pi^\eps\star\varphi_{\tilde{\eps}} $ is a continuous approximation of $ \Phi_\pi^\eps $ from the fact that  $ \rho $ is $ \alpha-$H\"{o}lder continuous, with $ \alpha $ as in Lemma \ref{lem:rho_emp-approx} and the convergence is uniform:
\begin{align*}
\left(\Phi_\pi^\eps(s,\cdot)\star\varphi_{\tilde{\eps}}\right)(u)
-\Phi_\pi^\eps(u)
&=\int_{\mathbb{T}}
\varphi_{\tilde{\eps}}(z)
\left(
\Phi_\pi^\eps(s,u-z)-\Phi_\pi^\eps(s,u)
\right)\rmd z,
\end{align*}
hence by  continuity of $ \rho $
\begin{align*}
\abs{\Phi_\pi^\eps(s,u-z)-\Phi_\pi^\eps(s,u)}
=\frac{1}{\eps}\abs{\int_u^{u+\eps}\rho_s(w-z)-\rho_s(w)\rmd w}
\lesssim
\frac{1}{\eps}\int_u^{u+\eps}z^\alpha \rmd w
=z^\alpha.
\end{align*}
In this way, since $ \varphi_{\tilde{\eps}} $ is normalized we have that $ \abs{\left(\Phi_\pi^\eps(s,\cdot)\star\varphi_{\tilde{\eps}}\right)(u)-\Phi_\pi^\eps(u)}\lesssim \tilde{\eps}^\alpha $, 
and we conclude that
\begin{align}\label{conv-to-moll}
\sup_{u\in\mathbb{T}}
\abs{
\left(\Phi_\pi^\eps(s,\cdot)\star\varphi_{\tilde{\eps}}\right)(u)
-\Phi_\pi^\eps(s,u)
}
\lesssim
\tilde{\eps}^\alpha.
\end{align}
In particular, consider the sequence $ (\tilde{\eps_k})_{k\geq1} $ with $ \tilde{\eps}_k=k^{-\beta}\tilde{\eps} $, similarly to the sequence $ (\eps_k)_{k\geq1} $ in Lemma \ref{lem:eps_seq},
\begin{align*}
\abs{
	\sum_{k=2}^{\ell_{1/\eps}}\binom{m}{k}(-1)^k
	\left[
	\prod_{j=0}^{k-1}
	\big(
	1-\Phi_\pi^{\eps_{k}}(s,u+j\eps_{k})
	\big)
	-\prod_{j=0}^{k-1}
	\Big(
	\big(1-(\Phi_\pi^{\eps_{k}}(s,\cdot+j\eps_{k})\star \varphi_{{\tilde{\eps_{k}}}})(u)\big)
	\Big)
	\right]
}
\\
\leq
\sum_{k=2}^{\ell_{1/\eps}}\abs{\binom{m}{k}}
\sum_{j=0}^{k-1}\big|
\Phi_\pi^{\eps_{k}}(s,u+j\eps_{k})
-(\Phi_\pi^{\eps_{k}}(s,\cdot+j\eps_{k})\star \varphi_{{\tilde{\eps_{k}}}})(u)
\big|
\end{align*}
where from \eqref{conv-to-moll} we see that the right-hand side of last display can be bounded from above by some positive constant times
\begin{align*}
\sum_{k=2}^{\ell_{1/\eps}}
\frac{\tilde{\eps_k^\alpha}}{k^m}
=\tilde{\eps}^\alpha
\sum_{k=2}^{\ell_{1/\eps}}
\frac{1}{k^{m+\alpha\beta}}
\lesssim \tilde{\eps}^\alpha,
\end{align*}
{and in the last estimate above} we used the fact that $ m+\alpha\beta>1 $, since by hypothesis we have $ m+\alpha\beta>2 $. At this point, it remains to show that the map
\begin{multline*}
\pi\mapsto
\sup_{t\in[0,T]}
\bigg|
\inner{\pi_t,G_t}-\inner{G_0,\rho_0}-\int_0^t\inner{\pi_s,\partial_sG_s}\rmd s
+m\int_0^t\inner{\partial_u^2G_s,1-\pi_s}\rmd s
\\
-\int_0^t
\sum_{k=2}^{\ell_{1/\eps}}\binom{m}{k}(-1)^k
\Big\langle
\partial_{uu} G_s,
\prod_{j=0}^{k-1}
\left(\inner{1-\pi_s,{\iota}_{\eps_{k}}^{\cdot\;+j\eps_{k}}}
\star \varphi_{{\tilde{\eps_k}}}\right)(\cdot)
\Big\rangle\rmd s\bigg|
\end{multline*}
is continuous with respect to the Skorokhod weak topology. From \cite[Proposition A.3]{FGN} it is enough to show the continuity of the map
\begin{align*}
\pi\mapsto
\sup_{t\in[0,T]}
\bigg|
\int_0^t
\sum_{k=2}^{\ell_{1/\eps}}\binom{m}{k}(-1)^k
\Big\langle
\partial_u^2 G_s,
\prod_{j=0}^{k-1}
\left(\inner{1-\pi_s,{\iota}_{\eps_{k}}^{\cdot\; +j\eps_{k}}}
\star \varphi_{{\tilde{\eps_k}}}\right)(\cdot)
\Big\rangle \rmd s\bigg|,
\end{align*}
which can be done using the definition of the Skorokhod metric and is also consequence of our definition of the sequences $ (\tilde{\eps}_k)_{k\geq1} $ and $ (\eps_k)_{k\geq1} $. Applying Portmanteau's Theorem, we are reduced to treat
\begin{align}\label{on_Q_N}
\begin{split}
\liminf_{N\to+\infty}\mathbb{Q}_N
\bigg(
\sup_{t\in[0,T]}
\bigg|
\inner{\pi_t^N,G_t}
&-\inner{G_0,\rho_0}
-\int_0^t\inner{\pi_s^N,\partial_sG_s}\rmd s
+m\int_0^t\inner{\partial_{uu}G_s,1-\pi_s^N}\rmd s
\\
&
-\int_0^t
\sum_{k=2}^{\ell_{1/\eps}}\binom{m}{k}(-1)^k
\Big\langle
\partial_{uu} G_s,
\prod_{j=0}^{k-1}
\left(
\inner{1-\pi_s^N,{\iota}_{\eps_{k}}^{\cdot\; +j\eps_{k}}}
\star \varphi_{{\tilde{\eps_k}}}
\right)(\cdot)
\Big\rangle \rmd s
\bigg|>\frac{\delta}{2^4}
\bigg).
\end{split}
\end{align}
We stress that, although for small $ \eps>0 $ we can have $ \ell_{1/\eps}>N $, for $ N $ fixed, the sum
\begin{align*}
\sum_{k=2}^{\ell_{1/\eps}}\binom{m}{k}(-1)^k
\prod_{j=0}^{k}
\left(\inner{1-\pi_s^N,{\iota}_{\eps_{k}}^{\cdot\; +j\eps_{k}}}
\star \varphi_{{\tilde{\eps_k}}}\right)(u)
\end{align*}
is indeed well-defined for any $ u\in\mathbb{T} $ and one obtains, for $ k $ large enough, repeated terms in the product above. Now we can replace back $ \big(\inner{\pi_s^N,{\iota}_{\eps_{k}}^{\cdot\;+j\eps_{k}}}\star \varphi_{{\tilde{\eps_k}}}\big)(\cdot) $ by $ \inner{\pi_s^N,{\iota}_{\eps_{k}}^{\cdot\;+j\eps_{k}}} $ with the previous rationale. Fixed $ N $, since the martingale \eqref{dynk} involves a sum up to $ \ell_N $, we compare again
\begin{align}\label{truncate2}
\abs{
\sum_{k= 2}^{\ell_N+1}\binom{m}{k}(-1)^ka_k
-\sum_{k=2}^{\ell_{1/\eps}}\binom{m}{k}(-1)^ka_k
}
\lesssim
\abs{(\ell_{1/\eps})^{-m}-(\ell_N)^{-m}}.
\end{align}
%
Summing and subtracting the appropriate terms, and recalling \eqref{dynk}, the first probability  \eqref{prob_eq0:1}, after the aforementioned replacements, is no larger than the sum of terms of order $ (\ell_{1/\eps})^{-m},\eps^\alpha,\tilde{\eps}^\alpha $ and also of order $ \abs{(\ell_{1/\eps})^{-m}-(\ell_N)^{-m}} $ plus
\begin{align}
\notag&\liminf_{N\to+\infty}\mathbb{Q}_N\left(
\sup_{t\in[0,T]}
\left|
M_t^N(G)+\sum_{k=2}^{\ell_N}\binom{m}{k}(-1)^k\int_0^t
\frac{1}{N}\sum_{x\in\mathbb{T}_N}\Delta^NG_s(\tfrac{x}{N})\mathbf{h}^{(k-1)}_s(\tau_x\overline{\eta})
\rmd s
\right.\right.\\
\notag&\qquad\qquad\qquad\qquad\qquad\;\left.\left.
+\sum_{k=2}^{\ell_N}\binom{m}{k}(-1)^k\int_0^t
\Big\langle
\partial_{uu}G_s,
\prod_{j=0}^{k-1}
\inner{1-\pi_s^N,{\iota}_{\eps_{k}}^{\cdot+j\eps_{k}}}
\Big\rangle \rmd s\right|>\frac{\delta}{2^6}
\right)\\
\notag&\leq
\mathbb{P}_{\mu_N}
\left(
\sup_{t\in[0,T]}
\abs{
\sum_{k=2}^{\ell_N}\binom{m}{k}(-1)^k
\int_0^t
\Big\langle
\partial_{uu}G_s-\Delta^NG_s,
\prod_{j=0}^{k-1}
\inner{1-\pi_s^N,{\iota}_{\eps_{k}}^{\cdot+j\eps_{k}}}
\Big\rangle \rmd s
}>\frac{\delta}{3\times2^6}
\right)
\\
\notag&\quad+
\mathbb{P}_{\mu_N}
\left(
\sup_{t\in[0,T]}
\left|
\sum_{k=2}^{\ell_N}\binom{m}{k}(-1)^k
\int_0^t
\frac1N \sum_{x\in\mathbb{T}_N}
\Delta^N G_s(\tfrac{x}{N})
\bigg[
\prod_{j=0}^{k-1}
\inner{1-\pi_s^N,{\iota}_{\eps_{k}}^{\frac{x}{N}+j\eps_{k}}}
-\mathbf{h}_s^{(k-1)}(\tau_x\overline{\eta})
\bigg]\rmd s
\right|>\frac{\delta}{3\times2^6}
\right)
\\
&
\quad+\mathbb{P}_{\mu_N}
\left(
\sup_{t\in[0,T]}
\abs{
M_t^N(G)
}>\frac{\delta}{3\times2^6}
\right)
.
\end{align}
Note that the linear term $ \inner{\partial_{uu}G_s,1-\pi_s^N} $ in \eqref{on_Q_N} was absorbed into the martingale $ M_t^N(G) $, and so the challenge is to treat the non-linear terms. The first probability on the right-hand side above vanishes as $ N\to+\infty $ since $ G_s\in C^2(\mathbb{T}) $ for all $ s\in[0,t] $; the second probability is treated using the replacement lemmas with a scheme that we present shortly; the third with Doob's inequality and the proof of the first condition in \eqref{aldous}. Let us give more details for the second one. {Recall the second expression of $ \mathbf{h}^{(k)} $ from Lemma \ref{lem:grad}. We split the second probability on the right-hand side of last display into
\begin{align}\label{eq:h_prob}
	&\mathbb{P}_{\mu_N}
	\left(
	\sup_{t\in[0,T]}
	\left|
	\sum_{k=2}^{\ell_N}\binom{m}{k}(-1)^k
	\int_0^t
	\frac1N \sum_{x\in\mathbb{T}_N}
	\Delta^N G_s(\tfrac{x}{N})
	\times\right.\right.
	\\\notag&
	\qquad\qquad\qquad\left.\left.
	\times\tau_x
	\left\{
	\sum_{i=0}^{k-2}
	(\eta_{N^2s}(i)-\eta_{N^2s}(i+1))\sum_{j=1}^{k-1-i}\mathbf{s}_j^{(k-1)}(\tau_i\overline{\eta}_{N^2s})
	\right\}
	\rmd s
	\right|>\frac{\delta}{3\times2^7}
	\right)
	\\
	\notag&+\mathbb{P}_{\mu_N}
	\left(
	\sup_{t\in[0,T]}
	\left|
	\sum_{k=2}^{\ell_N}\binom{m}{k}(-1)^k
	\int_0^t
	\frac1N \sum_{x\in\mathbb{T}_N}
	\Delta^N G_s(\tfrac{x}{N})
	\bigg[
	\prod_{j=0}^{k-1}
	\inner{1-\pi_s^N,{\iota}_{\eps_{k}}^{\frac{x}{N}+j\eps_{k}}}
	-\prod_{i=0}^{k-1}\overline{\eta}_{N^2s}(x+i)
	\bigg]\rmd s
	\right|>\frac{\delta}{3\times2^7}
	\right).
\end{align}
Focus on the first probability in the previous display. We apply Proposition \ref{prop:pedro} and triangle's inequality and then pass the summation over $ x $ to outside the expectation. For $ m\in(0,1) $, since the summation starts at $ k=2 $, the resulting quantity is treated using both  Lemma \ref{lem:rep_FDM-tight} and Lemma \ref{lem:rep_FDM} for each term of the summation over $ x $ with, for each $ x $ fixed and $ i\in\{0,\dots,k-2\} $, 
\begin{align*}
	\varphi_i^{(k-1)}(s,\eta)=\Delta^N G_s(\tfrac{x}{N})\sum_{j=1}^{k-1-i}\mathbf{s}_j^{(k-1)}(\tau_{i+x}\overline{\eta})
	\leq 
	\norm{\Delta^N G}_{L^\infty([0,T]\times\mathbb{T}_N)}
	\mathbf{c}^{(k-1)}(\tau_{i+x}\overline{\eta}_{N^2s}),
\end{align*}
 estimating it by
\begin{align*}
\frac1B+TB\frac{(\ell_N)^{1-m}}{N}, 
\end{align*}
for any $ B>0 $, which will be taken to infinity after $ N\to+\infty $.

For $ m\in(1,2) $ we could either prove an analogue of Lemma \ref{lem:rep_FDM-tight} for the slow regime, or take advantage of the tail of the sum of the binomial coefficients being just light enough, in this regime, to slow down the explosion of $ \ell_N $, avoiding further restrictions. We present the second alternative. {Let 
\begin{align}\label{PMM:n0}
0<n<\frac{2-m}{5-m}.
\end{align}}
Since
\begin{align*}
\sum_{i=1}^{k-2}
(\eta_{N^2s}(i)-\eta_{N^2s}(i+1))\sum_{j=1}^{k-1-i}\mathbf{s}_j^{(k-1)}(\tau_i\eta_{N^2s})
\leq k
\end{align*}
we can estimate
\begin{align*}
\sum_{k=(\ell_N)^n+1}^{\ell_N}&\binom{m}{k}(-1)^k
\frac1N \sum_{x\in\mathbb{T}_N}
\Delta^N G_s(\tfrac{x}{N})
\sum_{i=1}^{k-2}
\sum_{j=1}^{k-1-i}
\tau_x
\left\{
(\eta_{N^2s}(i)-\eta_{N^2s}(i+1))\mathbf{s}_j^{(k-1)}(\tau_i\overline{\eta}_{N^2s})
\right\}
\\
&	\leq
\frac1N \sum_{x\in\mathbb{T}_N}
\abs{\Delta^N G_s(\tfrac{x}{N})}
\sum_{k=(\ell_N)^n}^{\ell_N}\abs{\binom{m}{k}}
k
\\
&	\lesssim
\frac1N \sum_{x\in\mathbb{T}_N}
\abs{\Delta^N G_s(\tfrac{x}{N})}
\left(\frac{1}{(\ell_N)^{n(m-1)}}-\frac{1}{(\ell_N)^{m-1}}\right),
\end{align*}
which vanishes by taking the limit $ N\to+\infty $. This means that we can replace the summation up to $ \ell_N $ by a summation up to $ (\ell_N)^n $. In this way,
from Proposition \ref{prop:pedro}, the previous truncation at $ (\ell_N)^n $ and triangle's inequalities, we are then reduced to treating
\begin{align*}
\sum_{k=2}^{(\ell_N)^n}\abs{\binom{m}{k}}\sum_{i=1}^{k-2}
\sum_{j=1}^{k-1-i}
\frac1N \sum_{x\in\mathbb{T}_N}
\mathbb{E}_{\mu_N}
\left[
\abs{
\int_0^t\Delta^N G_s(\tfrac{x}{N})
\tau_x
\left\{
(\eta_{N^2s}(i)-\eta_{N^2s}(i+1))\mathbf{s}_j^{(k-1)}(\tau_i\overline{\eta}_{N^2s})
\right\}
\rmd s
}
\right].
\end{align*}
Applying the replacement Lemma \ref{lem:rep_shift} to each term of the sum over $ j $ with $ \varphi(s,\eta)=\Delta^N G_s(\tfrac{x}{N})\mathbf{s}_j^{(k-1)}(\tau_{i+x}\overline{\eta}) $ 
we obtain an upper bound of the order of
\begin{align*}
\sum_{k=2}^{(\ell_N)^n}\abs{\binom{m}{k}}k^2
\left(
\frac{1}{B_k}
+
TB_k\frac{(\ell_N)^{m-1}}{N}
\right).
\end{align*}
{
Let $ B_k=kB>0 $. Then last display is bounded from above by some constant times
\begin{align*}
\frac1B \sum_{k=2}^{(\ell_N)^n}
\frac{1}{k^{m}}
+
TB\frac{(\ell_N)^{m-1}}{N}
\sum_{k=2}^{(\ell_N)^n}\frac{1}{k^{m-2}} 
\lesssim \frac1B
+
TB\left(\frac{(\ell_N)^{m-1+n(3-m)}}{N}
+\frac{(\ell_N)^{m-1}}{N}\right),
\end{align*}
and the right-hand side converges to zero as $ N\to+\infty $ and $ B\to+\infty $ since by the definition of $ n $ in \eqref{PMM:n0} we have $ m-1+n(3-m)<1 $.
}

Now the main goal is to estimate {for $ m\in(0,2)\backslash\{1\} $} the quantity
\begin{align}\label{rep1:eq1}
&\sum_{k=2}^{\ell_N}\abs{\binom{m}{k}}\;
\mathbb{E}_{\mu_N}
\left[
\bigg|
\int_0^t
\frac1N \sum_{x\in\mathbb{T}_N}
\Delta^N G_s(\tfrac{x}{N})
\bigg(\prod_{j=0}^{k-1}
\left(1-\inner{\pi_s^N,{\iota}_{\eps_{k}}^{\frac{x}{N}+j\eps_{k}}}\right)
-
\prod_{i=0}^{k-1}\overline{\eta}_{N^2s}(x+i)\bigg)
\rmd s		\bigg|
\right]
\end{align}
where, again, we applied Proposition \ref{prop:pedro}. It will be important to slow down the explosion $ \ell_N\to+\infty $ for $ m\in(0,1) $ too before applying repeatedly the replacement lemmas. Consider the sequence $ (a_k)_{k\geq1} $ with $ a_k\equiv a_k(t,G,\eta) $ defined by
\begin{align*}
a_k
=\bigg|
\int_0^t
\frac1N \sum_{x\in\mathbb{T}_N}
\Delta^N G_s(\tfrac{x}{N})
\bigg(\prod_{j=0}^{k-1}
\left(1-\inner{\pi_s^N,{\iota}_{\eps_{k}}^{\frac{x}{N}+j\eps_{k}}}\right)
-
\prod_{i=0}^{k-1}\overline{\eta}_{N^2s}(x+i)\bigg)
\rmd s		\bigg|.
\end{align*}
From the triangle inequality and the fact that $ G_s\in C^2(\mathbb{T}) $ it is simple to see that the sequence $ (a_k)_k $ is uniformly bounded by $ \int_0^t N^{-1}\sum_{x\in\T_N}\abs{\Delta^NG_s(\tfrac{x}{N})}\rmd s\xrightarrow[N\to+\infty]{}\norm{\partial_{uu}G}_{L^1([0,T]\times\mathbb{T})}<\infty $. In particular, 
\begin{align*}
\abs{
\sum_{k=2}^{\ell_N}\abs{\binom{m}{k}}a_k
-\sum_{k=2}^{(\ell_N)^n}\abs{\binom{m}{k}}a_k
}
=\sum_{k=(\ell_N)^n+1}^{\ell_N}\abs{\binom{m}{k}}a_k
\lesssim (\ell_N)^{-nm}-(\ell_N)^{-m}\xrightarrow[N\to+\infty]{}0.
\end{align*}
In this way, the treatment of \eqref{rep1:eq1} gives place to the treatment of
\begin{align}\label{rep1:eq0}
&\sum_{k=2}^{(\ell_N)^n}\abs{\binom{m}{k}}
\frac1N \sum_{x\in\mathbb{T}_N}\;
\mathbb{E}_{\mu_N}
\left[
\bigg|
\int_0^t
\Delta^N G_s(\tfrac{x}{N})
\bigg(\prod_{j=0}^{k-1}
\left(1-\inner{\pi_s^N,{\iota}_{\eps_{k}}^{\frac{x}{N}+j\eps_{k}}}\right)
-
\prod_{i=0}^{k-1}\overline{\eta}_{N^2s}(x+i)\bigg)
\rmd s		\bigg|
\right].
\end{align}}
To treat \eqref{rep1:eq0} we now split into the slow and fast diffusion cases. In what follows, we fix $ \beta=4 $ in Lemma \ref{lem:eps_seq}, considering thus the sequence $ (\eps_k)_{k\geq1} $ with $ \eps_k=k^{-4}\eps $. 

\medskip

\textsc{$\bullet$ Slow-diffusion,} $ m\in(1,2) $: 		We can follow a slightly simplified version of the scheme in \cite{BDGN}. Consider a {non-increasing} sequence $ (L_k)_{k\geq 1}\subseteq \mathbb{N} $ having in mind that for each $ k,N\in\mathbb{N} $ we have $ L_k\equiv L_k(N) $. We will fix this sequence shortly. In what follows, we define $ \prod_{\emptyset}=1 $. The forthcoming lemmas will be applied to each term of the summation over $ x\in\mathbb{T}_N $.
\begin{enumerate}\addtocounter{enumi}{0}
\item \textit{Rearrangements}: rewrite
\end{enumerate}
\begin{align}\label{step1}
	\prod_{j=0}^{k-1}\overline{\eta}(jL_{k})
	-\prod_{j=0}^{k-1}\overline{\eta}(j)
	=\left(
	\eta(iL_{k})-\eta(i)
	\right)
	\tilde{\varphi}_i^{(1)}(\eta)
\end{align}
where for every $ i\in\{1,\dots,k-1\} $ we defined $ \tilde{\varphi}_i^{(1)}(\eta)=\prod_{j=0}^{i-1}\overline{\eta}(j)\prod_{j=i+1}^{k-1}\overline{\eta}(jL_{k}) $. The random variable $\varphi_i^{(1)}(s,\eta)\equiv \Delta^NG_s(\tfrac{x}{N})\tilde{\varphi}_i^{(1)}(\tau_x\eta) $ is independent of the occupation variables at sites $ \llbracket i,iL_{k}\rrbracket $ and, fixed $ x $ and applying the triangle inequality we treat each term of the summation over $ i $ in \eqref{step1} with Lemma \ref{lem:rep_shift}. With the choice $ B_k=Bk^{-b_1} $ for $ k\geq 2 $, with $ B>0 $ {and $ 0<b_1<m-1 $}, we obtain an upper bound of the order of
\begin{equation}\label{step1:est}
	\begin{split}
		\sum_{k=2}^{(\ell_N)^n}\abs{\binom{m}{k}}\sum_{i=1}^{k-1}
		\left\{
		\frac{1}{B_k}+TB_kiL_{k}\frac{(\ell_N)^{m-1}}{N}
		\right\}
		&\lesssim
		\sum_{k=2}^{(\ell_N)^n}\frac{1}{k^m}
		\left\{
		\frac{1}{B_k}+\frac{(\ell_N)^{m-1}}{N}TB_kL_kk
		\right\}
		\\
		&=\frac1B\sum_{k=2}^{(\ell_N)^n}\frac{1}{k^{m-b_1}}
		+
		\frac{(\ell_N)^{m-1}}{N}TBL
		\sum_{k=2}^{(\ell_N)^n}\frac{1}{k^{m+b_1+3}}.
	\end{split}
\end{equation}
{Note that both summations converge when taking the limit $ N\to+\infty $ since $ m-b_1,m+b_1+3>1 $.} At this point, we can set $ L=(\ell_N)^{2-m} $, then recall that $ \ell_N\leq N $ and take the limits accordingly;
%
\begin{enumerate}\addtocounter{enumi}{1}
\item \textit{One-block estimates}: rewrite
\end{enumerate}
\begin{equation}\label{step2}
	\prod_{j=0}^{k-1}\overline{\eta}^{L_{k}}(jL_{k})
	-\prod_{j=0}^{k-1}\overline{\eta}(jL_{k})
	=
	\sum_{i=0}^{k-1}
	\left(
	\eta^{L_{k}}(iL_{k})-\eta(iL_{k})
	\right)
	\tilde{\varphi}_i^{(2)}(\eta)
\end{equation}
where for every $ i\in\{1,\dots,k-1\} $ we defined $ \tilde{\varphi}_i^{(2)}(\eta)=\prod_{j=0}^{i-1}\overline{\eta}^L_{k}(jL_{k})\prod_{j=i+1}^{k-1}\overline{\eta}(jL_{k}) $. The random variable $ \varphi_i^{(2)}(s,\eta)\equiv \Delta^NG_s(\tfrac{x}{N})\tilde{\varphi}_i^{(2)}(\tau_x\eta) $ is independent of the occupation variables at sites $ \llbracket iL_{k},(i+1)L_{k}-1 \rrbracket $ and, fixed $ x $ and applying the triangle inequality we treat each term of the summation over $ i $ in \eqref{step2} with Corollary \ref{cor:rep_ell_box}. We obtain an upper bound of the order of
\begin{align*}
	\sum_{k=2}^{(\ell_N)^n}
	\abs{\binom{m}{k}}
	\sum_{i=0}^{k-1}
	\left\{
	\frac{1}{B_k}
	+TB_kL_{k}\frac{(\ell_N)^{m-1}}{N}
	\right\}.
\end{align*}
This quantity is no larger than the quantity on the left-hand side of \eqref{step1:est}, therefore the same rationale used there guarantees that these errors vanish by taking the limits;
\begin{enumerate}\addtocounter{enumi}{2}
\item \textit{Two-block estimates}: rewrite
\end{enumerate}
\begin{align}\label{step3}
	\prod_{j=0}^{k-1}\overline{\eta}^{\floor{N\eps_{k}}}(j\floor{N\eps_{k}})
	-\prod_{j=0}^{k-1}\overline{\eta}^{L_{k}}(jL_{k})
	=\sum_{i=0}^{k-1}
	\left(
	\eta^{\floor{N\eps_{k}}}(i\floor{N\eps_{k}})
	-\eta^{L_{k}}(iL_{k})
	\right)
	\tilde{\varphi}_i^{(3)}(\eta)
\end{align}
where for every $ i\in\{1,k-1\} $ we defined $ \tilde{\varphi}_i^{(3)}(\eta)=\prod_{j=0}^{i-1}\overline{\eta}^{L_{k}}(jL_{k})\prod_{j=i+1}^{k-1}\overline{\eta}^{\floor{N\eps_{k}}}(j\floor{N\eps_{k}}) $. The random variable $ \varphi_i^{(3)}(s,\eta)\equiv \Delta^NG_s(\tfrac{x}{N})\tilde{\varphi}_i^{(3)}(\tau_x\eta) $ is independent of the occupation variables at sites contained in \[ \llbracket iL_{k},(i+1)\floor{N\eps_{k}}-1 \rrbracket\cup  \llbracket iL_{k},iL_{k}+\floor{N\eps_{k}}-1\rrbracket \] provided $ \floor{N\eps_{k}}\geq L_{k} $, that is, $ \floor{N\eps}\geq L. $ Fixed $ x $ and applying the triangle inequality we treat each term of the summation over $ i $ in \eqref{step3} with Lemma \ref{lem:rep_boxes}, leading to an upper bound of the order of 
\begin{multline*}
	\sum_{k=2}^{(\ell_N)^n}\abs{\binom{m}{k}}
	\sum_{i=0}^{k-1}
	\left\{\frac{1}{B_k}
	+T
	\left[
	\frac{1}{L_{k}}
	+B_k
	\left(
	\frac{L_{k}(\ell_N)^{m-1}}{N}
	+\frac{iL_{k}}{N}
	+\eps_{k}(i+1)
	\right)
	\right]\right\}
	\\
	\lesssim
	\sum_{k=2}^{(\ell_N)^n}\frac{1}{B_kk^{m}}
	+T\sum_{k=2}^{(\ell_N)^n}\frac{1}{L_kk^{m}}
	+\frac{T(\ell_N)^{m-1}}{N}\sum_{k=2}^{(\ell_N)^n}\frac{B_kL_k}{k^{m}}
	+T\sum_{k=2}^{(\ell_N)^n}\frac{B_k}{k^{m-1}}\left(\frac{L_k}{N}+\eps_k\right).
\end{multline*}
{Fix $ B_k=Bk^{-b_3}>0 $ with $ B>0 $ and $ 0<b_3<m-1 $.} We analyse each term above. From $ m-1>b_3 $ and $ m-4<1 $ we have that 
\begin{align*}
	\sum_{k=2}^{(\ell_N)^n}\frac{1}{B_kk^{m}}
	\lesssim \frac1B
	\quad\text{and}\quad
	\sum_{k=2}^{(\ell_N)^n}\frac{1}{L_kk^{m}}
	\lesssim
	\frac{1}{L}
	\ell_N^{n(5-m)}
	=(\ell_N)^{m-2+n(5-m)},
\end{align*}
respectively. From \eqref{PMM:n0} it holds $ m-2+n(5-m)<0 $. Note that this is, indeed, the constraint \eqref{PMM:n0}, and defines the largest interval $ n $ can belong to. Similarly, since $ b_3>0 $ we have
\begin{align*}
	\frac{(\ell_N)^{m-1}}{N}\sum_{k=2}^{(\ell_N)^n}\frac{B_kL_k}{k^{m}}
	\lesssim
	B\frac{\ell_N}{N}
	\quad\text{and}\quad
	\sum_{k=2}^{(\ell_N)^n}\frac{B_k}{k^{m-1}}\left(\frac{L_k}{N}+\eps_k\right)
	\lesssim
	B\left(\frac{(\ell_N)^{2-m}}{N}+\eps\right),
\end{align*}
respectively;
\begin{enumerate}\addtocounter{enumi}{3}
	\item \textit{Conclusion}: rewrite
\end{enumerate}

\begin{multline}\label{step4:1}
\prod_{j=0}^{k-1}
\left(1-\inner{\pi^N,{\iota}_{\eps_{k}}^{j\eps_{k}}}\right)
-\prod_{j=0}^{k-1}\left(1-\eta^{\floor{N\eps_{k}}}(j\floor{N\eps_{k}})\right)
\\
=\sum_{i=0}^{k-1}
\left\{
	\bigg[
	\prod_{j=1}^{i-1}1-{\eta}^{\floor{N\eps_{k}}}(j\floor{N\eps_{k}})
	\bigg]
		\left(
		{\eta}^{\floor{N\eps_{k}}}(i\floor{N\eps_{k}})-\inner{\pi^N,{\iota}_{\eps_{k}}^{i\eps_{k}}}
		\right)
	\bigg[
		\prod_{j=i+1}^{k-1}1-\inner{\pi^N,{\iota}_{\eps_{k}}^{j\eps_{k}}}
	\bigg]
\right\},
\end{multline}
and since $ \mid\inner{\pi_s^N,{\iota}_{\eps}^{\frac{x}{N}}}-\eta_{N^2s}^{\floor{N\eps}}(x)\mid\leq\floor{N\eps}^{-1} $, previous display is no larger than $ k\floor{N\eps}^{-1} $. This way, we need to bound from above
\begin{align*}
	\sum_{k=2}^{(\ell_N)^n}\abs{\binom{m}{k}}\frac{k}{\floor{N\eps_{k}}}
	\lesssim
	\frac{1}{\floor{N\eps}}\sum_{k=2}^{(\ell_N)^n}\frac{1}{k^{m-4}}
	\lesssim\frac{\ell_N^{n(5-m)}}{\floor{N\eps}}.
\end{align*}
{Note that since $ 2-m<1 $, by the definition of $ n $ in \eqref{PMM:n0} we have $ n(5-m)<1 $.}
To conclude the proof it is enough to recall that $ \ell_N\leq N $, and then take the limit $ N\to+\infty $ and $ \eps\to0 $ and then $ B\to+\infty $.


\medskip
\textsc{$\bullet$ Fast-diffusion,} $ m\in(0,1) $: 		Recall that the goal is to treat \eqref{rep1:eq0}. The strategy now is similar but simpler than for the slow diffusion case. The specific maps $ \varphi:\Omega_N\to\mathbb{R} $ in the statement of the replacement lemmas in Subsection \ref{lem:rep_FDM-tight} can be introduced analogously to the slow-diffusion case, therefore we omit their definition.
\begin{enumerate}\addtocounter{enumi}{0}
\item \textit{Rearrangements}: rewrite
\end{enumerate}
\begin{align}\label{FDMstep1}
\prod_{j=0}^{k-1}\overline{\eta}(j\floor{N\eps_{k}})
-\prod_{j=0}^{k-1}\overline{\eta}(j)
=\sum_{i=1}^{k-1}\bigg[\prod_{j=0}^{i-1}\overline{\eta}(j)\bigg]
\left(
\eta(i\floor{N\eps_{k}})-\eta(i)
\right)
\bigg[\prod_{j=i+1}^{k-1}\overline{\eta}(j\floor{N\eps_{k}})\bigg]
\end{align}
and apply Lemma \ref{lem:rep_FDM} to each term of the summation in $ i $, with the total cost of the order of
\begin{align*}
\sum_{k=1}^{(\ell_N)^n}\abs{\binom{m}{k}}
\sum_{i=1}^{k-1}
\left\{
\frac{1}{B_k}+TB_k\frac{i(\floor{N\eps_{k}}-1)}{N}
\right\}
\lesssim
\sum_{k=1}^{(\ell_N)^n}
\left\{
\frac{1}{B_kk^{m}}
+\eps T\frac{B_k}{k^{3+m}}
\right\}
,
\end{align*}
for any $ B_k>0.$ {The choice $ B_k=Bk^b_1>0 $ with $ B>0 $ and $ 1-m<b_1<m+2 $ guarantees the convergence of the series as $ N\to+\infty $.}
\begin{enumerate}\addtocounter{enumi}{1}
\item \textit{One-block and two-blocks estimates}: rewrite
\end{enumerate}
\begin{multline}\label{FDMstep2}
\prod_{j=0}^{k-1}\overline{\eta}^{\floor{N\eps_{k}}}(j\floor{N\eps_{k}})
-\prod_{j=0}^{k-1}\overline{\eta}(j\floor{N\eps_{k}})\\
=
\sum_{i=0}^{k-1}
\bigg[\prod_{j=0}^{i-1}\overline{\eta}^{\floor{N\eps_{k}}}(j\floor{N\eps_{k}})\bigg]
\left(
\eta^{\floor{N\eps_{k}}}(i\floor{N\eps_{k}})-\eta(i\floor{N\eps_{k}})
\right)
\bigg[	\prod_{j=i+1}^{k-1}\overline{\eta}(j\floor{N\eps_{k}})	\bigg]
\end{multline}
and apply Lemma \ref{lem:rep_FDM} to each term of the summation in $ i $, leading to an upper bound of the order of
\begin{align*}
\sum_{k=1}^{(\ell_N)^n}\abs{\binom{m}{k}}
\sum_{i=0}^{k-1}
\bigg\{
\frac{1}{B_k}
+TB_k\frac{1}{\floor{N\eps_{k}}}
\sum_{y\in \Lambda^{\floor{N\eps_{k}}}_{i\floor{N\eps_{k}}}}
\frac{\abs{i\floor{N\eps_{k}}-y}}{N}
\bigg\}
&\lesssim
\sum_{k=1}^{(\ell_N)^n}
\frac{1}{B_kk^{m}}
+
\eps T\sum_{k=1}^{(\ell_N)^n}
\frac{B_k}{k^{4+m}}.
\end{align*}
The choice $ B_k=Bk^b_2>0 $ with $ B>0 $ and $ 1-m<b_2<3+m $ guarantees the convergence as $ N\to+\infty $.
\begin{enumerate}\addtocounter{enumi}{2}
\item \textit{Conclusion}: rewrite
\end{enumerate}
\begin{align}
\prod_{j=0}^{k-1}
\left(1-\inner{\pi^N,{\iota}_{\eps_{k}}^{j\eps_{k}}}\right)
&-\prod_{j=0}^{k-1}\left(1-\eta^{\floor{N\eps_{k}}}(j\floor{N\eps_{k}})\right)
\notag \\
&=\sum_{i=0}^{k-1}
\bigg[\prod_{j=1}^{i-1}1-\eta^{\floor{N\eps_{k}}}(j\floor{N\eps_{k}})\bigg]
\left(
\eta^{\floor{N\eps_{k}}}(j\floor{N\eps_{k}})-\inner{\pi^N,{\iota}_{\eps_{k}}^{i\eps_{k}}}
\right)
\bigg[\prod_{j=i+1}^{k-1}1-\inner{\pi^N,{\iota}_{\eps_{k}}^{j\eps_{k}}}\bigg],\label{FDMstep3}
\end{align}
and proceed as in \eqref{step4:1}, leading to an upper bound of the order of
\begin{align*}
\sum_{k=1}^{(\ell_N)^n}\abs{\binom{m}{k}}
\frac{k}{\floor{N\eps_{k}}}
\lesssim
\frac{1}{\floor{N\eps}}
\sum_{k=1}^{(\ell_N)^n}
\frac{1}{k^{m-4}}
\lesssim
\frac{1}{\floor{N\eps}}(\ell_N)^{n(5-m)}.
\end{align*}
It is enough to fix $ n>0 $ such that $ n(5-m)<1 $.

To conclude one takes the corresponding limits as previously.
\end{proof}

\section{Replacement Lemmas}\label{sec:replace}
\subsection{Dirichlet forms}
We start with some definitions and notation.
\begin{Def}[Dirichlet Form and Carré du Champ]
For a probability measure $ \mu $ on $ \Omega_N $ and $ f:\Omega_N\to\mathbb{R} $ density with respect to $ \mu $, we define the Dirichlet form for any $ m\in(0,2] $ as
\begin{align*}
\mathcal{E}_N^{(m-1)}(f,\mu)
=\inner{f,\big(-\mathcal{L}_N^{(m-1)}\big)f}_\mu
=\int_{\Omega_N}   f(\eta)\sum_{x\in\mathbb{T}_N}\big(-r_N^{(m-1)}(\tau_x\eta)\big)(\nabla_{x,x+1}f)(\eta) \mu(\rmd\eta)
\end{align*}
and the non-negative quadratic form
\begin{align*}
\Gamma_N^{(m-1)}(f,\mu)
=\int_{\Omega_N}\sum_{x\in\mathbb{T}_N}r_N^{(m-1)}(\tau_x\eta)\left[\left(\nabla_{x,x+1}f\right)(\eta)\right]^2\mu(\rmd\eta).
\end{align*}
\end{Def}
We remark that rewriting $- a(b-a)=(a-b)^2/2+(a^2-b^2)/2 $, one obtains the identity
\begin{align}\label{id:dir-car}
\frac12\Gamma_N^{(m-1)}(\sqrt{f},\mu)
=\mathcal{E}_N^{(m-1)}(\sqrt{f},\mu)
+\frac{1}{2}\big\langle \mathcal{L}_N^{(m-1)}f\big\rangle_\mu.
\end{align}

The key observation in order to proceed similarly to \cite{BDGN} is the following proposition.

\begin{Prop}[Energy lower bound]\label{prop:energy}
Let $ \nu_\gamma^N $ be the Bernoulli product measure on $ \Omega_N $ where $ \gamma:[0,1]\to(0,1) $ {is either Lipschitz non-constant or constant,} and let $ f $ be a density with respect to $ \nu_\gamma^N $. For any $ m\in(0,2)\backslash \{1\} $ and any $ N\in\mathbb{N}_+ $ such that $ \ell_N\geq 2 $ it holds
\begin{align}\label{dir:bound}
\mathcal{E}_N^{(m-1)}(\sqrt{f},\nu_\gamma^N)
\geq
& \; \mathbf{1}_{\{m\in(1,2)\}}\frac{m}{4}
\left(
\delta_N\Gamma_N^{(0)}(\sqrt{f},\nu_\gamma^N)
+\frac{m-1}{2}\Gamma_N^{(1)}(\sqrt{f},\nu_\gamma^N)
\right)
+\mathbf{1}_{\{m\in(0,1)\}}\frac14\Gamma_N^{(0)}(\sqrt{f},\mu)
-\frac{\mathbf{c}_\gamma}{4N} \vphantom{\Bigg(}
\end{align}
where $ \mathbf{c}_\gamma>0 $ if $ \gamma $ is Lipschitz {non-constant}, or $ \mathbf{c}_\gamma=0 $ if $ \gamma $ is constant.
\end{Prop}
\begin{Rem}
	We highlight that, throughout the rest of the article, we will fix $ \nu_{\gamma}^N $, with $ \gamma(\cdot)\equiv \gamma\in(0,1) $ a constant function as the reference measure. We chose to present the previous result with $ \gamma(\cdot) $ also not constant since it is a fundamental step in order to extend the present model to the open boundary setting.
\end{Rem}

\begin{proof}[Proof of Proposition \ref{prop:energy}]
Recalling the identity \eqref{id:dir-car}, let us focus on the rightmost term there. Note that $ b-a=\sqrt{a}(\sqrt{b}-\sqrt{a})+\sqrt{b}(\sqrt{b}-\sqrt{a}) $, thus
\begin{align*}
\big\langle \mathcal{L}_N^{(m-1)}f\big\rangle_{\nu_\gamma^N}
&=\sum_{x\in\mathbb{T}_N}\int_{\eta\in\Omega_N}
r_N^{(m-1)}(\tau_x\eta)\sqrt{f}(\eta)\left(\nabla_{x,x+1}\sqrt{f}\right)(\eta)\nu^N_\gamma(\rmd\eta)\\
&\quad +\sum_{x\in\mathbb{T}_N}\int_{\eta\in\Omega_N}
r_N^{(m-1)}(\tau_x\eta)\sqrt{f}(\eta^{x,x+1})\left(\nabla_{x,x+1}\sqrt{f}\right)(\eta)\nu^N_\gamma(\rmd\eta).
\end{align*}
Performing the change of variables $ \eta\mapsto\eta^{x,x+1} $ on the second term above and using the symmetry of the rates, $ r_N^{(m-1)}(\tau_x\eta^{x,x+1})=r_N^{(m-1)}(\tau_x\eta) $, we obtain
\begin{align*}
\big\langle \mathcal{L}_N^{(m-1)}f\big\rangle_{\nu_\gamma^N}
=\sum_{x\in\mathbb{T}_N}\sum_{\eta\in\Omega_N}
r_N^{(m-1)}(\tau_x\eta)\sqrt{f}(\eta)\left(\nabla_{x,x+1}\sqrt{f}\right)(\eta)
\bigg(1-\frac{\nu_\gamma^N(\eta^{x,x+1})}{\nu_\gamma^N(\eta)}\bigg)
\nu_\gamma^N(\eta).
\end{align*}
Note that the previous quantity equals zero if $ \gamma(\cdot) $ is constant. Otherwise, applying Young's inequality with $ A>0 $,
\begin{equation*}
\sqrt{f}(\eta)\left(\nabla_{x,x+1}\sqrt{f}\right)(\eta)
\bigg(1-\frac{\nu_\gamma^N(\eta^{x,x+1})}{\nu_\gamma^N(\eta)}\bigg)
\leq
\frac{1}{2A}\abs{\left(\nabla_{x,x+1}\sqrt{f}\right)(\eta)}^2
+\frac{A}{2}
f(\eta)\bigg|1-\frac{\nu_\gamma^N(\eta^{x,x+1})}{\nu_\gamma^N(\eta)}\bigg|^2
\end{equation*}
and therefore
\begin{align*}
\big\langle \mathcal{L}_N^{(m-1)}f\big\rangle_{\nu_\gamma^N}
\leq
\frac{1}{2A}\Gamma_N^{(m-1)}(\sqrt{f},\nu_\gamma^N)
+
\frac{A}{2}\sum_{x\in\mathbb{T}_N}\sum_{\eta\in\Omega_N}
r_N^{(m-1)}(\tau_x\eta)\bigg|1-\frac{\nu_\gamma^N(\eta^{x,x+1})}{\nu_\gamma^N(\eta)}\bigg|^2f(\eta)\nu_\gamma^N(\eta),
\end{align*}
where $ \abs{1-\nu_\gamma^N(\eta^{x,x+1})/\nu_\gamma^N(\eta)}^2\leq \mathbf{c}_\gamma N^{-2} $ with $ \mathbf{c}_\gamma>0 $ for $ \gamma(\cdot) $ a Lipschitz function. From Lemma \ref{lem:up_speed} we can bound
\begin{align*}
\sum_{x\in\mathbb{T}_N}
r_N^{(m-1)}(\tau_x\eta)
\leq
\sum_{k=1}^{\ell_N}\abs{\binom{m}{k}}
\sum_{x\in\mathbb{T}_N}
\mathbf{r}^{(k-1)}(\tau_x\eta)
\leq
2\sum_{k=1}^{\ell_N}\abs{\binom{m}{k}}(N+k-1)\lesssim N.
\end{align*}
In this way, recalling that $ f $ is a density with respect to $ \nu_\gamma^N $, we obtain the upper bound
\begin{align*}
\big\langle \mathcal{L}_N^{(m-1)}f\big\rangle_{\nu_\gamma^N}
&\leq
\frac{1}{2A}\Gamma_N^{(m-1)}(\sqrt{f},\nu_\gamma^N)
+
\frac{A}{2}\frac{\mathbf{c}_\gamma}{N}.
\end{align*}
Plugging this upper bound into identity \eqref{id:dir-car} with the choice $ A=1 $ we obtain
\begin{align}\label{energyup:eq0}
\mathcal{E}_N^{(m-1)}
\geq
\frac14\Gamma_N^{(m-1)}(\sqrt{f},\mu)
-\frac{\mathbf{c}_\gamma}{4N}.
\end{align}
To finish the proof, we see that the lower bounds in Proposition \ref{prop:low_bound_r} imply that
\begin{align*}
\Gamma_N^{(m-1)}(\sqrt{f},\mu)
\geq
\mathbf{1}_{m\in(1,2)}
m\left[
\delta_N\Gamma_N^{(0)}(\sqrt{f},\mu)+\frac{m-1}{2}\Gamma_N^{(1)}(\sqrt{f},\mu)
\right]
+\mathbf{1}_{m\in(0,1)}\Gamma_N^{(0)}(\sqrt{f},\mu).
\end{align*}
\end{proof}

The next two technical results are standard but will invoked in the proof of the replacement lemmas and in their applications. As such, we present them here for future reference.
\begin{Lemma}\label{lem:change}
	Consider $ x,y\in\mathbb{T} $ and let $ \varphi:[0,T]\times\Omega_N\to\mathbb{R} $ be invariant for the map $ \eta\mapsto\eta^{x,y} $. Moreover, consider the measure $ \nu_{\gamma}^N $ with $ \gamma(\cdot)\in(0,1) $ a constant function and let $ f:\Omega_N\to\mathbb{R} $. For all $ s\in[0,T] $ it holds that
	\begin{align*}
		\int_{\Omega_N}\varphi(s,\eta)(\eta(x)-\eta(y))f(\eta)\nu_{\gamma}^N(\rmd\eta)
		=\frac12\int_{\Omega_N}\varphi(s,\eta)(\eta(y)-\eta(x))(f(\eta^{x,y})-f(\eta))\nu_{\gamma}^N(\rmd\eta)
	\end{align*}
\end{Lemma}
\begin{proof}
	Summing and subtracting the appropriate term we have 
	\begin{align*}
		\int_{\Omega_N}\varphi(s,\eta)(\eta(x)-\eta(y))f(\eta)\nu_{\gamma}^N(\rmd\eta)
		&=\frac12 \int_{\Omega_N}\varphi(s,\eta)(\eta(x)-\eta(y))(f(\eta)-f(\eta^{x,y}))\nu_{\gamma}^N(\rmd\eta)
		\\
		&+\frac12 \int_{\Omega_N}\varphi(s,\eta)(\eta(x)-\eta(y))(f(\eta)+f(\eta^{x,y}))\nu_{\gamma}^N(\rmd\eta).
	\end{align*}
	To see that the second term in the right-hand side equals zero, simply note that performing the change of variables $ \eta\mapsto\eta^{x,y} $ and using that $ \varphi(s,\eta^{x,y})=\varphi(s,\eta) $ and $ \nu_\gamma^N(\eta^{x,y})=\nu_\gamma^N(\eta) $ we obtain
	\begin{align*}
		\int_{\Omega_N}\varphi(s,\eta)(\eta(x)-\eta(y))f(\eta^{x,y})\nu_{\gamma}^N(\rmd\eta)
		&=-\int_{\Omega_N}\varphi(s,\eta)(\eta(x)-\eta(y))f(\eta)\nu_{\gamma}^N(\rmd\eta).
	\end{align*} 
\end{proof}
The next proposition is applied in the second term in \eqref{aldous0} and in \eqref{eq:h_prob}
\begin{Prop}\cite[Lemma 4.3.2]{phd:pedro}.\label{prop:pedro}
	Assume there exists a family $ \mathcal{F} $ of functions $ F_{N,\eps}:[0,T]\times \mathcal{D}([0,T],\Omega)\to\mathbb{R} $ satisfying
	\begin{align*}
		\sup_{\substack{\eps\in(0,1),N\geq1\\s\in[0,T],\eta\in\mathcal{D}([0,T],\Omega)}}
		\abs{F_{N,\eps}(s,\eta)}\leq M<\infty.
	\end{align*}
	Above, the interval for (0,1) for $ \eps $ is arbitrary. We also assume that for all $ t\in[0,T] $,
	\begin{align*}
		\limsup_{\eps\to0^+}\limsup_{N\to+\infty}
		\mathbb{E}_{\mu_N}
		\left[
		\abs{
			\int_0^tF_{N,\eps}(s,\eta_s)\rmd s
		}
		\right]=0.
	\end{align*}
	Then we have for all $ \delta>0 $,
	\begin{align*}
		\limsup_{\eps\to0^+}\limsup_{N\to+\infty}\mathbb{P}_{\mu_N}
		\left(
		\sup_{t\in[0,T]}
		\abs{
			\int_0^t F_{N,\eps}(s,\eta_s)\rmd s
		}
		>\delta
		\right)=0.
	\end{align*}
\end{Prop}

\subsection{Replacement Lemmas for \texorpdfstring{$m\in(1,2)$}.}
\begin{Lemma}\label{lem:rep_shift}
Consider $ x,y\in\mathbb{T}_N $. Let $ \varphi:[0,T]\times\Omega_N\to\mathbb{R} $ be such that $ \norm{\varphi}_{L^\infty([0,T]\times\Omega_N)}\leq c_\varphi<\infty $ and invariant for the map $ \eta\mapsto\eta^{z,z+1} $ with $ z\in\llbracket x,y-1\rrbracket $. Then for all $ B>0 $ and for all $ t\in[0,T] $
\begin{align*}
\mathbb{E}_{\mu_N}
\left[\abs{
\int_0^t
\varphi(s,\eta_{N^2s})
(
\eta_{N^2s}(x)-\eta_{N^2s}(y)
)
\rmd s}	\right]
\lesssim
\frac{1}{B}
+
TB\abs{y-x}\frac{(\ell_N)^{m-1}}{N}
.
\end{align*}
\end{Lemma}
\begin{proof}
From the entropy inequality (see \cite[Appendix 1, Chapter 8]{KL:book}) with $ \nu_\gamma^N $ as reference measure and Feynman Kac's formula (see \cite[page 14]{BMNS} for instance), we bound the previous expectation from above by
\begin{align}\label{rep0:var0}
\frac{c_\gamma}{B}
+\int_0^T
\sup_{f
}
\left\{
\abs{
\inner{
\varphi(s,\eta)
(
\eta(x)-\eta(y)
)
,f}_{\nu_\gamma^N}
}
-\frac{N}{B}\mathcal{E}_N^{(m-1)}(\sqrt{f},\nu_\gamma^N)
\right\}
\rmd s,
\end{align}
where the supremum is over densities with respect to $ \nu_\gamma^N $. Rewriting $ \eta(x)-\eta(y)=\sum_{z=x}^{y-1}\eta(z)-\eta(z+1) $, from Lemma \ref{lem:change} the first term inside the supremum  in \eqref{rep0:var0} can be rewritten as
\begin{align}\label{rep0:eq0}
\frac12\int_{\Omega_N}\sum_{z=x}^{y-1}
\varphi(s,\eta)
(\eta(z)-\eta(z+1))
(f(\eta)-f(\eta^{z,z+1}))
\nu_\gamma^N(\rmd\eta).
\end{align}
From Young's inequality we bound this from above by $ c_\varphi $ times
\begin{align*}
\frac{1}{4 A}\int_{\Omega_N} & \sum_{z=x}^{y-1}
\left(\sqrt{f}(\eta^{z,z+1})+\sqrt{f}(\eta)\right)^2
\nu_\gamma^N(\rmd\eta)
+
\frac{A}{4}\int_{\Omega_N}\sum_{z=x}^{y-1}
\left(\nabla_{z,z+1}\sqrt{f}(\eta)\right)^2
\nu_\gamma^N(\rmd\eta)
\\
& \leq
\frac{\abs{y-x}}{2A}
+
\frac{A}{2}\Gamma_N^{(0)}(\sqrt{f},\nu_\gamma^N),
\end{align*}
where we performed a change of variables on the first term. Summarizing, applying Proposition \ref{prop:energy} on \eqref{rep0:var0} we bound \eqref{rep0:var0} from above by
\begin{align*}
\frac{c_\gamma}{B}
+T
\left(
\frac12 \Gamma_N^{(0)}(\sqrt{f},\nu_\gamma^N)
\left(\frac{c_\varphi}{4}A
-\frac{N}{B}\delta_N
\right)
+c_\varphi
\frac{\abs{y-x}}{2A}
\right).
\end{align*}
Fixing $ A=4N\delta_N/c_\varphi B $ and recalling from Proposition \ref{prop:low_bound_r} that $ 0<\delta_N=\sum_{k\geq\ell_N}\abs{\binom{m-1}{k}}\lesssim (\ell_N)^{-(m-1)} $, the proof is concluded.
\end{proof}
\begin{Cor}\label{cor:rep_ell_box}
Fixed $ N $, for any $ i\in\mathbb{N}_+ $ and $ L\in\mathbb{N}_+ $ such that $ L<N $, let $ \varphi:[0,T]\times\Omega_N\to\mathbb{R} $ be {such that} $ \norm{\varphi}_{L^\infty([0,T]\times\Omega_N)}\leq c_\varphi<\infty $ and invariant for the map $ \eta\mapsto\eta^{z,z+1} $ with $ z\in\llbracket iL,(i+1)L-2\rrbracket $. 
Then, for all $ B>0 $ and for all $ t\in[0,T] $ it holds
\begin{align*}
\mathbb{E}_{\mu_N}
\left[
\abs{
\int_0^t
\varphi(s,\eta_{N^2s})
\left(\eta_{N^2s}(i L)-\eta_{N^2s}^L(iL)\right)
\rmd s		}
\right]
\lesssim
\frac1B
+TB\frac{(L+1)(\ell_N)^{m-1}}{N}
.
\end{align*}
\end{Cor}
\begin{proof}
Observing that $ \eta(0)-\eta^L(0)=\frac{1}{L}\sum_{y\in \Lambda_0^L}\left(\eta(0)-\eta(y)\right) $, from Lemma \ref{lem:rep_shift} we can bound from above the expectation in the statement of the corollary by a constant times
\begin{align*}
\frac{1}{L}\sum_{y\in \Lambda_0^L}
\frac{1}{B}
+TBy\frac{(\ell_N)^{m-1}}{N}
\lesssim
\frac1B
+TB\frac{(L+1)(\ell_N)^{m-1}}{N}.
\end{align*}
\end{proof}

Let us now state the two-blocks estimate:

\begin{Lemma}\label{lem:rep_boxes}
Fix $ \eps>0 $ and $ N\in\mathbb{N} $. For $ i\in\mathbb{N} $ and $ L<\eps N $ fixed, let $ \varphi:[0,T]\times\Omega_N\to\mathbb{R}_+ $ be such that $ \norm{\varphi}_{L^\infty([0,T]\times\Omega_N)}\leq c_\varphi<\infty $ and invariant for the map $ \eta\mapsto\eta^{z,z+1} $ with $ z\in\llbracket iL,iL+\floor{N\eps}-1\rrbracket $. Then for all $ B>0 $ and for all $ t\in[0,T] $ it holds that
\begin{align*}
\mathbb{E}_{\mu_N}
\left[
\abs{
\int_0^t
\varphi(s,\eta_{N^2s})
\left(
\eta_{N^2s}^{\floor{N\eps}}(i\floor{N\eps})
-\eta_{N^2s}^{L}(iL)
\right)
\rmd s		}
\right]
\lesssim
\frac{1}{B}
+T
\left[
\frac{1}{L}
+B
\bigg(\frac{L(\ell_N)^{m-1}}{N}
+i
\frac{L}{N}
+
\eps(i+1)
\bigg)
\right].
\end{align*}
\end{Lemma}

Before proving this lemma, let us comment on the proof: we will follow closely the path argument in \cite{BDGN}, although with some warm up before its application and some minor adjustments. Although for $ m\in(1,2) $ the state-space is irreducible, the exclusion rates are not fast enough to travel along $ \floor{N\eps} $--distances for every configuration, which would avoid the use of the path argument below (as it is the case for $ m\in(0,1) $). A simple way to see this quantitatively is to take $ \abs{y-x}=\eps N $ in Lemma \ref{lem:rep_shift}. The main reason for the resulting blow up is that the rate decreases as $ \inf\{k\in\llbracket1,\ell_N\rrbracket : \br^{(k)}(\overline{\eta})=0\} $ increases, and so for certain configurations the jumping rate can be as small as $ \delta_N\lesssim (\ell_N)^{-(m-1)} $ (see Proposition \ref{prop:low_bound_r}).

In order not to use the path argument we would need to replace the lower bound \eqref{dir:bound} by $ \kappa\Gamma_N^{(0)}(\sqrt{f},\nu_\gamma^N) $, for some constant $ \kappa>0 $ independent of $ N $, which cannot be done because there is no such constant such that $ \inf_{\eta\in\Omega_N} c_N^{(m-1)}(\eta)\geq \kappa $.
Then again, we cannot relate the function inside the expectation in the statement of the two-blocks estimate solely with $ \Gamma_N^{(1)} $, since this would require an initial shuffling of the configuration in order to move the particles with the PMM(1), hence there is the need to compare it with a SSEP term as well. In this way, we are restricted to finding some useful lower bound, such as \eqref{dir:bound}. This introduces a second issue: in \cite{BDGN}, the replacement scheme relies on the treatment of
\begin{align*}
\mathbb{E}_{\mu_N}
\left[\abs{\int_0^t
\eta_{N^2s}^L(-L)
\left(
\eta_{N^2s}^{\floor{N\eps}}(0)-\eta_{N^2s}^{L}(0)
\right)\rmd s}\right],
\end{align*}
analogously to \eqref{step3}. There, the authors start by conditioning on the number of particles in $ \eta_{N^2s}^L(-L) $, which allows them to introduce the PMM($ 1 $) rates via Young's inequality. In our case however, we have $ \overline{\eta}_{N^2s}^L(0) $, meaning that we must condition on the number of \textit{holes} instead. Controlling the holes does not allow us to introduce the PMM($ 1 $) rates, but the $ \overline{\text{PMM}}( 1 ) $ rates instead, which are incompatible with the lower bound \eqref{dir:bound}. To avoid this, one could distribute the products of empirical averages in \eqref{step3}, but doing so would necessarily lead to restrictions on the explosion rate of $ \ell_N $. The simple workaround is to replace \textit{directly} $ \eta_{N^2s}^{\floor{N\eps}}(0) $ by $ \eta_{N^2s}^{L}(0) $ with the conditioning happening inside either the $ \floor{N\eps}$ or $ L $--boxes and not outside, and at the final step of the proof invoke Proposition \ref{prop:energy}.

Let us now go into the proof.

\begin{proof}[Proof of Lemma \ref{lem:rep_boxes}]
Analogously to the previous replacement lemmas, the expectation in the statement of the lemma can be estimated by some constant times
\begin{align}\label{rep_boxes_eq0}
\frac{c_\gamma}{B}
+\int_0^t
\sup_{f
}
\left\{
\abs{
\big\langle
\varphi(s,\eta)
\left(
\eta^{\floor{N\eps}}(i\floor{N\eps})
-\eta^{L}(iL)
\right)
,f\big\rangle_{\nu_\gamma^N}
}
-\frac{N}{B}\mathcal{E}_N^{(m-1)}(\sqrt{f},\nu_\gamma^N)
\right\}
\rmd s,
\end{align}
where the supremum is over densities with respect to $ \nu_\gamma^N $ and $ \gamma(\cdot)\in(0,1) $ is a constant function. Now we break the box $ \Lambda^{\floor{N\eps}}_0 $ into $ K $ smaller $ L-$sized boxes:
\begin{align*}
\llbracket 0,\floor{N\eps}-1\rrbracket
=
\llbracket 0,KL-1\rrbracket
=\bigcup_{j=1}^{K}\llbracket (j-1)L,jL-1\rrbracket
,\qquad K=\frac{\floor{N\eps}}{L},
\end{align*}
leading to
\begin{align*}
\eta^{\floor{N\eps}}(i\floor{N\eps})
-\eta^{L}(iL)
=\frac{1}{K}\sum_{j=1}^{K}
\left(
\eta^{L}(i\floor{N\eps}+(j-1)L)
-\eta^{L}(iL)
\right).
\end{align*}
Note that we can do this only if $ \floor{N\eps}>L $, which is the case given that $ L/N<\eps $. Moreover, $ K $ might not be an integer. Nevertheless, since for any bounded function $ \psi:\Omega_N\to\mathbb{R} $ we have
\begin{align}
\frac{1}{\ceil{K}}\sum_{j=1}^{\ceil{K}}\psi(\tau_j\eta)
-
\frac{1}{\floor{K}}\sum_{j=1}^{\floor{K}}\psi(\tau_j\eta)
\lesssim
\frac{1}{\ceil{K}}+\left(1-\frac{\floor{K}}{\ceil{K}}\right)\xrightarrow{K\to+\infty}0
\end{align}
we proceed as if $ K\in\mathbb{N}_+ $.

For each $ j\in\{1,\dots,K\} $ consider the event
\begin{align*}
X_j=
\left\{
\eta\in\Omega_N :
\eta^L (iL)\geq \frac{3}{L}
\right\}
\bigcup
\left\{
\eta\in\Omega_N :
\eta^L (i\floor{N\eps}+(j-1)L)\geq \frac{3}{L}
\right\},
\end{align*}
meaning that there are at least $ 3 $ particles in at least one of the boxes
\begin{align*}
\Lambda_{iL}^L=\llbracket iL,(i+1)L-1\rrbracket
\quad\text{or}\quad
\Lambda_{i\floor{N\eps}+(j-1)L}^L=\llbracket i\floor{N\eps}+(j-1)L,i\floor{N\eps}+jL-1\rrbracket.
\end{align*}
The integral, over $ (X_j)^c $, of the first term in the variational formula \eqref{rep_boxes_eq0} is of order $ L^{-1} $, therefore we can bound from above the first term in the aforementioned variational formula by a term of order $ L^{-1} $ plus
\begin{align}\label{rep_boxes_eq1}
\frac{1}{2KL}
\sum_{j=1}^K
\sum_{z\in \Lambda^L}
\left|
\int_{X_j}
\varphi(s,\eta)
\left(\eta(z+i\floor{N\eps}+(j-1)L)-\eta(z+iL)\right)
\left(f(\eta)-f(\eta^{z+i\floor{N\eps}+(j-1)L,z+iL})\right)
\nu_\gamma^N(\rmd\eta)
\right|,
\end{align}
where we used Lemma \ref{lem:change}. To estimate the quantity in the previous display, we use a path argument in the same spirit as in \cite[Lemma 5.8]{BDGN}, we claim that we can decompose
\begin{align}\label{path}
f(\eta)-f(\eta^{z+i\floor{N\eps}+(j-1)L,z+iL})
=\sum_{n\in J^{\text{PMM}(0)}}\left(f(\eta^{(n-1)})-f(\eta^{(n)})\right)
+\sum_{n\in J^{\text{PMM}(1)}}\left(f(\eta^{(n-1)})-f(\eta^{(n)})\right)
\end{align}
where
\begin{itemize}
\item $ \eta^{(0)}=\eta,\; \eta^{(n+1)}=(\eta^{(n)})^{x(n),x(n)+1} $; \medskip
\item $ \left(x(n)\right)_{n=0,\dots,N(x_1)} $ is a sequence of moves (following the procedure to be described shortly) taking values in the set $ \{x_1,\dots,z+i\floor{N\eps}+(j-1)L\} $, with $ N(x_1) $ the number of nodes we have to exchange;\medskip
\item $ J^{\text{PMM}(0)},J^{\text{PMM}(1)} $ are the sets of indexes that count the nodes used with the PMM($ 0 $) and PMM($ 1 $) dynamics, respectively, and are such that
\begin{align*}
\big|J^{\text{PMM}(0)}\big|\leq J_0 L
\quad\text{and}\quad
\big|J^{\text{PMM}(1)}\big|\leq J_1( iL+jL+ i\floor{N\eps})
\end{align*}
for some finite constants $ J_0,J_1>0 $;
\item for each $ n\in J^{\text{PMM}(1)} $ we have $ \mathbf{c}^{(1)}(\tau_{x(n-1)}\eta^{(n-1)})>0 $.
\end{itemize}
Assuming all this, for $ i\in\{0,1\} $ and $ j\in\{1,\dots,K\} $ we have that
\begin{align}
\begin{split}\label{path1}
\sum_{n\in J^{\text{PMM}(i)}}\int_{X_j}&\abs{f(\eta^{(n-1)})-f(\eta^{(n)})}\nu_\gamma^N(\rmd\eta)
\\
&=\sum_{n\in J^{\text{PMM}(i)}}\int_{X_j}
\abs{\sqrt{f}(\eta^{(n-1)})-\sqrt{f}(\eta^{(n)})}
\abs{\sqrt{f}(\eta^{(n-1)})+\sqrt{f}(\eta^{(n)})}
\nu_\gamma^N(\rmd\eta)
\\
&\leq
\frac{A_i}{2}\sum_{n\in J^{\text{PMM}(i)}}\int_{X_j}
\mathbf{r}^{(i)}(\tau_{x(n-1)}\eta^{(n-1)})\abs{\sqrt{f}(\eta^{(n-1)})-\sqrt{f}(\eta^{(n)})}^2
\nu_\gamma^N(\rmd\eta)
\\
&\quad +\frac{1}{2A_i}\sum_{n\in J^{\text{PMM}(i)}}\int_{X_j}
\frac{1}{\mathbf{r}^{(i)}(\tau_{x(n-1)}\eta^{(n-1)})}\abs{\sqrt{f}(\eta^{(n-1)})+\sqrt{f}(\eta^{(n)})}^2
\nu_\gamma^N(\rmd\eta)
\end{split}
\end{align}
for any $ A_i>0 $. The inequality requires some justification. Fix some $ n\in J^{\text{PMM}(0)}\cup J^{\text{PMM}(1)} $ and let us write $ \xi=\eta^{(n-1)} $. Then
\begin{align*}
f(\eta^{(n-1)})-f(\eta^{(n)})
&=\mathbf{1}_{\{\eta\in\Omega_N:\;\eta_{x(n-1)}+\eta_{x(n-1)+1}=1\}}(\xi)
\left(
f(\xi)-f(\xi^{x(n-1),x(n-1)+1})
\right)
\\
&=\mathbf{r}^{(0)}(\tau_{x(n-1)}\xi)
\left(
f(\xi)-f(\xi^{x(n-1),x(n-1)+1})
\right).
\end{align*}
If $ n\in J^{\text{PMM}(0)} $ we are done. Otherwise, since $ n\in J^{\text{PMM}(1)} $,  we have $ \xi_{x(n-1)-1}+\xi_{x(n-1)+2}>0 $. Consider the set $ \Omega_x^{(1)}=\{\eta\in\Omega_N:\;\mathbf{c}^{(1)}(\tau_x\eta)>0\} $. Then $ f(\xi)=f(\xi)\mathbf{1}_{\{\Omega_{x(n-1)}^{(2)}\}}(\xi) $, and since the constraints are independent of the occupation at the sites $ x(n-1),x(n-1)+1 $ we also have that $ \xi^{x(n-1),x(n-1)+1}\in\Omega_{x(n-1)}^{(1)} $. As such,
\begin{align*}
f(\eta^{(n-1)})-f(\eta^{(n)})
=
\mathbf{r}^{(0)}(\tau_{x(n-1)}\xi)
\left(
f(\xi)\mathbf{1}_{\Omega_{x(n-1)}^{(1)}}(\xi)
-f(\xi^{x(n-1),x(n-1)+1})
\mathbf{1}_{\Omega_{x(n-1)}^{(1)}}(\xi^{x(n-1),x(n-1)+1})\right).
\end{align*}
And since the change of variables $ \xi\mapsto \xi^{x(n-1),x(n-1)+1}\in\Omega_{x(n-1)}^{(1)}$ is a bijection of $\Omega_{x(n-1)}^{(1)} $, we conclude that
\begin{align*}
f(\eta^{(n-1)})-f(\eta^{(n)})
&=
\mathbf{1}_{\Omega_{x(n-1)}^{(1)}}(\xi)\mathbf{r}^{(0)}(\tau_{x(n-1)}\xi)
\left(
f(\xi)
-f(\xi^{x(n-1),x(n-1)+1})
\right)
\end{align*}
and the rates for the PMM($ 1 $) can be introduced by using Young's inequality.

We treat the integral on the first term on the right-hand side of \eqref{path1}. Recall that $ \xi=\eta^{(n-1)} $. For $ i\in\{0,1\} $ and $ j\in\{1,\dots,K\} $,
\begin{multline*}
\sum_{\eta\in X_j}
\mathbf{r}^{(i)}(\tau_{x(n-1)}\xi)\abs{\sqrt{f}(\xi)-\sqrt{f}(\xi^{x(n-1),x(n-1)+1})}^2
\nu_\gamma^N(\eta)
\\
\leq
\sum_{\eta\in \Omega_N}
\mathbf{r}^{(i)}(\tau_{x(n-1)}\xi)\abs{\sqrt{f}(\xi)-\sqrt{f}(\xi^{x(n-1),x(n-1)+1})}^2
\nu_\gamma^N(\xi)
.
\end{multline*}
Since $ \eta\in\Omega_N\Leftrightarrow\xi=\eta^{(n-1)}\in\Omega_N $, rearranging the first summation in the previous display and relabelling the terms yields
\begin{align*}
\sum_{\eta\in \Omega_N}
\mathbf{r}^{(i)}(\tau_{x(n-1)}\eta)\abs{\sqrt{f}(\eta)-\sqrt{f}(\eta^{x(n-1),x(n-1)+1})}^2
\nu_\gamma^N(\eta).
\end{align*}
Consequently, the first term on the right-hand side of \eqref{path1} can be bounded from above by $ A_i\Gamma_N^{(i)}(\sqrt{f},\nu_\gamma^N) $, while the second can be bounded from above by
\begin{align*}
\frac{1}{A_i}\sum_{n\in J^{\text{PMM}(i)}}\int_{\Omega_N}
\left(f(\eta^{(n-1)})+f(\eta^{(n)})\right)
\nu_\gamma^N(\rmd\eta)
=\frac{2}{A_i}\abs{J^{\text{PMM}(i)}}.
\end{align*}
In this way, \eqref{rep_boxes_eq1} is no larger than
\begin{align*}
\frac12 A_0\Gamma_N^{(0)}(\sqrt{f},\nu_\gamma^N)
+\frac12 A_1\Gamma_N^{(1)}(\sqrt{f},\nu_\gamma^N)
+
J_0\frac{L}{A_0}
+J_1
\left(\frac{iL+i\floor{N\eps}}{A_1}
+\frac{KL}{A_1}
\right).
\end{align*}
Recalling Proposition \ref{prop:energy}, the quantity \eqref{rep_boxes_eq0} is overestimated by
\begin{multline*}
\frac{c_\gamma}{B}+T
\sup_f
\left\{
3\frac{c_\varphi}{L}
+J_0\frac{L}{A_0}
+J_1
\frac{iL+(i+1)\floor{N\eps}}{A_1}
\right.\\\left.
+
\frac12\Gamma_N^{(0)}(\sqrt{f},\nu_\gamma^N)
\left(
A_0-\frac{N}{B}\frac{m}{2}\delta_N
\right)
+\frac12\Gamma_N^{(1)}(\sqrt{f},\nu_\gamma^N)
\left(
A_1-\frac{N}{B}\frac{m-1}{4}
\right)
\right\},
\end{multline*}
where we recall that $ KL=\floor{N\eps} $. Setting
\begin{align*}
A_0=\delta_N\frac{N}{B}\frac{m}{2}
\quad\text{and}\quad
A_1=\frac{N}{B}\frac{m-1}{4}
\end{align*}
we obtain an upper bound of the order of 
\begin{align*}
	\frac{1}{B}
	+T
	\left[
	\frac{1}{L}
	+B
	\bigg(\frac{L(\ell_N)^{m-1}}{N}
	+i
	\frac{L}{N}
	+
	\eps(i+1)
	\bigg)
	\right].
\end{align*}
Now we prove our claim with the path argument.
The goal is to exchange the occupation variables of the sites
\begin{align*}
z_{i,\eps,L}:=z+i\floor{N\eps}+(j-1)L \quad\text{and}\quad z_{i,L}:=z+iL,
\qquad
z\in \llbracket 0\;, L-1 \rrbracket.
\end{align*}
Recall that there are at least three particles either in $ \Lambda^{L}_{iL} $ or in $ \Lambda_{i\floor{N\eps}+(j-1)L}^L $. We outline the argument only for the case of at least three particles in $ \Lambda^{L}_{iL} $ since the other one is analogous and leads to an equivalent estimate. It is sufficient to consider configurations in \eqref{rep_boxes_eq1} such that $ \eta(z_{i,\eps,L})+\eta(z_{i,L})=1 $. The decomposition \eqref{path} illustrates a path on the state-space starting from the configuration $ \eta $ and ending at $ \eta^{z_{i,\eps,L},z_{i,L}} $. Note that we can consider without loss of generalization that $ \eta(z_{i,L})=1 $, since if $ \eta(z_{i,L})=0 $ then we construct an analogous path starting from $ \eta^{z_{i,\eps,L},z_{i,L}} $ and ending at $ \eta $.

Recall that a \textit{mobile cluster} with respect to the PMM($ 1 $) is a local configuration which can be translated on the lattice by a sequence of jumps dictated by the PMM($ 1 $). For example, the smallest mobile cluster for the PMM($ 1 $) corresponds to a local configuration where $ \eta(x)+\eta(x+1)+\eta(x+2)=2 $, for some $ x\in\mathbb{T}_N $.

Since $ \eta(z+iL)=1 $, there are at least two other particles in $ \Lambda_{iL}^L $. Pick the two closest to the site $ z+iL $ and label them as $ P_1 $ and $ P_2 $. Let us also denote the particle at site $ z_{i,L} $ by $ P_{z_{i,L}} $. We use the SSEP dynamics to move $ P_1 $ and $ P_2 $ to the vicinity of $ P_{z_{i,L}} $, forming a "mobile cluster". This can be done with a number of steps of order $ L $. We arrive at one of the following three local configurations.
\begin{figure}[H]

\centering

\tikzset{every picture/.style={line width=0.75pt}} 

\tikzset{every picture/.style={line width=0.75pt}} 

\begin{tikzpicture}[x=0.75pt,y=0.75pt,yscale=-1,xscale=1]

\draw    (117.1,140.65) -- (249.9,140.65) (151.1,136.65) -- (151.1,144.65)(185.1,136.65) -- (185.1,144.65)(219.1,136.65) -- (219.1,144.65) ;
\draw  [fill={rgb, 255:red, 155; green, 155; blue, 155 }  ,fill opacity=0.5 ] (134.3,119.5) .. controls (134.3,110.39) and (141.89,103) .. (151.25,103) .. controls (160.62,103) and (168.21,110.39) .. (168.21,119.5) .. controls (168.21,128.61) and (160.62,136) .. (151.25,136) .. controls (141.89,136) and (134.3,128.61) .. (134.3,119.5) -- cycle ;
\draw    (270.1,140.65) -- (402.9,140.65) (304.1,136.65) -- (304.1,144.65)(338.1,136.65) -- (338.1,144.65)(372.1,136.65) -- (372.1,144.65) ;
\draw  [fill={rgb, 255:red, 155; green, 155; blue, 155 }  ,fill opacity=0.5 ] (287.3,119.5) .. controls (287.3,110.39) and (294.89,103) .. (304.25,103) .. controls (313.62,103) and (321.21,110.39) .. (321.21,119.5) .. controls (321.21,128.61) and (313.62,136) .. (304.25,136) .. controls (294.89,136) and (287.3,128.61) .. (287.3,119.5) -- cycle ;
\draw  [fill={rgb, 255:red, 155; green, 155; blue, 155 }  ,fill opacity=0.5 ] (321.21,119.5) .. controls (321.21,110.39) and (328.8,103) .. (338.16,103) .. controls (347.53,103) and (355.12,110.39) .. (355.12,119.5) .. controls (355.12,128.61) and (347.53,136) .. (338.16,136) .. controls (328.8,136) and (321.21,128.61) .. (321.21,119.5) -- cycle ;
\draw  [fill={rgb, 255:red, 155; green, 155; blue, 155 }  ,fill opacity=0.5 ] (355.12,119.5) .. controls (355.12,110.39) and (362.71,103) .. (372.07,103) .. controls (381.44,103) and (389.03,110.39) .. (389.03,119.5) .. controls (389.03,128.61) and (381.44,136) .. (372.07,136) .. controls (362.71,136) and (355.12,128.61) .. (355.12,119.5) -- cycle ;
\draw    (424.3,140.65) -- (557.1,140.65) (458.3,136.65) -- (458.3,144.65)(492.3,136.65) -- (492.3,144.65)(526.3,136.65) -- (526.3,144.65) ;
\draw  [fill={rgb, 255:red, 155; green, 155; blue, 155 }  ,fill opacity=0.5 ] (441.5,119.5) .. controls (441.5,110.39) and (449.09,103) .. (458.45,103) .. controls (467.82,103) and (475.41,110.39) .. (475.41,119.5) .. controls (475.41,128.61) and (467.82,136) .. (458.45,136) .. controls (449.09,136) and (441.5,128.61) .. (441.5,119.5) -- cycle ;
\draw  [fill={rgb, 255:red, 155; green, 155; blue, 155 }  ,fill opacity=0.5 ] (475.41,119.5) .. controls (475.41,110.39) and (483,103) .. (492.36,103) .. controls (501.73,103) and (509.32,110.39) .. (509.32,119.5) .. controls (509.32,128.61) and (501.73,136) .. (492.36,136) .. controls (483,136) and (475.41,128.61) .. (475.41,119.5) -- cycle ;
\draw  [fill={rgb, 255:red, 155; green, 155; blue, 155 }  ,fill opacity=0.5 ] (509.32,119.5) .. controls (509.32,110.39) and (516.91,103) .. (526.27,103) .. controls (535.64,103) and (543.23,110.39) .. (543.23,119.5) .. controls (543.23,128.61) and (535.64,136) .. (526.27,136) .. controls (516.91,136) and (509.32,128.61) .. (509.32,119.5) -- cycle ;
\draw  [fill={rgb, 255:red, 155; green, 155; blue, 155 }  ,fill opacity=0.5 ] (168.21,119.5) .. controls (168.21,110.39) and (175.8,103) .. (185.16,103) .. controls (194.53,103) and (202.12,110.39) .. (202.12,119.5) .. controls (202.12,128.61) and (194.53,136) .. (185.16,136) .. controls (175.8,136) and (168.21,128.61) .. (168.21,119.5) -- cycle ;
\draw  [fill={rgb, 255:red, 155; green, 155; blue, 155 }  ,fill opacity=0.5 ] (202.12,119.5) .. controls (202.12,110.39) and (209.71,103) .. (219.07,103) .. controls (228.44,103) and (236.03,110.39) .. (236.03,119.5) .. controls (236.03,128.61) and (228.44,136) .. (219.07,136) .. controls (209.71,136) and (202.12,128.61) .. (202.12,119.5) -- cycle ;

\draw (144.4,113.05) node [anchor=north west][inner sep=0.75pt]    {$P_{1}$};
\draw (178.4,113.05) node [anchor=north west][inner sep=0.75pt]    {$P_{2}$};
\draw (212.6,113.05) node [anchor=north west][inner sep=0.75pt]    {$P_{z_{i,L}}$};
\draw (212.14,147.39) node [anchor=north west][inner sep=0.75pt]  [font=\scriptsize]  {$z_{i,L}$};
\draw (297.4,113.05) node [anchor=north west][inner sep=0.75pt]    {$P_{1}$};
\draw (331.4,113.05) node [anchor=north west][inner sep=0.75pt]    {$P_{z_{i,L}}$};
\draw (365.6,113.05) node [anchor=north west][inner sep=0.75pt]    {$P_{2}$};
\draw (331.14,147.39) node [anchor=north west][inner sep=0.75pt]  [font=\scriptsize]  {$z_{i,L}$};
\draw (451.6,113.05) node [anchor=north west][inner sep=0.75pt]    {$P_{z_{i,L}}$};
\draw (485.6,113.05) node [anchor=north west][inner sep=0.75pt]    {$P_{1}$};
\draw (519.8,113.05) node [anchor=north west][inner sep=0.75pt]    {$P_{2}$};
\draw (451.34,147.39) node [anchor=north west][inner sep=0.75pt]  [font=\scriptsize]  {$z_{i,L}$};

\end{tikzpicture}

\end{figure}

Note that we still need an empty site in the vicinity of these three particles to construct a mobile cluster. Nevertheless, if this is not the case we can assume that they are part of a larger mobile cluster. Moreover, we can relabel the particles and use the SSEP dynamics to have the local configuration (for example) as in the first case of the previous figure.  Now we move this mobile cluster to the left of the (empty) site $ z_{i,\eps,L} $ with the PMM($ 1 $) dynamics.
\begin{figure}[H]
\centering

\tikzset{every picture/.style={line width=0.75pt}} 

\begin{tikzpicture}[x=0.75pt,y=0.75pt,yscale=-1,xscale=1]

\draw    (149.1,84.66) -- (316.53,84.7) (183.1,80.67) -- (183.1,88.67)(217.1,80.67) -- (217.1,88.67)(251.1,80.68) -- (251.1,88.68)(285.1,80.69) -- (285.1,88.69) ;
\draw  [fill={rgb, 255:red, 155; green, 155; blue, 155 }  ,fill opacity=0.5 ] (200.21,63.51) .. controls (200.21,54.39) and (207.8,47.01) .. (217.16,47.01) .. controls (226.53,47.01) and (234.12,54.39) .. (234.12,63.51) .. controls (234.12,72.62) and (226.53,80.01) .. (217.16,80.01) .. controls (207.8,80.01) and (200.21,72.62) .. (200.21,63.51) -- cycle ;
\draw  [fill={rgb, 255:red, 155; green, 155; blue, 155 }  ,fill opacity=0.5 ] (234.12,63.51) .. controls (234.12,54.39) and (241.71,47.01) .. (251.07,47.01) .. controls (260.44,47.01) and (268.03,54.39) .. (268.03,63.51) .. controls (268.03,72.62) and (260.44,80.01) .. (251.07,80.01) .. controls (241.71,80.01) and (234.12,72.62) .. (234.12,63.51) -- cycle ;
\draw  [color={rgb, 255:red, 208; green, 2; blue, 27 }  ,draw opacity=1 ] (268.03,63.51) .. controls (268.03,54.39) and (275.62,47.01) .. (284.98,47.01) .. controls (294.35,47.01) and (301.94,54.39) .. (301.94,63.51) .. controls (301.94,72.62) and (294.35,80.01) .. (284.98,80.01) .. controls (275.62,80.01) and (268.03,72.62) .. (268.03,63.51) -- cycle ;
\draw    (352.77,84.65) -- (520.2,84.69) (386.77,80.66) -- (386.77,88.66)(420.77,80.67) -- (420.77,88.67)(454.77,80.68) -- (454.77,88.68)(488.77,80.68) -- (488.77,88.68) ;
\draw  [fill={rgb, 255:red, 155; green, 155; blue, 155 }  ,fill opacity=0.5 ] (369.97,63.5) .. controls (369.97,54.39) and (377.56,47) .. (386.92,47) .. controls (396.28,47) and (403.88,54.39) .. (403.88,63.5) .. controls (403.88,72.61) and (396.28,80) .. (386.92,80) .. controls (377.56,80) and (369.97,72.61) .. (369.97,63.5) -- cycle ;
\draw  [color={rgb, 255:red, 208; green, 2; blue, 27 }  ,draw opacity=1 ] (437.78,63.5) .. controls (437.78,54.39) and (445.38,47) .. (454.74,47) .. controls (464.1,47) and (471.69,54.39) .. (471.69,63.5) .. controls (471.69,72.61) and (464.1,80) .. (454.74,80) .. controls (445.38,80) and (437.78,72.61) .. (437.78,63.5) -- cycle ;
\draw  [fill={rgb, 255:red, 155; green, 155; blue, 155 }  ,fill opacity=0.5 ] (471.69,63.5) .. controls (471.69,54.39) and (479.28,47) .. (488.65,47) .. controls (498.01,47) and (505.6,54.39) .. (505.6,63.5) .. controls (505.6,72.61) and (498.01,80) .. (488.65,80) .. controls (479.28,80) and (471.69,72.61) .. (471.69,63.5) -- cycle ;
\draw    (352.8,183.65) -- (520.23,183.69) (386.8,179.66) -- (386.8,187.66)(420.8,179.67) -- (420.8,187.67)(454.8,179.68) -- (454.8,187.68)(488.8,179.68) -- (488.8,187.68) ;
\draw  [fill={rgb, 255:red, 155; green, 155; blue, 155 }  ,fill opacity=0.5 ] (370,162.5) .. controls (370,153.39) and (377.59,146) .. (386.95,146) .. controls (396.32,146) and (403.91,153.39) .. (403.91,162.5) .. controls (403.91,171.61) and (396.32,179) .. (386.95,179) .. controls (377.59,179) and (370,171.61) .. (370,162.5) -- cycle ;
\draw  [fill={rgb, 255:red, 155; green, 155; blue, 155 }  ,fill opacity=0.5 ] (437.82,162.5) .. controls (437.82,153.39) and (445.41,146) .. (454.77,146) .. controls (464.14,146) and (471.73,153.39) .. (471.73,162.5) .. controls (471.73,171.61) and (464.14,179) .. (454.77,179) .. controls (445.41,179) and (437.82,171.61) .. (437.82,162.5) -- cycle ;
\draw  [color={rgb, 255:red, 208; green, 2; blue, 27 }  ,draw opacity=1 ] (403.91,162.5) .. controls (403.91,153.39) and (411.5,146) .. (420.86,146) .. controls (430.23,146) and (437.82,153.39) .. (437.82,162.5) .. controls (437.82,171.61) and (430.23,179) .. (420.86,179) .. controls (411.5,179) and (403.91,171.61) .. (403.91,162.5) -- cycle ;
\draw   (165.86,47.01) -- (268.47,47.01) -- (268.47,80.01) -- (165.86,80.01) -- cycle ;
\draw  [fill={rgb, 255:red, 155; green, 155; blue, 155 }  ,fill opacity=0.5 ] (166.3,63.51) .. controls (166.3,54.39) and (173.89,47.01) .. (183.25,47.01) .. controls (192.62,47.01) and (200.21,54.39) .. (200.21,63.51) .. controls (200.21,72.62) and (192.62,80.01) .. (183.25,80.01) .. controls (173.89,80.01) and (166.3,72.62) .. (166.3,63.51) -- cycle ;
\draw  [fill={rgb, 255:red, 155; green, 155; blue, 155 }  ,fill opacity=0.5 ] (403.88,63.5) .. controls (403.88,54.39) and (411.47,47) .. (420.83,47) .. controls (430.19,47) and (437.78,54.39) .. (437.78,63.5) .. controls (437.78,72.61) and (430.19,80) .. (420.83,80) .. controls (411.47,80) and (403.88,72.61) .. (403.88,63.5) -- cycle ;
\draw  [fill={rgb, 255:red, 155; green, 155; blue, 155 }  ,fill opacity=0.5 ] (471.73,162.5) .. controls (471.73,153.39) and (479.32,146) .. (488.68,146) .. controls (498.05,146) and (505.64,153.39) .. (505.64,162.5) .. controls (505.64,171.61) and (498.05,179) .. (488.68,179) .. controls (479.32,179) and (471.73,171.61) .. (471.73,162.5) -- cycle ;
\draw    (149.1,183.65) -- (316.53,183.69) (183.1,179.66) -- (183.1,187.66)(217.1,179.67) -- (217.1,187.67)(251.1,179.68) -- (251.1,187.68)(285.1,179.68) -- (285.1,187.68) ;
\draw  [fill={rgb, 255:red, 155; green, 155; blue, 155 }  ,fill opacity=0.5 ] (200.21,162.5) .. controls (200.21,153.39) and (207.8,146) .. (217.16,146) .. controls (226.53,146) and (234.12,153.39) .. (234.12,162.5) .. controls (234.12,171.61) and (226.53,179) .. (217.16,179) .. controls (207.8,179) and (200.21,171.61) .. (200.21,162.5) -- cycle ;
\draw  [fill={rgb, 255:red, 155; green, 155; blue, 155 }  ,fill opacity=0.5 ] (234.12,162.5) .. controls (234.12,153.39) and (241.71,146) .. (251.07,146) .. controls (260.44,146) and (268.03,153.39) .. (268.03,162.5) .. controls (268.03,171.61) and (260.44,179) .. (251.07,179) .. controls (241.71,179) and (234.12,171.61) .. (234.12,162.5) -- cycle ;
\draw  [color={rgb, 255:red, 208; green, 2; blue, 27 }  ,draw opacity=1 ] (166.3,162.5) .. controls (166.3,153.39) and (173.89,146) .. (183.25,146) .. controls (192.62,146) and (200.21,153.39) .. (200.21,162.5) .. controls (200.21,171.61) and (192.62,179) .. (183.25,179) .. controls (173.89,179) and (166.3,171.61) .. (166.3,162.5) -- cycle ;
\draw  [fill={rgb, 255:red, 155; green, 155; blue, 155 }  ,fill opacity=0.5 ] (268.03,162.5) .. controls (268.03,153.39) and (275.62,146) .. (284.98,146) .. controls (294.35,146) and (301.94,153.39) .. (301.94,162.5) .. controls (301.94,171.61) and (294.35,179) .. (284.98,179) .. controls (275.62,179) and (268.03,171.61) .. (268.03,162.5) -- cycle ;
\draw   (199.77,146) -- (302.38,146) -- (302.38,179) -- (199.77,179) -- cycle ;
\draw    (234.12,89) -- (234.12,139) ;
\draw [shift={(234.12,141)}, rotate = 270] [color={rgb, 255:red, 0; green, 0; blue, 0 }  ][line width=0.75]    (10.93,-3.29) .. controls (6.95,-1.4) and (3.31,-0.3) .. (0,0) .. controls (3.31,0.3) and (6.95,1.4) .. (10.93,3.29)   ;

\draw (210.4,57.06) node [anchor=north west][inner sep=0.75pt]    {$P_{2}$};
\draw (244.6,57.06) node [anchor=north west][inner sep=0.75pt]    {$P_{z_{i,L}}$};
\draw (244.07,91.91) node [anchor=north west][inner sep=0.75pt]  [font=\scriptsize]  {$z_{i,L}$};
\draw (380.07,57.05) node [anchor=north west][inner sep=0.75pt]    {$P_{1}$};
\draw (482.27,57.05) node [anchor=north west][inner sep=0.75pt]    {$P_{z_{i,L}}$};
\draw (447.81,91.39) node [anchor=north west][inner sep=0.75pt]  [font=\scriptsize]  {$z_{i,L}$};
\draw (380.1,156.05) node [anchor=north west][inner sep=0.75pt]    {$P_{1}$};
\draw (448.1,156.25) node [anchor=north west][inner sep=0.75pt]    {$P_{2}$};
\draw (446.84,190.39) node [anchor=north west][inner sep=0.75pt]  [font=\scriptsize]  {$z_{i,L}$};
\draw (327.5,57.9) node [anchor=north west][inner sep=0.75pt]    {$\mapsto $};
\draw (442,114.65) node [anchor=north west][inner sep=0.75pt]  [rotate=-89.91]  {$\mapsto $};
\draw (175.6,57.66) node [anchor=north west][inner sep=0.75pt]    {$P_{1}$};
\draw (413.8,57.26) node [anchor=north west][inner sep=0.75pt]    {$P_{2}$};
\draw (482.17,155.85) node [anchor=north west][inner sep=0.75pt]    {$P_{z_{i,L}}$};
\draw (210.07,156.38) node [anchor=north west][inner sep=0.75pt]    {$P_{1}$};
\draw (244.4,156.25) node [anchor=north west][inner sep=0.75pt]    {$P_{2}$};
\draw (244.14,190.39) node [anchor=north west][inner sep=0.75pt]  [font=\scriptsize]  {$z_{i,L}$};
\draw (278.47,155.85) node [anchor=north west][inner sep=0.75pt]    {$P_{z_{i,L}}$};
\draw (344.33,167.08) node [anchor=north west][inner sep=0.75pt]  [rotate=-180.49]  {$\mapsto $};

\end{tikzpicture}
\end{figure}

The number of steps can be crudely bounded above by a term of order $ L+(i\floor{N\eps}+(j-1)L) $.
By hypothesis, $ \eta(z_{i,\eps,L})=0 $ and so we leave $ P_{z_{i,L}} $ at site $ z_{i,\eps,L} $ using either the SSEP or the PMM($ 1 $) dynamics, and transport the hole to the site $ z_{i,L} $ with the PMM($ 1 $) dynamics.
\begin{figure}[H]
\centering

\tikzset{every picture/.style={line width=0.75pt}} 

\begin{tikzpicture}[x=0.75pt,y=0.75pt,yscale=-1,xscale=1]

\draw    (303.1,99.66) -- (470.53,99.7) (337.1,95.67) -- (337.1,103.67)(371.1,95.67) -- (371.1,103.67)(405.1,95.68) -- (405.1,103.68)(439.1,95.69) -- (439.1,103.69) ;
\draw  [fill={rgb, 255:red, 155; green, 155; blue, 155 }  ,fill opacity=0.5 ] (320.3,78.51) .. controls (320.3,69.39) and (327.89,62.01) .. (337.25,62.01) .. controls (346.62,62.01) and (354.21,69.39) .. (354.21,78.51) .. controls (354.21,87.62) and (346.62,95.01) .. (337.25,95.01) .. controls (327.89,95.01) and (320.3,87.62) .. (320.3,78.51) -- cycle ;
\draw  [fill={rgb, 255:red, 155; green, 155; blue, 155 }  ,fill opacity=0.5 ] (354.21,78.51) .. controls (354.21,69.39) and (361.8,62.01) .. (371.16,62.01) .. controls (380.53,62.01) and (388.12,69.39) .. (388.12,78.51) .. controls (388.12,87.62) and (380.53,95.01) .. (371.16,95.01) .. controls (361.8,95.01) and (354.21,87.62) .. (354.21,78.51) -- cycle ;
\draw  [fill={rgb, 255:red, 155; green, 155; blue, 155 }  ,fill opacity=0.5 ] (422.03,78.51) .. controls (422.03,69.39) and (429.62,62.01) .. (438.98,62.01) .. controls (448.35,62.01) and (455.94,69.39) .. (455.94,78.51) .. controls (455.94,87.62) and (448.35,95.01) .. (438.98,95.01) .. controls (429.62,95.01) and (422.03,87.62) .. (422.03,78.51) -- cycle ;
\draw    (99.1,99.66) -- (266.53,99.7) (133.1,95.67) -- (133.1,103.67)(167.1,95.67) -- (167.1,103.67)(201.1,95.68) -- (201.1,103.68)(235.1,95.69) -- (235.1,103.69) ;
\draw  [fill={rgb, 255:red, 155; green, 155; blue, 155 }  ,fill opacity=0.5 ] (116.3,78.51) .. controls (116.3,69.39) and (123.89,62.01) .. (133.25,62.01) .. controls (142.62,62.01) and (150.21,69.39) .. (150.21,78.51) .. controls (150.21,87.62) and (142.62,95.01) .. (133.25,95.01) .. controls (123.89,95.01) and (116.3,87.62) .. (116.3,78.51) -- cycle ;
\draw  [fill={rgb, 255:red, 155; green, 155; blue, 155 }  ,fill opacity=0.5 ] (150.21,78.51) .. controls (150.21,69.39) and (157.8,62.01) .. (167.16,62.01) .. controls (176.53,62.01) and (184.12,69.39) .. (184.12,78.51) .. controls (184.12,87.62) and (176.53,95.01) .. (167.16,95.01) .. controls (157.8,95.01) and (150.21,87.62) .. (150.21,78.51) -- cycle ;
\draw  [fill={rgb, 255:red, 155; green, 155; blue, 155 }  ,fill opacity=0.5 ] (184.12,78.51) .. controls (184.12,69.39) and (191.71,62.01) .. (201.07,62.01) .. controls (210.44,62.01) and (218.03,69.39) .. (218.03,78.51) .. controls (218.03,87.62) and (210.44,95.01) .. (201.07,95.01) .. controls (191.71,95.01) and (184.12,87.62) .. (184.12,78.51) -- cycle ;

\draw (330.4,72.06) node [anchor=north west][inner sep=0.75pt]    {$P_{1}$};
\draw (364.4,72.06) node [anchor=north west][inner sep=0.75pt]    {$P_{2}$};
\draw (432.6,72.06) node [anchor=north west][inner sep=0.75pt]    {$P_{z_{i,L}}$};
\draw (430.14,106.4) node [anchor=north west][inner sep=0.75pt]  [font=\scriptsize]  {$z_{i,\epsilon ,L}$};
\draw (126.4,72.06) node [anchor=north west][inner sep=0.75pt]    {$P_{1}$};
\draw (160.4,72.06) node [anchor=north west][inner sep=0.75pt]    {$P_{2}$};
\draw (194.6,72.06) node [anchor=north west][inner sep=0.75pt]    {$P_{z_{i,L}}$};
\draw (226.14,106.4) node [anchor=north west][inner sep=0.75pt]  [font=\scriptsize]  {$z_{i,\epsilon ,L}$};
\draw (276.69,72.54) node [anchor=north west][inner sep=0.75pt]    {$\mapsto $};

\end{tikzpicture}
\end{figure}

If the site to the left of $ P_1 $ is either empty or occupied, we can perform the following transport with either the PMM($ 1 $) or a relabelling in the last step.
\begin{figure}[H]
\centering

\tikzset{every picture/.style={line width=0.75pt}} 

\begin{tikzpicture}[x=0.75pt,y=0.75pt,yscale=-1,xscale=1]

\draw    (47.1,129.66) -- (214.53,129.7) (81.1,125.67) -- (81.1,133.67)(115.1,125.67) -- (115.1,133.67)(149.1,125.68) -- (149.1,133.68)(183.1,125.69) -- (183.1,133.69) ;
\draw  [fill={rgb, 255:red, 155; green, 155; blue, 155 }  ,fill opacity=0.5 ] (98.21,108.51) .. controls (98.21,99.39) and (105.8,92.01) .. (115.16,92.01) .. controls (124.53,92.01) and (132.12,99.39) .. (132.12,108.51) .. controls (132.12,117.62) and (124.53,125.01) .. (115.16,125.01) .. controls (105.8,125.01) and (98.21,117.62) .. (98.21,108.51) -- cycle ;
\draw  [fill={rgb, 255:red, 155; green, 155; blue, 155 }  ,fill opacity=0.5 ] (132.12,108.51) .. controls (132.12,99.39) and (139.71,92.01) .. (149.07,92.01) .. controls (158.44,92.01) and (166.03,99.39) .. (166.03,108.51) .. controls (166.03,117.62) and (158.44,125.01) .. (149.07,125.01) .. controls (139.71,125.01) and (132.12,117.62) .. (132.12,108.51) -- cycle ;
\draw    (250.77,129.65) -- (418.2,129.69) (284.77,125.66) -- (284.77,133.66)(318.77,125.67) -- (318.77,133.67)(352.77,125.68) -- (352.77,133.68)(386.77,125.68) -- (386.77,133.68) ;
\draw  [fill={rgb, 255:red, 155; green, 155; blue, 155 }  ,fill opacity=0.5 ] (267.97,108.5) .. controls (267.97,99.39) and (275.56,92) .. (284.92,92) .. controls (294.28,92) and (301.88,99.39) .. (301.88,108.5) .. controls (301.88,117.61) and (294.28,125) .. (284.92,125) .. controls (275.56,125) and (267.97,117.61) .. (267.97,108.5) -- cycle ;
\draw  [fill={rgb, 255:red, 155; green, 155; blue, 155 }  ,fill opacity=0.5 ] (335.78,108.5) .. controls (335.78,99.39) and (343.38,92) .. (352.74,92) .. controls (362.1,92) and (369.69,99.39) .. (369.69,108.5) .. controls (369.69,117.61) and (362.1,125) .. (352.74,125) .. controls (343.38,125) and (335.78,117.61) .. (335.78,108.5) -- cycle ;
\draw    (454.1,129.65) -- (621.53,129.69) (488.1,125.66) -- (488.1,133.66)(522.1,125.67) -- (522.1,133.67)(556.1,125.68) -- (556.1,133.68)(590.1,125.68) -- (590.1,133.68) ;
\draw  [fill={rgb, 255:red, 155; green, 155; blue, 155 }  ,fill opacity=0.5 ] (471.3,108.5) .. controls (471.3,99.39) and (478.89,92) .. (488.25,92) .. controls (497.62,92) and (505.21,99.39) .. (505.21,108.5) .. controls (505.21,117.61) and (497.62,125) .. (488.25,125) .. controls (478.89,125) and (471.3,117.61) .. (471.3,108.5) -- cycle ;
\draw  [fill={rgb, 255:red, 155; green, 155; blue, 155 }  ,fill opacity=0.5 ] (505.21,108.5) .. controls (505.21,99.39) and (512.8,92) .. (522.16,92) .. controls (531.53,92) and (539.12,99.39) .. (539.12,108.5) .. controls (539.12,117.61) and (531.53,125) .. (522.16,125) .. controls (512.8,125) and (505.21,117.61) .. (505.21,108.5) -- cycle ;
\draw  [color={rgb, 255:red, 208; green, 2; blue, 27 }  ,draw opacity=1 ] (64.3,108.51) .. controls (64.3,99.39) and (71.89,92.01) .. (81.25,92.01) .. controls (90.62,92.01) and (98.21,99.39) .. (98.21,108.51) .. controls (98.21,117.62) and (90.62,125.01) .. (81.25,125.01) .. controls (71.89,125.01) and (64.3,117.62) .. (64.3,108.51) -- cycle ;
\draw  [color={rgb, 255:red, 208; green, 2; blue, 27 }  ,draw opacity=1 ] (301.88,108.5) .. controls (301.88,99.39) and (309.47,92) .. (318.83,92) .. controls (328.19,92) and (335.78,99.39) .. (335.78,108.5) .. controls (335.78,117.61) and (328.19,125) .. (318.83,125) .. controls (309.47,125) and (301.88,117.61) .. (301.88,108.5) -- cycle ;
\draw  [color={rgb, 255:red, 208; green, 2; blue, 27 }  ,draw opacity=1 ] (539.12,108.5) .. controls (539.12,99.39) and (546.71,92) .. (556.07,92) .. controls (565.44,92) and (573.03,99.39) .. (573.03,108.5) .. controls (573.03,117.61) and (565.44,125) .. (556.07,125) .. controls (546.71,125) and (539.12,117.61) .. (539.12,108.5) -- cycle ;
\draw   (98.21,92.01) -- (200.82,92.01) -- (200.82,125.01) -- (98.21,125.01) -- cycle ;

\draw (108.4,102.06) node [anchor=north west][inner sep=0.75pt]    {$P_{1}$};
\draw (142.6,102.06) node [anchor=north west][inner sep=0.75pt]    {$P_{2}$};
\draw (174.14,136.4) node [anchor=north west][inner sep=0.75pt]  [font=\scriptsize]  {$z_{i,\epsilon ,L}$};
\draw (278.07,102.05) node [anchor=north west][inner sep=0.75pt]    {$P_{1}$};
\draw (346.27,102.05) node [anchor=north west][inner sep=0.75pt]    {$P_{2}$};
\draw (377.81,136.39) node [anchor=north west][inner sep=0.75pt]  [font=\scriptsize]  {$z_{i,\epsilon ,L}$};
\draw (481.4,102.05) node [anchor=north west][inner sep=0.75pt]    {$P_{1}$};
\draw (515.4,102.05) node [anchor=north west][inner sep=0.75pt]    {$P_{2}$};
\draw (581.14,136.39) node [anchor=north west][inner sep=0.75pt]  [font=\scriptsize]  {$z_{i,\epsilon ,L}$};
\draw (222.5,101.9) node [anchor=north west][inner sep=0.75pt]    {$\mapsto $};
\draw (426.16,101.9) node [anchor=north west][inner sep=0.75pt]    {$\mapsto $};

\end{tikzpicture}
\end{figure}

If the aforementioned site was occupied, we can exchange the hole and the particle at site $ z+i\floor{N\eps}+(j-1)L-4 $ with the PMM($ 1 $) dynamics, otherwise there is nothing to do and we relabel the hole, obtaining
\begin{figure}[H]
\centering

\tikzset{every picture/.style={line width=0.75pt}} 

\begin{tikzpicture}[x=0.75pt,y=0.75pt,yscale=-1,xscale=1]

\draw    (454.1,129.65) -- (621.53,129.69) (488.1,125.66) -- (488.1,133.66)(522.1,125.67) -- (522.1,133.67)(556.1,125.68) -- (556.1,133.68)(590.1,125.68) -- (590.1,133.68) ;
\draw  [fill={rgb, 255:red, 155; green, 155; blue, 155 }  ,fill opacity=0.5 ] (471.3,108.5) .. controls (471.3,99.39) and (478.89,92) .. (488.25,92) .. controls (497.62,92) and (505.21,99.39) .. (505.21,108.5) .. controls (505.21,117.61) and (497.62,125) .. (488.25,125) .. controls (478.89,125) and (471.3,117.61) .. (471.3,108.5) -- cycle ;
\draw  [fill={rgb, 255:red, 155; green, 155; blue, 155 }  ,fill opacity=0.5 ] (505.21,108.5) .. controls (505.21,99.39) and (512.8,92) .. (522.16,92) .. controls (531.53,92) and (539.12,99.39) .. (539.12,108.5) .. controls (539.12,117.61) and (531.53,125) .. (522.16,125) .. controls (512.8,125) and (505.21,117.61) .. (505.21,108.5) -- cycle ;
\draw  [color={rgb, 255:red, 208; green, 2; blue, 27 }  ,draw opacity=1 ] (573.03,108.5) .. controls (573.03,99.39) and (580.62,92) .. (589.98,92) .. controls (599.35,92) and (606.94,99.39) .. (606.94,108.5) .. controls (606.94,117.61) and (599.35,125) .. (589.98,125) .. controls (580.62,125) and (573.03,117.61) .. (573.03,108.5) -- cycle ;
\draw   (470.86,92) -- (573.47,92) -- (573.47,125) -- (470.86,125) -- cycle ;

\draw (481.4,102.05) node [anchor=north west][inner sep=0.75pt]    {$P_{1}$};
\draw (515.4,102.05) node [anchor=north west][inner sep=0.75pt]    {$P_{2}$};
\draw (581.14,136.39) node [anchor=north west][inner sep=0.75pt]  [font=\scriptsize]  {$z_{i,\epsilon ,L}$};

\end{tikzpicture}
\end{figure}
This procedure is repeated at most an order of $ L+(i\floor{N\eps}+(j-1)L) $ steps, moving the mobile cluster to the vicinity of the site $ z_{i,L} $. The SSEP dynamics is then used to shuffle the configuration restricted to the box $ \Lambda_{iL}^L $, moving $ P_1 $ and $ P_2 $ to their original sites with a cost of at most an order of $ L $ steps.

\end{proof}

\subsection{Replacement Lemmas for \texorpdfstring{$m\in(0,1)$}.}
{\begin{Lemma}\label{lem:rep_FDM-tight}
For each $ n,k\in\mathbb{N}_+ $ such that $ n\leq k $, let $ \varphi_n^{(k)}:[0,T]\times\Omega_N\to\mathbb{R} $ be invariant for the map $ \eta\mapsto\eta^{n,n+1} $ and such that for $ \eta\in\Omega_N $ and every $ t\in [0,T] $ 
\begin{align*}
\varphi_n^{(k)}(t,\eta)\leq M(t)\mathbf{c}^{(k)}(\tau_n\eta)
\end{align*}
where $ M:[0,T]\to\mathbb{R}_+$ is uniformly bounded by some constant $ M>0 $. Then for all $ B>0 $ and $ \mathcal{T}\subseteq[0,T] $ it holds that
\begin{align*}
\mathbb{E}_{\mu_N}
\left[
\abs{
\int_{\mathcal{T}}
\sum_{k=1}^{\ell_N}\abs{\binom{m}{k}}
\sum_{n=0}^{k}
(\overline{\eta}_{N^2s}(n)-\overline{\eta}_{N^2s}(n+1))
\varphi_n^{(k-1)}(s,\overline{\eta}_{N^2s})
\rmd s
}	\right]
\lesssim
\frac{1}{B}
+
\abs{\mathcal{T}}
B\frac{(\ell_N)^{1-m}}{N}
.
\end{align*}
\end{Lemma}
}\begin{proof}
Proceeding as previously, we have to estimate
\begin{align}\label{rep:FDM_var}
\frac{c_\gamma}{B}
+\int_{\mathcal{T}}
\sup_{f
}
\left\{
\abs{
\sum_{k=1}^{\ell_N}\abs{\binom{m}{k}}\sum_{n=0}^{k}\inner{\varphi_n^{(k-1)}(s,\overline{\eta})
(\overline{\eta}(n)-\overline{\eta}(n+1))
,f}_{\nu_\gamma^N}
}
-\frac{N}{B}\mathcal{E}_N^{(m-1)}(\sqrt{f},\nu_\gamma^N)
\right\}
\rmd s.
\end{align}
By the hypothesis on $ \varphi_n^{(k-1)} $ and Young's inequality we can bound from above the summation over $ n $ by:
\begin{multline*}
\frac{M}{4A}\int_{\Omega_N}
\sum_{n=0}^{k}
\abs{\eta(n)-\eta(n+1)}\mathbf{c}^{(k-1)}(\tau_n\overline{\eta})
\left(\sqrt{f}(\eta^{n,n+1})+\sqrt{f}(\eta)\right)^2
\nu_\gamma^N(\rmd\eta)
\\
+
\frac{AM}{4}\int_{\Omega_N}
\sum_{n=0}^{k}
\abs{\eta(n)-\eta(n+1)}\mathbf{c}^{(k-1)}(\tau_n\overline{\eta})
\left(\nabla_{n,n+1}\sqrt{f}(\eta)\right)^2
\nu_\gamma^N(\rmd\eta)
.
\end{multline*}
Taking the binomial coefficients into consideration, and since $ \abs{\eta(n)-\eta(n+1)}=\mathbf{a}(\tau_n\eta), $ which is the exclusion constraint, we bound from above
\begin{align*}
\int_{\Omega_N}
\sum_{k=1}^{\ell_N}\abs{\binom{m}{k}}
\sum_{n=0}^{k}
\abs{\eta(n)-\eta(n+1)}\mathbf{c}^{(k-1)}(\tau_n\overline{\eta})
\left(\nabla_{n,n+1}\sqrt{f}(\eta)\right)^2
\nu_\gamma^N(\rmd\eta)
\leq
\Gamma_N^{(m-1)}(\sqrt{f},\nu_{\gamma}^N).
\end{align*}
We have the following upper bounds
\begin{align*}
\int_{\Omega_N}
\sum_{k=1}^{\ell_N}\abs{\binom{m}{k}}
&	\sum_{n=0}^{k}
\abs{\eta(n)-\eta(n+1)}\mathbf{c}^{(k-1)}(\tau_n\overline{\eta})
\left(\sqrt{f}(\eta^{n,n+1})+\sqrt{f}(\eta)\right)^2
\nu_\gamma^N(\rmd\eta)
\\
&	\leq
2\int_{\Omega_N}
\sum_{k=1}^{\ell_N}\abs{\binom{m}{k}}
\sum_{n=0}^{k}
\abs{\eta(n)-\eta(n+1)}\mathbf{c}^{(k-1)}(\tau_n\overline{\eta})
\left(f(\eta^{n,n+1})+f(\eta)\right)
\nu_\gamma^N(\rmd\eta)
\\
&	\leq
4\int_{\Omega_N}
\sum_{k=1}^{\ell_N}\abs{\binom{m}{k}}
\left(
\sum_{n=0}^{k}
\abs{\eta(n)-\eta(n+1)}\mathbf{c}^{(k-1)}(\tau_n\overline{\eta})
\right)
f(\eta)
\nu_\gamma^N(\rmd\eta)
\\
&	\leq
8M
\sum_{k=1}^{\ell_N}\abs{\binom{m}{k}}k
\int_{\Omega_N}
f(\eta)
\nu_\gamma^N(\rmd\eta)
\leq
8M (\ell_N)^{1-m}.
\end{align*}
The previous inequalities follow, respectively, from Young's inequality, Lemma \ref{lem:change}, Lemma \ref{lem:up_speed}, the fact of $ f $ being a density and then Lemma \ref{lem:bin_bound} and an integral comparison. With all this, we obtain the following estimate for \eqref{rep:FDM_var}
\begin{align*}
\frac{c_\gamma}{B}
+\abs{\mathcal{T}}
\left(
2M\frac{(\ell_N)^{1-m}}{A}
+\Gamma^{(m-1)}_N(\sqrt{f},\nu_{\gamma}^N)
\left(
\frac14AM-\frac{N}{B}
\right)
\right).
\end{align*}
Fixing $ A=4N/BM $ concludes the proof.
\end{proof}

\begin{Lemma}\label{lem:rep_FDM}
Consider $ x,y\in\mathbb{T}_N $. Let $ \varphi:[0,T]\times\Omega_N\to\mathbb{R} $ such that $ \norm{\varphi}_{L^\infty([0,T]\times\Omega_N)}<\infty $ and invariant for the map $ \eta\mapsto\eta^{z,z+1} $ with $ z\in\llbracket x,y-1\rrbracket $. Then, for all $ B>0 $ and for all $ t\in[0,T] $ it holds
\begin{align*}
\mathbb{E}_{\mu_N}
\left[\abs{
\int_0^t
\varphi(s,\eta_{N^2s})
(
\eta_{N^2s}(x)-\eta_{N^2s}(y)
)
\rmd s}	\right]
\lesssim
\frac{1}{B}
+
T
B\frac{\abs{y-x}}{N}
.
\end{align*}

\end{Lemma}

\begin{proof}
Repeating the computations in the proof of Lemma \ref{lem:rep_shift}, there exist constants $ c_0,c_1,c_2>0 $ such that we can overestimate the expectation by
\begin{align*}
\frac{c_0}{B}
+T\sup_{f}
\left\{
c_1A\Gamma_N^{(0)}(\sqrt{f},\nu_\gamma^N) + c_2\frac{\abs{y-x}}{A}
-\frac{N}{B}\mathcal{E}^{(m-1)}_N(\sqrt{f},\nu_\gamma^N)
\right\}.
\end{align*}
Recalling the lower bound for the Dirichlet form in Proposition \ref{prop:energy} we can choose $ A=mN/c_1B $.
\end{proof}

\section{Energy Estimate}\label{sec:energy}

We recall some classical results that we will invoke throughout this section.
Let $ \mathcal{H} $ be a Hilbert space with corresponding norm $ \norm{\cdot}_{\mathcal{H}} $ and $ \mathcal{f}f:\mathcal{H}\to\mathbb{R} $ a linear functional. The (dual) norm of the linear functional $ f $ is defined as
\begin{align*}
\nnorm{f}=\sup_{\norm{x}_{\mathcal{H}}\leq1,x\in\mathcal{H}}\abs{f(x)}.
\end{align*}
We know that (see for instance \cite[Proposition A.1.1.]{phd:adriana}) if there exists $ K_0>0 $ and
a positive real number $ \kappa $ such that
$
\sup_{x\in\mathcal{H}}
\left\{
f(x)-\kappa \norm{x}_{\mathcal{H}}^2
\right\}
\leq K_0,
$
then $ f $ is bounded.
%
%
Let us now introduce:
\begin{Def}
Let $ L^2([0,T]\times \mathbb{T}) $ be the (Hilbert) space of measurable functions $ G:[0,T]\times \mathbb{T}\to\mathbb{R} $ such that
\begin{align*}
\int_0^T\norm{G_s}_{L^2(\mathbb{T})}^2\rmd s<\infty,
\end{align*}
endowed with the scalar product $ \llangle G,H\rrangle $ defined by
\begin{align*}
\llangle G,H\rrangle=\int_0^T\inner{G_s,H_s}\rmd s.
\end{align*}
For any $ r\in\mathbb{R}_+ $ fixed, define the linear functional $ \ell^{(r)} $ on $ C^{0,1}\left([0,T]\times\mathbb{T}\right) $ by $ \ell_\rho^{(r)}(G)=\llangle\partial_u G,\rho^r\rrangle $.
\end{Def}

An important result is the following:\begin{Lemma}\cite[Lemma A.1.9]{phd:adriana}.\label{lem:reg_xi}
If  $ \xi\in L^2([0,T]\times \mathbb{T}) $ is such that there exists a function $ \partial\xi\in L^2([0,T]\times \mathbb{T}) $ satisfying for all $ G\in C^{0,1}([0,T]\times\mathbb{T}) $ the identity
\begin{align*}
\llangle\partial_uG,\xi\rrangle=-\llangle G,\partial\xi\rrangle,
\end{align*}
then $ \xi\in L^2([0,T];\mathcal{H}^1(\mathbb{T})) $.
\end{Lemma}

\begin{Def}
For $ G\in C^{0,1}([0,T]\times\mathbb{T}),\; r,\kappa\in\mathbb{R}_+ $ define $ \mathscr{E}_{G,\kappa}^{(m)}:\mathcal{D}([0,T],\mathcal{M}_+)\to\mathbb{R}\cup\{\infty\} $ by
\begin{align*}
\mathscr{E}_{G,\kappa}^{(r)}(\pi)
=
\begin{cases}
\ell^{(r)}(G)-\kappa\norm{G}^2_2, & \text{ if }
\pi\in\mathcal{D}([0,T],\mathcal{M}_+),\\
+\infty, & \text{ otherwise},
\end{cases}
\end{align*}
and the energy functional $ \mathscr{E}^{(r)}_\kappa:\mathcal{D}([0,T],\mathcal{M}_+)\to\mathbb{R}\cup\{\infty\} $ by
\begin{align*}
\mathscr{E}^{(r)}_\kappa(\pi)=\sup_{G\in C^{0,1}([0,T]\times\mathbb{T})}\mathscr{E}_{G,\kappa}^{(r)}(\pi).
\end{align*}
\end{Def}
\begin{Rem}
Note that $ \mathscr{E}^{(r)}_\kappa(\pi)\geq0 $. To see this it is enough to take $ G=0 $.
\end{Rem}
Recall that the measure $ \mathbb{Q} $ is the weak limit of a subsequence of $ \mathbb{Q}_N $ as $ N\to+\infty $, where $ \mathbb{Q}_N $ is the measure induced by the empirical measure in the Skorokhod space of trajectories $ \mathcal{D}([0,T],\Omega_N) $. Recall also the definition of the target Sobolev space (Definition \ref{def:sob}). The main goal of this section is to prove the next proposition.
\begin{Prop}\label{prop:power_in_sob}
The measure $ \mathbb{Q} $ is concentrated on trajectories of absolutely continuous measures with respect to the Lebesgue measure, $ \pi_\cdot(\mathrm{d}u)=\rho_\cdot(u)\mathrm{d}u $, such that $ \rho^m\in L^2([0,T];\mathcal{H}^1(\mathbb{T})) $, for $ m\in(1,2) $, and $ \rho\in L^2([0,T],\mathcal{H}^1(\mathbb{T})) $, for $ m\in(0,1) $.
\end{Prop}
This will be shown to be consequence of existing positive real numbers $ \kappa_0,\kappa_1,K_0 $ and $ K_1 $ such that for $ m\in(0,1) $ holds $ E_\mathbb{Q}\left[
\mathscr{E}_{\kappa_0}^{(1)}(\pi)
\right]\leq K_0 $, and for $ m\in(1,2) $ holds $ E_\mathbb{Q}\left[
\mathscr{E}_{\kappa_1}^{(m)}(\pi)
\right]\leq K_1 $, 
{where $E_\mathbb{Q}$ denotes the expectation with respect to $\mathbb{Q}$.} This will be proved in Proposition \ref{prop:energy_est_FDE} and \ref{prop:energy_est_PME}, respectively. Recall \eqref{grad:non_int}.  For the slow diffusion case, the argument is analogous to \cite[Section $ 6 $]{BDGN} but we make evident that this argument works due to the fact that the rates are uniformly bounded by a constant independent of $ N $ and the fact that the model is gradient. In particular, the argument is suited to show that the "macroscopic" quantity
\begin{align*}
\rho^m=\lim_{N\to+\infty} \int
h_N^{(m-1)}(\eta) \nu_\rho^N(\rmd \eta)
\end{align*}
lives in the target Sobolev space, where $ \rho(\cdot)\in(0,1) $ is a constant function. As in \cite{BDGN}, the argument does not allow us to show that $ \rho $ has a weak derivative, the reason being that $ \rho^m\leq \rho $.

For $ m\in(0,1) $ we have the opposite problem. Without imposing any restriction on the initial profile, we cannot show that $ \rho^m\in L^2([0,T],\mathcal{H}^1(\mathbb{T})) $, the reason being that  (see Remark \ref{rem:bound})
\begin{align*}
\lim_{N\to+\infty}\sup_{\eta\in\Omega_N}r_N^{(m-1)}(\eta)=+\infty.
\end{align*}
This is the discrete analogous to $ \rho^{m-1}\to +\infty $ as $\rho\to0$. Yet, we can show that $ \rho\in L^2([0,T];\mathcal{H}^1(\mathbb{T})) $ because the transition rates, in this case, are larger than the ones for the SSEP (analogous to $ \rho\leq \rho^{m} $ in this case), which is again a gradient model. \begin{proof}[Proof of Proposition \ref{prop:power_in_sob}]
Recall that up to this point we have proved that the measure $ \mathbb{Q} $ is a \textit{Dirac measure}, namely $ \mathbb{Q}=\delta_{\pi} $ with $ \pi_\cdot $ the trajectory of absolutely continuous measures $ \pi_\cdot(\rmd u)=\rho_\cdot(u)\rmd u $, where $ \rho $ satisfies the weak formulation \eqref{weak}.  For $ m\in(1,2) $, from Proposition \ref{prop:energy_est_PME} the functional $ \ell^{(m)} $ is bounded $ \mathbb{Q}-$a.s. Since $ C^{0,1}([0,T]\times\mathbb{T}) $ is dense in $ L^2([0,T]\times\mathbb{T}) $, we can extend $ \ell^{(m)} $ to a $ \mathbb{Q}-$a.s.~bounded functional in $ L^2([0,T]\times\mathbb{T}) $. One can thus invoke Riesz's representation Theorem and conclude that for any $ m\in(1,2) $ there exists a function $ \partial\rho^{m}\in L^2([0,T]\times\mathbb{T}) $ such that
\begin{align*}
\ell_\rho^{(m)}(G)=-\llangle G,\partial\rho^{m}\rrangle.
\end{align*}
To finish the proof, since $ \rho^m\in L^2([0,T]\times\mathbb{T}) $, one invokes Lemma \ref{lem:reg_xi}.\par
For $ m\in(0,1) $ the same argument leads to $ \rho\in L^2([0,T],\mathcal{H}^1(\mathbb{T})) $ but now one should invoke instead  Proposition \ref{prop:energy_est_FDE} which states that the functional $ \ell^{(1)} $ is bounded.
\end{proof}

\begin{Prop}\label{prop:energy_est_PME}
For any $ m\in(1,2) $ there are finite constants $ \kappa,K>0 $ such that
\begin{align*}
E_{\mathbb{Q}}
\left[
\mathscr{E}^{(m)}_{\kappa}(\pi)
\right]
\leq K.
\end{align*}
\end{Prop}

\begin{proof}
Recall that from the binomial theorem we can expand
\begin{align*}
\rho^m=\sum_{k\geq 0}\binom{m}{k}(-1)^k(1-\rho)^k,
\end{align*}
and since we are on the torus we can treat
\begin{align*}
E_{\mathbb{Q}}
\left[
\sup_{G\in C^{0,1}\left([0,T]\times\mathbb{T}_N\right)}
\bigg\{
\sum_{k\geq 1}\binom{m}{k}(-1)^{k}
\llangle
(1-\rho)^{k}
,\partial_u G
\rrangle-\kappa_1\norm{G}_2^2
\bigg\}
\right]
.
\end{align*}
Recalling that $ C^{0,1}\left([0,T]\times \mathbb{T}\right) $ is separable with respect to the norm $ \norm{\cdot}_{\mathcal{H}^1(\mathbb{T})} $, consider a countable dense subset, $ \{G_p\}_{p\in\mathbb{N}} $, in $ C^{0,1}\left([0,T]\times \mathbb{T}\right) $. An application of the monotone convergence theorem then reduces the problem to that of treating
\begin{align*}
\lim_{\ell\to+\infty}
E_{\mathbb{Q}}
\bigg[\max_{\substack{G_p\\p\leq \ell}}
\left\{\mathcal{E}_{G_p}(\pi)\right\}\bigg].
\end{align*}
Fixed $ G_p $, Lemma \ref{lem:eps_seq} allow us to replace $ (1-\rho)^{k}
$ by $ \prod_{j=0}^{k-1}
\big(1-\inner{\pi,{\iota}_{\eps_{k}}^{\cdot+j\eps_{k}}}\big) $, with the sequence $ (\eps_k)_{k\geq0} $ depending on the regime of $ m $ and with a cost of $ \mathcal{O}(\eps^{\frac14}) $, leaving us with
\begin{align*}
E_{\mathbb{Q}}
\bigg[\max_{\substack{G_p\\p\leq \ell}}
\left\{\mathcal{E}_{G_p}(\pi)\right\}\bigg]
\leq
E_{\mathbb{Q}}
\bigg[
\max_{\substack{G_p\\p\leq \ell}}
\left\{
\sum_{k\geq 1}\binom{m}{k}(-1)^k
\Big\llangle{
\prod_{j=0}^{k-1}
\left(1-\inner{\pi,{\iota}_{\eps_{k}}^{\cdot+j\eps_{k}}}\right)
,\partial_u G_p
}\Big\rrangle
-\kappa_1\norm{G_p}_2^2
\right\}
\bigg]
+\mathcal{O}(\eps^\frac14).
\end{align*}
Recalling that the $ \limsup $ is monotone, we can take $ \limsup_{\eps\to0} $ on both sides of the inequality above. Note that we need to take the $ \limsup_{\eps\to0} $ outside of the expectation, since otherwise we get from the reverse of Fatou's lemma that $ E_\mathbb{Q}\limsup\geq \limsup E_\mathbb{Q} $. And so, we further reduce the problem to the study of
\begin{align*}
\lim_{\ell\to+\infty}\limsup_{\eps\to 0}
E_{\mathbb{Q}}
\bigg[
\max_{\substack{G_p\\p\leq \ell}}
\left\{
\sum_{k\geq 1}\binom{m}{k}(-1)^k
\Big\llangle{
\prod_{j=0}^{k-1}
\left(1-\inner{\pi,{\iota}_{\eps_{k}}^{\cdot+j\eps_{k}}}\right)
,\partial_u G_p
}\Big\rrangle
-\kappa_1\norm{G_p}_2^2
\right\}
\bigg].
\end{align*}
To make the link between the microscopic system and the macroscopic PDE we want to express $ \mathbb{Q} $ as the limit of a subsequence of $ (\mathbb{Q}_N)_{N\geq 0} $, thus replacing $ \pi $ by $ \pi^N $ and then recovering the occupation variables from the application of replacement lemmas. To do this, as previously, one wants to argue that the map
\begin{align*}
\pi\mapsto\Psi(\pi)
=\max_{\substack{G_p\\p\leq \ell}}
\left\{
\sum_{k\geq 1}\binom{m}{k}(-1)^k
\Big\llangle{
\prod_{j=0}^{k-1}
\left(1-\inner{\pi,{\iota}_{\eps_{k}}^{\cdot+j\eps_{k}}}\right)
,\partial_u G_p
}\Big\rrangle
-\kappa_1\norm{G}_2^2
\right\}
\end{align*}
is continuous with respect to the Skorokhod topology, hence lower semicontinuous and therefore $ \Psi(\pi)\leq \liminf_{N\to+\infty}\Psi(\pi^N). $ Although this is not the case, one can first truncate the series at an $ \ell_{1/\eps} $ step, then replace $ \inner{\pi,{\iota}_{\eps_{k}}^{\cdot+j\eps_{k}}} $ by $ \inner{\pi,{\iota}_{\eps_{k}}^{\cdot+j\eps_{k}}}\star \varphi_{\tilde{\eps}_k} $ as in \eqref{conv-to-moll}, and then argue by lower semicontinuity. Next, we replace the sum up to $ \ell_{1/\eps} $ by a sum up to $ \ell_N $, as in \eqref{truncate2}, and finally we replace back $ \inner{\pi^N,{\iota}_{\eps_{k}}^{\cdot+j\eps_{k}}}\star \varphi_{\tilde{\eps}_k} $ by $ \inner{\pi^N,{\iota}_{\eps_{k}}^{\cdot+j\eps_{k}}} $. In this way, we have to treat
\begin{align*}
\lim_{\ell\to+\infty}\limsup_{\eps\to 0}
E_{\mathbb{Q}_N}\bigg[
\liminf_{N\to+\infty}
\max_{\substack{G_p\\p\leq \ell}}
\left\{
\sum_{k= 1}^{\ell_N}\binom{m}{k}(-1)^k
\Big\llangle{
\prod_{j=0}^{k-1}
\left(1-\inner{\pi^N,{\iota}_{\eps_{k}}^{\cdot+j\eps_{k}}}\right)
,\partial_u G_p
}\Big\rrangle
-\kappa_1\norm{G_p}_2^2
\right\}
\bigg],
\end{align*}
where we recall again that $ \mathbb{Q} $ is a Dirac measure. Now we apply Fatou's lemma to exchange the expectation and the $ \liminf $. Hence, it is enough to show that there exists some constant $ K_1 $ independent of $ \{G_p\}_{p\leq\ell} $ for any $ \ell\in\mathbb{N} $ such that
\begin{align*}
\mathbb{E}_{\mu_N}\bigg[
\max_{\substack{G_p\\p\leq \ell}}
\left\{
\sum_{k= 1}^{\ell_N}\binom{m}{k}(-1)^k
\Big\llangle{
\prod_{j=0}^{k-1}
\left(1-\inner{\pi^N,{\iota}_{\eps_{k}}^{\cdot+j\eps_{k}}}\right)
,\partial_u G_p
}\Big\rrangle
-\kappa_1\norm{G_p}_2^2
\right\}
\bigg]
\leq K_1.
\end{align*}
Because $ \partial_uG_p $ is bounded in $ L^1 $ and the products involving the empirical measure are bounded by $ 1 $, we can replace $ \ell_N $ by $ (\ell_N)^n $ with $ n $ as previously (see \eqref{PMM:n0} and the computations that follow it). Now we are able to proceed backwards in the replacement lemmas' scheme (from \eqref{step1} to \eqref{step4:1}), approximating the space integral by the Riemann sum along the way. At this point we have to estimate
\begin{align*}
\mathbb{E}_{\mu_N}
\bigg[
\max_{p\leq \ell}
\left\{
\int_0^T
\bigg(
\frac{1}{N}\sum_{x\in\mathbb{T}_N}\sum_{k=1}^{(\ell_N)^n}\binom{m}{k}(-1)^k
\prod_{j=0}^{k-1}\overline{\eta}_{N^2s}(x+j)
\partial_u G_p(s,\tfrac{x}{N})
-\kappa_1\norm{G_p(s,\cdot)}_2^2
\bigg)
\rmd s
\right\}
\bigg],
\end{align*}
where we recall that $ (\ell_N)^n $ can be replaced back by $ \ell_N $ since the terms involving $ \eta $ are bounded and $ \partial_uG_p $ is bounded in $ L^1 $. We are able to introduce
\begin{align*}
\tau_x
\left\{
\sum_{i=1}^{k-2}
(\eta(i)-\eta(i+1))\sum_{j=1}^{k-1-i}\mathbf{s}_j^{(k-1)}(\tau_i\eta)
\right\}
\end{align*}
inside the summations over $ x $ and $ k $ (see the treatment of the first probability in \eqref{eq:h_prob}). This is important because now we have
\begin{align*}
\mathbb{E}_{\mu_N}
\bigg[
\max_{p\leq \ell}
\left\{
\int_0^T
\bigg(
\frac{1}{N}\sum_{x\in\mathbb{T}_N}
h_{N}^{(m-1)}(\tau_x\eta_{N^2s})
\partial_u G_p(s,\tfrac{x}{N})
-\kappa_1\norm{G_p(s,\cdot)}_2^2
\bigg)
\rmd s
\right\}
\bigg]
\end{align*}
which will be used to exploit the gradient property of the model. Analogously to the replacement lemmas, we obtain the upper bound 
\begin{align}\label{energy:var}
c_\gamma+\int_0^T\sup_{f}
\left\{
\bigg\langle{\frac{1}{N}\sum_{x\in\mathbb{T}_N}
h_{N}^{(m-1)}(\tau_x\overline{\eta})
\partial_u G_p(s,\tfrac{x}{N})
,\;f}\bigg\rangle_{\nu^N_\gamma}
-\kappa_1\norm{G_p(s,\cdot)}_2^2
-N\mathcal{E}_N^{(m-1)}(\sqrt{f},\nu_\gamma^N)
\right\}\rmd s
\end{align}
where $ c_\gamma>0 $ is a constant. 
Let us now focus on the inner product above, specifically on
\begin{align*}
\frac{1}{N}\sum_{x\in\mathbb{T}_N}
\partial_u G_p(s,\tfrac{x}{N})
\sum_{\eta\in \Omega_N}
h_{N}^{(m-1)}(\tau_x\overline{\eta})
f(\eta)\nu_{\gamma}^N(\eta).
\end{align*}
One can replace the space derivative by its discrete version with a cost
\begin{align*}
\frac{1}{N}\sum_{x\in\mathbb{T}_N}
\abs{\left(\partial_u -N\nabla^+\right)G_p(s,\tfrac{x}{N})}
\sum_{\eta\in \Omega_N}
h_{N}^{(m-1)}(\tau_x\overline{\eta})
f(\eta)\nu_{\gamma}^N(\eta)
\leq
\norm{\left(\partial_u -N\nabla^+\right)G_p(s,\cdot)}_{\infty}
\sup_{\eta\in\Omega_N}h_{N}^{(m-1)}(\eta).
\end{align*}
This vanishes on the limit $ N\to+\infty $ since $ \norm{\left(\partial_u -N\nabla^+\right)G_p(s,\cdot)}_{\infty}\lesssim \frac1N $ and, since $ \mathbf{h}^{(k)}\leq k $, we have that $h_{N}^{(m-1)}(\eta)\lesssim 1 $.

At this point the discrete derivative can be passed to $ h^{(m-1)}_{N} $ by performing a summation by parts, which puts us in place to use the gradient property of the model:
\begin{align*}
\sum_{x\in\mathbb{T}_N}
\nabla^{+}G_p(s,\tfrac{x}{N})
&	\sum_{\eta\in \Omega_N}
h_{N}^{(m-1)}(\tau_x\overline{\eta})
f(\eta)\nu_{\gamma}^N(\eta)
=
-\sum_{x\in\mathbb{T}_N}
G_p(s,\tfrac{x+1}{N})
\sum_{\eta\in \Omega_N}
c_N^{(m-1)}(\overline{\eta})
\nabla^{+}\eta(x)
f(\eta)\nu_{\gamma}^N(\eta).
\end{align*}
From Lemma \ref{lem:change},
\begin{align*}
\sum_{\eta\in \Omega_N}
c_N^{(m-1)}(\tau_x\overline{\eta})\nabla^{+}\eta(x)
f(\eta)\nu_{\gamma}^N(\eta)
&=
-\frac12
\sum_{\eta\in \Omega_N}
c_N^{(m-1)}(\tau_x\overline{\eta})\nabla^{+}\eta(x)
\nabla_{x,x+1}f(\eta)\nu_{\gamma}^N(\eta).
\end{align*}
and we are left with
\begin{align*}
\frac12&\sum_{x\in\mathbb{T}_N}
G_p(s,\tfrac{x+1}{N})\sum_{\eta\in \Omega_N}
c_N^{(m-1)}(\tau_x\overline{\eta})\nabla^{+}\eta(x)
\nabla_{x,x+1}f(\eta)\nu_{\gamma}^N(\eta)
\\
&	\leq
\frac{1}{4A}\sum_{x\in\mathbb{T}_N}\sum_{\eta\in \Omega_N}
c_N^{(m-1)}(\tau_x\overline{\eta})
\left(G_p(s,\tfrac{x+1}{N})\right)^2
\left(\sqrt{f}(\eta)+\sqrt{f}(\eta^{x,x+1})\right)^2
\nu_{\gamma}^N(\eta)
\\
&\quad +\frac{A}{4}\sum_{x\in\mathbb{T}_N}\sum_{\eta\in \Omega_N}
c_N^{(m-1)}(\tau_x\overline{\eta})\left(\nabla^{+}\eta(x)\right)^2
\left(
(\nabla_{x,x+1}\sqrt{f})(\eta)
\right)^2
\nu_{\gamma}^N(\eta)
\\
&\leq
\frac{1}{A}\sup_{\eta\in\Omega_N}\{c_N^{(m-1)}(\eta)\}\sum_{x\in\mathbb{T}_N}\left(G_p(s,\tfrac{x+1}{N})\right)^2
+\frac{A}{4}\Gamma_N^{(m-1)}(\sqrt{f},\nu_\gamma^N).
\end{align*}
Recalling that $ \sup_{\eta\in\Omega_N}\{c_N^{(m-1)}(\eta)\}\leq m $, fixing $ A=N $ and replacing all this into \eqref{energy:var}, then taking the corresponding limits finishes the proof.
\end{proof}

\begin{Prop}\label{prop:energy_est_FDE}
For any $ m\in(0,1) $ there are finite constants $ \kappa,K>0 $ such that
\begin{align*}
E_{\mathbb{Q}}
\left[
\mathscr{E}^{(1)}_{\kappa}(\pi)
\right]
\leq K.
\end{align*}
\end{Prop}

\begin{proof}
The proof is almost identical to the one for the SSEP (see for example \cite[Proposition B.1]{EGN1}). Besides the fact that there the authors have a boundary term, the differences lie in that we apply the replacement lemmas in the present text, and that in the final step we need to invoke Proposition \ref{prop:energy}. We outline the main steps. The treatment of the expectation in the statement can be reduced to the treatment of
\begin{align}\label{energy:var_FDE}
\sup_{f}
\left\{
\bigg\langle{\frac{1}{N}\sum_{x\in\mathbb{T}_N}
\eta^{\eps N}(x)
\partial_u G_p(s,\tfrac{x}{N})
-\kappa_0\norm{G_p(s,\cdot)}_2^2,f}\bigg\rangle_{\nu^N_\gamma}
-N\mathcal{E}_N^{(m-1)}(\sqrt{f},\nu_\gamma^N)
\right\}
.
\end{align}
Above, the $ \sup $ is taken over the set of densities with respect to $ \nu_\gamma^N $, and $ \{G_p\}_{p\in\mathbb{N}} $ is a countable dense subset in $ C^{0,1}([0,T]\times\mathbb{T}) $. Exchanging the continuous derivative by a discrete one, then performing a summation by parts we end up having to treat
\begin{align*}
\sup_{f}
\left\{
\bigg\langle{-\sum_{x\in\mathbb{T}_N}
\nabla^+\eta^{\eps N}(x)
G_n(s,\tfrac{x}{N})
-\kappa_0\norm{G_n(s,\cdot)}_2^2,f}\bigg\rangle_{\nu^N_\gamma}
-N\mathcal{E}_N^{(m-1)}(\sqrt{f},\nu_\gamma^N)
\right\}.
\end{align*}
Then again, from Lemma \ref{lem:change} we have that
\begin{align*}
\int_{\Omega_N}
\nabla^+\eta^{\eps N}(x)
f(\eta)\nu_\gamma^N(\rmd\eta)
=\frac{1}{2\eps N}\sum_{i\in\Lambda_N^{\eps N}}
\int_{\Omega_N}
\left(\eta(x+i+1)-\eta(x+i)\right)
\left(f(\eta)-f(\eta^{x+i,x+i+1})\right)\nu_\gamma^N(\rmd\eta).
\end{align*}
Taking our function $ G_p $ back into consideration and recalling that the process is of exclusion type we have that
\begin{align*}
\sum_{x\in\mathbb{T}_N}
G_p(s,\tfrac{x}{N})&\frac{1}{2\eps N}\sum_{i\in\Lambda_N^{\eps N}}
\int_{\Omega_N}
\left(\eta(x+i+1)-\eta(x+i)\right)
\left(f(\eta)-f(\eta^{x+i,x+i+1})\right)\nu_\gamma^N(\rmd\eta)
\\
&\leq
\frac{1}{4A}
\sum_{x\in\mathbb{T}_N}
\left(G_p(s,\tfrac{x}{N})\right)^2
\int_{\Omega_N}
\left(\sqrt{f}(\eta)+\sqrt{f}(\eta^{x+i,x+i+1})\right)^2\nu_\gamma^N(\rmd\eta)
+\frac{A}{4}\Gamma_N^{(0)}(\sqrt{f},\nu_\gamma^N).
\end{align*}
Since $ f $ is a density, last display is no larger than
\begin{align*}
\frac{1}{2A}
\sum_{x\in\mathbb{T}_N}
\left(G_p(s,\tfrac{x}{N})\right)^2
+\frac{A}{4}\Gamma_N^{(0)}(\sqrt{f},\nu_\gamma^N).
\end{align*}
Plugging this into \eqref{energy:var_FDE}, we obtain
\begin{align*}
\sup_{f}
\left\{
\frac{1}{2A}
\sum_{x\in\mathbb{T}_N}
\left(G_p(s,\tfrac{x}{N})\right)^2
-\kappa\norm{G_p(s,\cdot)}_2^2
+\frac{A}{4}\Gamma_N^{(0)}(\sqrt{f},\nu_\gamma^N)
-N\mathcal{E}_N^{(m-1)}(\sqrt{f},\nu_\gamma^N)
\right\}.
\end{align*}
Applying the lower bound for $ \mathcal{E}_N^{(m-1)}(\sqrt{f},\nu_\gamma^N) $ from Proposition \ref{prop:energy} and
hence setting $ A=4N $ and $ \kappa=1/8 $, to conclude we just need to perform the necessary limits.
\end{proof}

\appendix
\section{Auxiliary results}\label{app:aux_res}

\begin{Lemma}\label{lem:bin_bound}
For any $ m\in\mathbb{R}_+ $ and any $ k\in\mathbb{N}_+ $ such that $ k\geq 2 $, it holds
\begin{align*}
\frac{\Gamma(m)\abs{\sin(\pi(k-m))}}{\pi(k+1)^m}
<
\abs{\binom{m-1}{k}}
<
\frac{\Gamma(m)\abs{\sin(\pi(k-m))}}{\pi(k-m)^m}
\lesssim \frac{1}{k^m},
\end{align*}
where the $ \Gamma-$function is defined, for any $ z\in\mathbb{C} $ such that $ \mathrm{Re}(z)>0 $, as
\begin{align*}
\Gamma(z)=\int_0^{+\infty} u^{z-1}e^{-u}\rmd u.
\end{align*}
\end{Lemma}

\begin{proof}
The binomial coefficients have the following classical representation in terms of the $ \Gamma-$function 
\begin{align*}
\binom{m-1}{k}
=\frac{\Gamma(m)}{\Gamma(k+1)\Gamma(m-k)}.
\end{align*}
From the reflection formula
\begin{align*}
\Gamma(m-k)\Gamma(k+1-m)=\frac{\pi}{\sin(\pi(m-k))},
\end{align*}
we can rewrite
\begin{align*}
\binom{m-1}{k}
=\frac{\sin(\pi(m-k))}{\pi}\frac{\Gamma(m)\Gamma(k+1-m)}{\Gamma(k+1)}.
\end{align*}
Recall now the $ \mathrm{B}-$function, defined on $  z,w\in\mathbb{C}:\; \mathrm{Re}(z),\mathrm{Re}(w)>0 $, as 
\begin{align*}
\mathrm{B}(z,w)=\int_0^1v^{z-1}(1-v)^{w-1}\rmd v
=
\int_0^{+\infty}\frac{s^{w-1}}{(s+1)^{w+z}}\rmd s,
\end{align*}
where the equality of the representations above can be checked by performing the change of variables $ v=s/(s+1) $ on the first integral.
From the definition of $ \Gamma $, one can show the following classical relationship between the $ \Gamma $ and $ \mathrm{B} $ functions, for all $ z,w\in\mathbb{C}:\; \mathrm{Re}(z),\mathrm{Re}(w)>0 $:
\begin{align*}
\mathrm{B}(z,w)=\frac{\Gamma(z)\Gamma(w)}{\Gamma(z+w)}.
\end{align*}
In this way, we can rewrite
\begin{align}\label{eq:bin-beta}
\binom{m-1}{k}
=\frac{\sin(\pi(m-k))}{\pi}\mathrm{B}(m,k+1-m).
\end{align}
Recall that for $ k\geq 2 $ holds $ (m-1)_k=(-1)^{k-\floor{m}}\abs{(m-1)_k} $. Noticing that $ \mathrm{B}(m,k+1-m)>0 $, we then have that $ \sin(\pi(m-k))=(-1)^{k-\floor{m}}\abs{\sin(\pi(m-k))} $ and we need only to find an upper and lower bound for the $ \mathrm{B}-$function.	From the inequality $ e^x\geq 1+x $, valid for $ x\in\mathbb{R} $, the rescaling $ v=u/(w-1) $ with $ w>1 $ on
\begin{align*}
\Gamma(z)
=\int_0^{+\infty}u^{z-1}e^{-u}du
=(w-1)^z\int_0^{+\infty}v^{z-1}e^{-(w-1)v}dv
>
(w-1)^z\mathrm{B}(z,w),
\end{align*}
and from the rescaling $ v=u/(z+w)^z $,
\begin{align*}
\Gamma(z)
=\int_0^{+\infty}u^{z-1}e^{-u}du
=(z+w)^z\int_0^{+\infty}v^{z-1}e^{-(w-1)v}dv
<
(z+w)^z\mathrm{B}(z,w).
\end{align*}
We conclude that
\begin{align*}
\frac{\Gamma(m)}{(k+1)^m}<\mathrm{B}(m,k+1-m)<\frac{\Gamma(m)}{(k-m)^m}.
\end{align*}
\end{proof}

We now prove Lemma \ref{lem:grad}.

\begin{proof}
From \cite{GNP21} we have the following expression
\begin{align}\label{expr:h1}
	\mathbf{h}^{(k)}(\eta)=\sum_{j=1}^{k+1}\prod_{i=j-(k+1)}^{j-1}\eta(i)
	-\sum_{j=1}^{k}\prod_{\substack{i=-(k+1)+j\\i\neq0}}^j\eta(i).	
\end{align}
Expression \eqref{expr:h2}
is a consequence of a rearrangement which turns out to be fundamental for maintaining $ \ell_N $ with no restrictions. Indeed, we can rewrite
\begin{align*}
\sum_{j=1}^{k+1}\prod_{i=j-(k+1)}^{j-1}\eta(i)
-\sum_{j=1}^{k}\prod_{\substack{i=-(k+1)+j\\i\neq0}}^j\eta(i)
=\prod_{i=0}^{k}\eta(i)
+\sum_{j=1}^{k}\left(\eta(0)-\eta(j)\right)\prod_{\substack{i=-(k+1)+j\\i\neq 0}}^{j-1}\eta(i).
\end{align*}
Note that
\begin{align*}
\left(\eta(0)-\eta(j)\right)
\prod_{\substack{i=-(k+1)+j\\i\neq 0}}^{j-1}\eta(i)
=\prod_{i=-(k+1)+j}^{j-1}\eta(i)
-\prod_{\substack{i=-(k+1)+j\\i\neq 0}}^{j}\eta(i).
\end{align*}
Now we reorganize the products on the second term above. For $ n\in\llbracket -(k+1)+j,j-1\rrbracket $ we have
\begin{align*}
\prod_{\substack{i=-(k+1)+j\\i\neq n+1}}^{j}\eta(i)
=
\left(\eta(n)-\eta(n+1)\right)
\prod_{\substack{i=-(k+1)+j\\i\neq n,n+1}}^{j}\eta(i)
+\prod_{\substack{i=-(k+1)+j\\i\neq n}}^{j}\eta(i).
\end{align*}
Observing that a change of variables yields
\begin{align*}
\prod_{\substack{i=-(k+1)+j\\i\neq n,n+1}}^{j}\eta(i)
=\prod_{\substack{i=-(k+1)+j-n\\i\neq 0,1}}^{j-n}\eta(i+n)
=\mathbf{s}_{j-n}^{(k)}(\tau_n\eta),
\end{align*}
by iteration we see that
\begin{align*}
\prod_{\substack{i=-(k+1)+j\\i\neq 0}}^{j}\eta(i)
=
\prod_{\substack{i=-(k+1)+j\\i\neq j}}^{j}\eta(i)
-\sum_{i=0}^{j-1}
(\eta(i)-\eta(i+1))\mathbf{s}_{j-i}^{(k)}(\tau_i\eta).
\end{align*}
Exchanging the summations and performing a change of variables,
\begin{align*}
\sum_{j=1}^{k}
\sum_{i=0}^{j-1}
(\eta(i)-\eta(i+1))\mathbf{s}_{j-i}^{(k)}(\tau_i\eta)
=\sum_{i=0}^{k-1}
(\eta(i)-\eta(i+1))
\sum_{j=1}^{k-i}
\mathbf{s}_{j}^{(k)}(\tau_i\eta),
\end{align*}
which ends the proof.
\end{proof}

\section{PDE results}\label{app:PDE}

\subsection{Slow diffusion}

The following result extends   \cite[Lemma $ 6.2 $]{BDGN}  to  the case $ m\in(1,2) $.
\begin{Prop}\label{prop:power_diff}
Let $ f,g\in[0,1] $ with $ f\neq g $. If $ m\in(1,2) $ then, for all $ A>0 $ we have
\begin{align*}
\abs{f-g}\leq \frac{(f)^m-(g)^m}{V^{(m)}(f,g)+A}+A\frac{2}{m(m-1)}.
\end{align*}
where
\begin{align*}
0<V^{(m)}(f,g)
=\sum_{k\geq1}\binom{m}{k}(-1)^{k+1}v_k(1-f,1-g)<\infty
\end{align*}
and
\begin{align*}
v_k(f,g)
=\mathbf{1}_{k=1}
+\mathbf{1}_{k=2}(f+g)
+\mathbf{1}_{k\geq 3}\left(f^{k-1}+g^{k-1}+\sum_{i=1}^{k-2}g^if^{k-1-i}\right).
\end{align*}
\end{Prop}
\begin{proof}
We start with $ f,g\in(0,1) $.
\begin{align*}
(f)^m-(g)^m
&=\sum_{k\geq 1}\binom{m}{k}(-1)^k\left((1-f)^k-(1-g)^k\right).
\end{align*}
We now recall that one can rewrite, for any $ k\in\mathbb{N}_+ $,
\begin{align}\label{diff_powers}
a^k-b^k
=(a-b)
\left[
\mathbf{1}_{k=1}
+\mathbf{1}_{k=2}(a+b)
+\mathbf{1}_{k\geq 3}\left(a^{k-1}+b^{k-1}+\sum_{i=1}^{k-2}b^ia^{k-1-i}\right)
\right]
=(a-b)v_k(a,b).
\end{align}
In this way,
\begin{align*}
(f)^m-(g)^m
=(f-g)\sum_{k\geq1}\binom{m}{k}(-1)^{k+1}v_k(1-f,1-g)
=(f-g)V^{(m)}(f,g).
\end{align*}
We show that $ V^{(m)}(f,g)>0 $. Assume $ f,g\in(0,1) $ with $ f>g $. Then, $ f^m-g^m>0 $ implies $ V^{(m)}(f,g)>0 $. Similarly, if $ f<g $ then $ f^m-g^m<0\implies V^{(m)}(f,g)>0 $. With this in mind, we can rewrite
\begin{align*}
(f)^m-(g)^m=(f-g)\left(V^{(m)}(f,g)\pm A\right)
\Leftrightarrow
f-g=\frac{(f)^m-(g)^m}{V^{(m)}(f,g)+A}+A\frac{f-g}{V^{(m)}(f,g)+A}, \quad \text{for any }A>0.
\end{align*}
Now we will treat the second term on the right-hand side of last display. Note that
\begin{align*}
V^{(m)}(f,g)
=m\sum_{k\geq0}\binom{m-1}{k}(-1)^{k}\frac{v_{k+1}(1-f,1-g)}{k+1}.
\end{align*}
Since $ m\in(1,2) $ and $ v_{1}(1-f,1-g)=1 $, then
\begin{align*}
V^{(m)}(f,g)
&=m\left(1-\sum_{k\geq1}\abs{\binom{m-1}{k}}\frac{v_{k+1}(1-f,1-g)}{k+1}\right)
=m\sum_{k\geq1}\abs{\binom{m-1}{k}}\left(1-\frac{v_{k+1}(1-f,1-g)}{k+1}\right),
\end{align*}
where we note that
\begin{align*}
1-\sum_{k\geq1}\abs{\binom{m-1}{k}}=0.
\end{align*}
Since $ f,g\in(0,1) $ we also have $ 0<\frac{v_{k+1}(1-f,1-g)}{k+1}<1 $, and so let us introduce
\begin{align*}
W^{(m)}(f,g)=m\sum_{k\geq2}\abs{\binom{m-1}{k}}\left(1-\frac{v_{k+1}(1-f,1-g)}{k+1}\right)>0.
\end{align*}
In this way, we can write
\begin{align*}
V^{(m)}(f,g)
&=m(m-1)\left(1-\frac{v_2(1-f,1-g)}{2}\right)+W^{(m)}(f,g)
=m\frac{m-1}{2}(f+g)+W^{(m)}(f,g).
\end{align*}
Now back to our main problem,
\begin{align*}
A\frac{f-g}{V^{(m)}(f,g)+A}
=A\frac{2}{m(m-1)}\frac{m\frac{m-1}{2}(f+g)+W^{(m)}(f,g)+A-\left(m(m-1)g+W^{(m)}(f,g)+A\right)}{m\frac{m-1}{2}(f+g)+W^{(m)}(f,g)+A},
\end{align*}
hence,
\begin{align*}
A\frac{f-g}{V^{(m)}(f,g)+A}\leq A\frac{2}{m(m-1)}.
\end{align*}
If $ f=1 $ we can write $ 1-(g)^m=(1-g)V(1,g), $ 
while if $ f=0 $, we use instead that $ 0=	\sum_{k\geq 0}\binom{m}{k}(-1)^k. $
For either $ f\in\{0,1\} $, the rest of the proof is analogous.

To check that $ V^{(m)} $ is bounded is enough to bound from above $ v_k\leq k $ and use the estimate for the binomial coefficients from Lemma \ref{lem:bin_bound}.
\end{proof}
\begin{Cor}[$ \frac14-$H\"{o}lder continuity]\label{cor:cont}
If $ \rho^m\in L^2([0,T];\mathcal{H}^1(\mathbb{T})) $, with $ m\in(1,2) $, then {for any} $t\in[0,T] $
\begin{align*}
\abs{\rho_t(u)-\rho_t(v)}\leq
\abs{v-u}^{\frac14}\left(\frac{2}{m(m-1)}
+\norm{\partial_u(\rho_t^m)}_{L^2(\mathbb{T})}\right)\quad a.e.\;u,v\;{ \in\mathbb T.}
\end{align*}
\end{Cor}
\begin{proof}
Since $ \rho^m $ is in the target Sobolev space, we have a weak derivative of $ \rho $ and can write a.e., from the previous proposition
\begin{align*}
\abs{\rho_t(u)-\rho_t(v)}
\leq
\frac{\int_u^v\partial_w(\rho_t^m)\rmd w}{V^{(m)}(\rho_t(u),\rho_t(v))+A}+\frac{2A}{m(m-1)}
\leq
\frac{1}{A}\int_u^v\partial_w(\rho_t^m)\rmd w+\frac{2A}{m(m-1)}.
\end{align*}
We now apply Cauchy-Schwarz's inequality and set $ A=\abs{v-u}^{\frac14} $.
\end{proof}

\begin{Lemma}[Uniqueness of weak solutions]\label{lem:uniq_PME}
For $ \rho^{\rm ini}:\mathbb{T}\to[0,1] $ a measurable initial profile the weak solution of \eqref{PDE:formal}, in the sense of Definition \ref{def:weak}, is unique.
\end{Lemma}
\begin{proof}
The proof relies on the same choice of test function as in \cite[Lemma 6.3]{HJV20}, there for solutions of the FDE with $ m=-1 $. Note that for $ m\in(1,2) $ holds $ \rho^m\in L^2([0,T];\mathcal{H}^1(\mathbb{T})) $. A solution $ \rho $ of \eqref{PDE:formal} satisfies then the formulation \eqref{weak} or, equivalently,
\begin{align*}
0=\inner{\rho_t,G_t}-\inner{\rho^{\rm ini},G_0}
-\int_0^t
\inner{\rho_s,\partial_sG_s}	\rmd s
+\int_0^t
\inner{\partial_u(\rho_s)^m,\partial_u G_s}
\rmd s
\end{align*}
for any $ G\in C^{1,2}([0,T]\times\mathbb{T}) $. In particular, one can consider the alternative formulation where the regularity of $ G $ above is reduced to $ G\in L^2([0,T];\mathcal{H}^1(\mathbb{T})) $ and $ \partial_tG\in L^2([0,T];L^2[0,1]) $ (satisfying the equality on the previous display), and then show the equivalence of formulations by approximating $ G $ by a sequence of functions in $ C^{1,2}([0,T]\times\mathbb{T}) $. 
Assume that $ \rho^{(1)},\rho^{(2)} $ are two solutions starting from the same profile $ \rho^{\rm ini} $ and write $ w=\rho^{(1)}-\rho^{(2)} $.
	Then $ w $ satisfies the equality
	\begin{align*}
		\inner{w_t,G_t}
		=\int_0^t\inner{w_s(u),\partial_sG_s}\rmd s
		-\int_0^t\inner{\partial_u\left((\rho_s^{(1)})^m-(\rho_s^{(2)})^m\right),\partial_u G_s}\rmd s.
	\end{align*}
With the choice of test function
\begin{align}\label{uniq:test}
G_s(u)=\int_s^t
(\rho_r^{(1)}(u))^m-(\rho_r^{(2)}(u))^m
\rmd r,
\end{align}
we obtain
\begin{align*}
\inner{w_t,G_t}
=0
=-\int_0^t
\big\langle{
w_s,
\left((\rho_s^{(1)})^m-(\rho_s^{(2)})^m\right)
}\big\rangle
\rmd s
-\frac12 \norm{\int_0^t\partial_u
	\left(
	(\rho_r^{(1)})^m-(\rho_r^{(2)})^m
	\right)\rmd r}_2^2
.
\end{align*}
It is simple to see that
$ 
w_s(u)
\left(\rho_1^m(s,u)-\rho_2^m(s,u)\right)
\geq 0 
$ 
for a.e. $ u\in\mathbb{T} $, implying $ w=0 $ almost everywhere.
\end{proof}


\subsection{Fast diffusion}

\begin{Prop}[$ \frac12-$H\"{o}lder continuity]\label{prop:continuity_FDE}
If $ \rho\in L^2([0,T];\mathcal{H}^1(\mathbb{T})) $ then {for any} $ t\in[0,T] $ it holds that
\begin{align*}
\abs{\rho_t(u)-\rho_t(v)}\leq \abs{u-v}^\frac12\norm{\partial\rho_t}_{L^2(\mathbb{T})}\quad a.e.\;u,v { \in \mathbb{T}}
\end{align*}
\end{Prop}
\begin{proof}
{This is a simple consequence of Cauchy-Schwarz's inequality. }
\end{proof}

\begin{Lemma}[Uniqueness of weak solutions]\label{lem:uniq_FDE}  
For $ \rho^{\rm ini}:\mathbb{T}\to[0,1] $ a measurable initial profile the weak solution of \eqref{PDE:formal} in the sense of Definition \ref{def:weak} is unique.
\end{Lemma}

\begin{proof}
For $ m\in(0,1) $ our weak formulation can be shown to be equivalent to
\begin{equation}
\inner{\rho_t,G_t}-\inner{\rho^{\rm ini},G_0}
=\int_0^t
\big\{
\inner{\rho_s,\partial_sG_s}
+\inner{(\rho_s)^m,\partial_{uu} G_s}
\big\}
\rmd s, \qquad\forall  t\in(0,T],
\end{equation}
where $ G\in C^{1,2}([0,T]\times\mathbb{T})$. Recall also that we already showed, in Proposition \ref{prop:power_in_sob}, that there exists a solution $ \rho\in L^2([0,T];\mathcal{H}^1(\mathbb{T})) $. Let $ \rho^{(1)},\rho^{(2)} $ be two solutions starting from the same initial data and write $ w=\rho^{(1)}-\rho^{(2)} $. Then we have the following equation
\begin{align*}
\inner{w_t,G_t}
=\int_0^t
\big\{
\inner{w_s,\partial_sG_s}
+\inner{(\rho^{(1)}_s)^m-(\rho^{(2)}_s)^m,\partial_{uu} G_s}
\big\}
\rmd s=0.
\end{align*}
We will write $ (\rho^{(1)})^m-(\rho^{(2)})^m $ as a function of $ w $. To do so, we consider the binomial expansion of these powers. Since $ m\in(0,1) $ we have
\begin{align*}
(\rho^{(1)})^m-(\rho^{(2)})^m
=\sum_{k\geq1}\abs{\binom{m}{k}}\left((1-\rho^{(2)})^k-(1-\rho^{(1)})^k\right)
.
\end{align*}
It is important to truncate now the series at some step $ \ell $ which will be taken to infinity later on. Let $ \ell\in\mathbb{N}_+ $. Then
\begin{align*}
	\sum_{k\geq\ell+1}\abs{\binom{m}{k}}\left((1-\rho^{(2)})^k-(1-\rho^{(1)})^k\right)
	\leq
	\sum_{k\geq\ell+1}\abs{\binom{m}{k}}=\mathcal{O}\left(\ell^{-m}\right).
\end{align*}
As such, from \eqref{diff_powers}
\begin{align*}
	\sum_{k=1}^\ell\abs{\binom{m}{k}}\left((1-\rho^{(2)})^k-(1-\rho^{(1)})^k\right)
	=w\sum_{k=1}^\ell\abs{\binom{m}{k}}v_k(1-\rho^{(2)},1-\rho^{(1)})=:wV^\ell
\end{align*}
where we shorten $ V_s^\ell(u)\equiv V^\ell(\rho_s^{(1)}(u),\rho_s^{(2)}(u)) $ and $ v_k(s,u)\equiv v_k(1-\rho_s^{(1)}(u),1-\rho_s^{(2)}(u)) $. 
Note that for each $ \ell $ fixed we have the crude upper bound
\begin{align}\label{V_l:upper}
V_s^\ell(u)\leq
\sum_{k=1}^\ell\abs{\binom{m}{k}}k=\mathcal{O}\left(\ell^{1-m}\right).
\end{align}
This truncation allows us to obtain
\begin{align*}
\int_0^t
\inner{(\rho_s^1)^m-(\rho_s^2)^m,\partial_{uu} G_s}
\rmd s
\lesssim
\int_0^t
\inner{w_sV_s^\ell,\partial_{uu} G_s}
\rmd s
+
\frac{1}{\ell^m}\int_0^t
\int_{\mathbb{T}}
\abs{\partial_{uu} G_s(u)}
\rmd u \rmd s.
\end{align*}
Because for each fixed $ \ell $ we have $ V^\ell\in L^p([0,t]\times\mathbb{T}) $, for any $ 1\leq p\leq \infty $, one can approximate $ V^\ell $ by a sequence of functions in $ C^\infty([0,t];L^\infty(\mathbb{T})) $, with $ t\in[0,T] $, and with respect to the $ L^p([0,t]\times\mathbb{T}) $ norm. Let $ \varphi $ be some positive mollifier and define $ \varphi_{{\eps}}={\eps}^{-1}\varphi({\eps}^{-1}\;\cdot) $ for $ \eps>0 $. Define $  $
\begin{align*}
V_\cdot^{\ell,\eps}(u)=V_{\cdot}^\ell(u)\star\varphi_{{\eps}}.
\end{align*}
Note that $ V^{\ell,\eps}\in L^p([0,T]\times\mathbb{T}) $ for any $ 1\leq p\leq\infty $ because $ V^\ell $ is uniformly bounded in both time and space. Denote by $ \hat{f} $ the Fourier transformation of a function $ f $ defined on $ [0,t] $. From Parseval-Plancherel's identity we have the isometry
\begin{align*}
\norm{V_{\cdot}^{\ell,\eps}(u)
-
V_{\cdot}^\ell(u)}_{L^2([0,t])}
&=
\norm{
\widehat{V_{\cdot}^{\ell,\eps}(u)}
-
\widehat{V_{\cdot}^\ell(u)}
}_{L^2([0,t])}
=
\left[\int_0^t
\abs{
\widehat{V_{\cdot}^\ell(u)}(\xi)
}^2
\abs{
1-\widehat{\varphi_{{\eps}}}(\xi)
}^2
\rmd \xi\right]^\frac12.
\end{align*}
Because the mollifier is normalized and positive,
\begin{align*}
\abs{
1-\widehat{\varphi_{{\eps}}}(\xi)
}
\leq
\int_{B_{\eps}(0)}
\varphi_{\eps}(v) \abs{(1-e^{-iv\xi})} \rmd v,
\end{align*}
where $ B_\eps(0) $ is the open ball in $ \mathbb{T} $ centred in zero and with radius $ \eps>0 $. Since $ e^{-x}\geq 1-x $ we can see that
\begin{align*}
\sup_{v\in B_{\eps}(0)}\abs{(1-e^{-iv\xi})}
\leq
\sup_{v\in B_{\eps}(0)}\abs{iv\xi}
\leq
{\eps}\abs{\xi}.
\end{align*}
With this we obtain the estimate
\begin{align*}
\norm{V_{\cdot}^{\ell,\eps}(u)
-
V_{\cdot}^\ell(u)}_{L^2([0,t])}
\leq
\eps
\left[\int_{0}^t
\abs{
\widehat{V_{\cdot}^\ell(u)}(\xi)
}^2
\abs{\xi}^2
\rmd \xi\right]^\frac12
\leq
t\eps
\left[\int_{0}^t
\abs{
\widehat{V_{\cdot}^\ell(u)}(\xi)
}^2
\rmd \xi\right]^\frac12
=t\eps
\norm{
V_{\cdot}^\ell(u)
}_{L^2([0,t])}
\end{align*}
and the right-hand side of the previous display is no larger than a constant times $ t^\frac{3}{2}\eps\ell^{1-m} $. 
In particular, from Cauchy-Schwarz's inequality
\begin{align*}
\int_0^t
\inner{w_sV_s^\ell,\partial_{uu} G_s}
\rmd s
&\leq
\int_0^t
\inner{w_sV_s^{\ell,\eps},\partial_{uu} G_s}\rmd s
\\
&\quad +
\int_{\mathbb{T}}\left[\int_0^t\abs{V_s^\ell(u)-V_s^{\ell,\eps}(u)}^2\rmd s\right]^\frac12\left[\int_0^t\abs{\partial_{uu}G_s(u)}^2 \rmd s\right]^\frac12\rmd u.
\end{align*}
From the previous computations and again from the Cauchy-Schwarz's inequality, the second line in last display is bounded above by $ t^{\frac32}\eps\ell^{1-m}\norm{\partial_{uu}G}_{L^2([0,t]\times\mathbb{T})} $. 
We just showed that
\begin{align*}
\inner{w_t,G_t}
\lesssim
\int_0^t
\int_{\mathbb{T}}
w_s(u)
\big\{
\partial_sG_s(u)
+V_s^{\ell,\eps}(u)\partial_{uu} G_s(u)
\big\}
\rmd u\rmd s
+
t^{\frac12}\ell^{-m}
\left(
1
+ \eps t\ell
\right)
\norm{\partial_{uu}G}_{L^2([0,t]\times\mathbb{T})}
.
\end{align*}
We want to fix $ G $ as a solution to the backwards problem
\begin{align}\label{dual:prob}
\begin{cases}
\partial_sf+\lambda\partial_{uu}f=0, & (s,u)\in[0,t)\times\mathbb{T},\\
f(t,u)=\phi(u), &u\in\mathbb{T},
\end{cases}
\end{align}
with $ \phi $ to be chosen suitably later on. This is a well-posed problem and has a solution $ f\in C^{1,2}([0,t]\times\mathbb{T}) $ given some conditions on $ \phi $ and $ \lambda $: under the new time $ \tau=t-s $ a solution to this problem is equivalently a solution to
\begin{align*}
\begin{cases}
\partial_\tau g=\lambda\partial_{uu}g, & (\tau,u)\in(0,t]\times\mathbb{T},\\
g(0,u)=\phi(u), &u\in\mathbb{T}.
\end{cases}
\end{align*}
According to \cite[Thm.~4.5, Ch.~6, Sec.~4]{sde:friedman}, for $ \lambda $ positive and bounded uniformly in $ [0,t]\times\mathbb{T} $, continuous with respect to time (uniformly in $ \mathbb{T} $) and $ \alpha-$H\"{o}lder continuous with respect to the space variable; and $ \phi $ a continuous function, there exists a solution to this Cauchy problem in $ C^{1,2}([0,t]\times\mathbb{T}) $. Note that we have already checked that $ V^{\ell,\eps} $ satisfies all the requirements for $ \lambda $ above (for $ \ell $ fixed) except the H\"{o}lder continuity condition. Noting that $ \rho^{(1)},\rho^{(2)} $ is $ \frac12-$H\"{o}lder so is $ V_{\eps} $. To see this we sum and subtract appropriate terms and use the triangle inequality to estimate
\begin{align*}
\abs{
v_k(s,x)-v_k(s,y)
}
&\leq
\abs{
\rho_s^{(1)}(y)-\rho_s^{(1)}(x)
}
\sum_{i=0}^{k-1}
v_i(1-\rho_s^{(1)}(x),1-\rho_s^{(1)}(y))
(1-\rho_s^{(2)}(x))^{k-1-i}
\\
&\quad +
\abs{
\rho_s^{(2)}(y)-\rho_s^{(2)}(x)
}
\sum_{i=0}^{k-1}
v_{k-1-i}(1-\rho_s^{(2)}(x),1-\rho_s^{(2)}(y))(1-\rho_s^{(1)}(x))^i
\lesssim
k^2\abs{x-y}^\frac12.
\end{align*}
In this way,
\begin{align*}
\abs{
(v_k(\cdot,x)-v_k(\cdot,y))\star\varphi_{\eps}(s)
}
=\int_{0}^t\varphi_{\eps}(s-r)(v_k(r,x)-v_k(r,y))\rmd r
\leq
k^2\abs{x-y}^{\frac12}\int_{0}^t\varphi_{\eps}(s-r)\rmd r.
\end{align*}
Recalling that the integral on the right-hand side equals one, we see that
\begin{align*}
\abs{
V_s^{\ell,\eps}(x)-V_s^{\ell,\eps}(y)
}
\leq
\sum_{k=1}^\ell\abs{\binom{m}{k}}
\abs{
v_k(\cdot,x)-v_k(\cdot,y)
\star\varphi_{\eps}(s)
}
\lesssim
\abs{x-y}^\frac12
\ell^{2-m}.
\end{align*}
%
In this way, fixing our test function as $ G=f $ with $ \lambda=V^{\ell,\eps} $ we see that
\begin{align*}
\inner{w_t,\phi}
\lesssim
t^{\frac12}\ell^{-m}(1+\eps t\ell)
\norm{\partial_{uu}G}_{L^2([0,t]\times\mathbb{T})}
\end{align*}
and we need to estimate the integral on the right-hand side above. 

Let us multiply both sides of \eqref{dual:prob} by $ \partial_{uu}G $ and integrate once in space and time, obtaining
\begin{align*}
0
=\int_0^t\int_{\mathbb{T}}\partial_sG\partial_{uu}G\rmd u\rmd s
+\int_0^t\int_{\mathbb{T}}V^{\ell,\eps}\abs{\partial_{uu}G}^2\rmd u\rmd s.
\end{align*}
An integration by parts on the first integral on the right-hand side above yields
\begin{align*}
-
\int_0^t\int_{\mathbb{T}}\partial_{u}(\partial_sG)\partial_{u}G\rmd u\rmd s
=&-\frac12
\int_0^t\int_{\mathbb{T}}
\partial_s\left(\partial_uG\right)^2
\rmd u\rmd s
\\
=&
-\frac12\int_{\mathbb{T}}
\left\{
\left(\partial_uG_t(u)\right)^2
-\left(\partial_uG_0(u)\right)^2
\right\}
\rmd u
.
\end{align*}
Using the terminal condition and bounding from below $ (\partial_uG_0(u))^2\geq 0 $ and $ V_\eps^\ell>m $ we conclude that
\begin{align*}
\int_0^t\int_{\mathbb{T}}\abs{\partial_{uu}G}^2\rmd u\rmd s
\leq
\frac{1}{2m}\norm{\phi'}_{L^2(\mathbb{T})}^2.
\end{align*}
With this, and fixing $ \eps=1/\ell $ we obtain
\begin{align}\label{almost}
\inner{w_t,\phi}
\lesssim
t^{\frac12}\ell^{-m}
(
1+t
)
\norm{\phi'}_{L^2(\mathbb{T})}.
\end{align}
Denoting by $ w^\pm $ the positive/negative part of $ w $, we want to fix $ \phi(\cdot)=\mathbf{1}_{\{u\in\mathbb{T}:\; w_t(u)\geq0\}}(t,\cdot) $, obtaining that $ \rho^{(1)}\leq \rho^{(2)} $ a.e., and analogously take $ \phi(\cdot)=\mathbf{1}_{\{u\in\mathbb{T}:\; w_t(u)\leq0\}}(t,\cdot) $, obtaining instead $ \rho^{(1)}\geq \rho^{(2)} $ and leading to $ \rho^{(1)}=\rho^{(2)} $ a.e. To do so we need to consider in \eqref{almost} a sequence $ (\phi_n)_n\subset C(\mathbb{T}) $ converging to $ \phi $ at least in $ L^2 $ and such that $ \norm{\phi_n'}_{L^2(\mathbb{T})}<\infty $ for all $ n>0 $. Regarding the convergence, since $ \phi\in L^2(\mathbb{T}) $ and $ C(\mathbb{T}) $ is dense in $ L^p(\mathbb{T}) $ for all $ 1\leq p <\infty $, there is a sequence of continuous functions $ (\phi_n)_n $ approximating $ \phi $ in $ L^2(\mathbb{T}) $. This sequence of continuous functions can be approximated (in $ L^2 $) by a sequence of smooth functions $ (\phi_{n,k})_k $ \emph{via} mollification. We fix one of these smooth representatives as the terminal condition on the problem \eqref{dual:prob}. Taking the limit $ \ell\to+\infty $ in \eqref{almost} and then the limits on $ n $ and $ k $, and recalling that $ t\in[0,T] $ is arbitrary concludes the proof.
\end{proof}
\section*{Acknowledgements}{ P.G.  thanks Professor  Tadahisa Funaki for posing the question of deriving the porous medium equation with a non-integer power from interacting particle systems, which motivated the study developed in this work. 
This project was partially supported by the ANR grant MICMOV (ANR-19-CE40-0012)
of the French National Research Agency (ANR), and by the European Union with the program FEDER ``Fonds europ\'een de d\'eveloppement r\'egional'' with the R\'egion Hauts-de-France  and from
Labex CEMPI (ANR-11-LABX-0007-01). It has also received funding from the European
Research Council (ERC) under the European Union’s Horizon 2020 research and innovative program
(grant agreement n. 715734). 
G.N. gratefully acknowledges the support from FCT/Portugal through the Lisbon Mathematics PhD (Lismath), grant PD/BD/150345/2019.  P.G. and G.N. thank FCT/Portugal for financial support through CAMGSD, IST-ID, projects UIDB/04459/2020 and UIDP/04459/2020. G.N. thanks FCT for the financial support provided by the Portugal/France agreement Programa PESSOA cotutelas (FCT/Campus France).
}

\end{document}